\documentclass[12pt]{article}
\usepackage[kerning, tracking, spacing]{microtype}
\microtypecontext{spacing=nonfrench}

\usepackage{amsmath, amsthm}
\usepackage{amssymb}
\usepackage{enumitem}
\usepackage{graphicx}
\usepackage{mathdots}
\usepackage{color}
\usepackage{hyperref,url}
\usepackage{tikz, tikz-cd}
\usetikzlibrary{matrix}
\usetikzlibrary{shapes}
\usetikzlibrary{arrows,decorations.markings, tikzmark}
\usepackage{mathtools}
\usepackage{ stmaryrd }
\usepackage[normalem]{ulem}
\usepackage[utf8]{inputenc}

 \usepackage{xcolor}
\hypersetup{
    colorlinks,
    linkcolor={red!50!black},
    citecolor={blue!50!black},
    urlcolor={blue!80!black}
}

\pgfarrowsdeclare{bad to}{bad to}
{
  \pgfarrowsleftextend{-2\pgflinewidth}
  \pgfarrowsrightextend{\pgflinewidth}
}
{
  \pgfsetlinewidth{0.8\pgflinewidth}
  \pgfsetdash{}{0pt}
  \pgfsetroundcap
  \pgfsetroundjoin
  \pgfpathmoveto{\pgfpoint{-3\pgflinewidth}{4\pgflinewidth}}
  \pgfpathcurveto
  {\pgfpoint{-2.75\pgflinewidth}{2.5\pgflinewidth}}
  {\pgfpoint{0pt}{0.25\pgflinewidth}}
  {\pgfpoint{0.75\pgflinewidth}{0pt}}
  \pgfpathcurveto
  {\pgfpoint{0pt}{-0.25\pgflinewidth}}
  {\pgfpoint{-2.75\pgflinewidth}{-2.5\pgflinewidth}}
  {\pgfpoint{-3\pgflinewidth}{-4\pgflinewidth}}
  \pgfusepathqstroke
}

\DeclareMathAlphabet{\mymathbb}{U}{BOONDOX-ds}{m}{n}

\def\Mbar{\overline{\mathcal{M}}}
\def\M{\mathfrak{M}}

\def\C{\mathfrak{C}}
\def\CH{\mathrm{CH}}
\def\CHOP{\mathrm{CH}_{\mathrm{OP}}}

\def\R{\mathrm{R}}
\def\st{\mathrm{st}}
\def\aut{\mathrm{Aut}}
\def\zero{\mymathbb{0}}
\def\one{\mymathbb{1}}

\def\A{\mathcal{A}}

\def\id{\mathrm{id}}
\def\AA{\mathbb{A}}

\def\PP{\mathbb{P}}
\def\QQ{\mathbb{Q}}

\def\TT{\mathcal{T}}

\def\Hom{\mathrm{Hom}}

\def\PGL{\mathrm{PGL}}
\def\Rat{\mathrm{Rat}}

\def\mcE{\mathcal{E}}
\def\mcO{\mathcal{O}}
\def\mfA{\mathfrak{A}}
\def\mfB{\mathfrak{B}}
\def\rk{\mathrm{rk}\,}

\def\codim{\mathrm{codim}}
\def\Hom{\mathrm{Hom}}
\newcommand{\rar}[1]       {\stackrel{#1}{\longrightarrow}}
\newcommand{\isomarrow}         {\rar{\sim}}
\newcommand{\isomto}            {\stackrel{\sim}{\to}}
\newcommand{\isomfrom}          {\stackrel{\sim}{\gets}}
\usepackage{amsthm}
\theoremstyle{definition}
\newtheorem{definition}{Definition}[section]
\newtheorem{theorem}[definition]{Theorem}
\newtheorem{example}[definition]{Example}
\newtheorem{proposition}[definition]{Proposition}
\newtheorem{corollary}[definition]{Corollary}
\newtheorem{lemma}[definition]{Lemma}

\newtheorem{remark}[definition]{Remark}

\newtheorem*{conjecture*}{Conjecture}
\newtheorem*{question*}{Question}
\newtheorem*{theorem*}{Theorem}

\title{Chow rings of stacks of prestable curves I}
\author{Younghan Bae\footnote{Department of Mathematics, ETH Zurich} , Johannes Schmitt\footnote{Institute for Mathematics, University of Zurich},\\with an appendix joint with Jonathan Skowera}
\date{May 2022}

\begin{document}
\maketitle
\begin{abstract}
We study the Chow ring of the moduli stack $\M_{g,n}$ of prestable curves and define the notion of tautological classes on this stack. We extend formulas for intersection products and functoriality of tautological classes under natural morphisms from the case of the tautological ring of the moduli space $\Mbar_{g,n}$ of stable curves. This paper provides foundations for the paper \cite{BaeSchmitt2}.

In the appendix (joint with J. Skowera), we develop the theory of a proper, but not necessary projective, pushforward of algebraic cycles. The proper pushforward is necessary for the construction of the tautological rings of $\M_{g,n}$ and is important in its own right. We also develop operational Chow groups for algebraic stacks.
    \\~\\
    \noindent 2020 Mathematics Subject Classification:\\
    \noindent 14H10, 14C15 (Primary) 14C17 (Secondary)
\end{abstract}

\tableofcontents
\section{Introduction}
Let $\Mbar_{g,n}$ be the moduli space of stable curves. It parameterizes tuples $(C,p_1, \ldots, p_n)$ of a nodal curve $C$ of arithmetic genus $g$ with $n$ distinct smooth marked points such that $C$ has only finitely many automorphisms fixing the points $p_i$. After Mumford's seminal paper \cite{mumford83}, there has been a substantial study of the structure of the tautological rings \[\R^{*}(\Mbar_{g,n}) \subseteq \CH^{*}(\Mbar_{g,n})_\mathbb{Q}\,.\]
The tautological rings form a system of subrings of $\CH^{*}(\Mbar_{g,n})_\mathbb{Q}$ with explicit generators defined using the universal curve and the boundary gluing maps of the spaces $\Mbar_{g,n}$, see \cite{graberpandh}.

A natural extension of $\Mbar_{g,n}$ is the moduli stack $\M_{g,n}$ of marked prestable curves, in which we drop the condition of having only finitely many automorphisms. It is a smooth algebraic stack, locally of finite type over the base field $k$ and containing $\Mbar_{g,n}$ as an open substack. However, by allowing infinite automorphism groups, the stacks of prestable curves are no longer  Deligne-Mumford stacks and not of finite type\footnote{In fact, the stack $\M_{0,0}$ contains a finite type open substack which is not even a quotient stack, see \cite[Proposition 5.2]{kreschflattening}.}. 

A recent application of Chow groups of such non-finite type algebraic stacks appeared in the paper \cite{BHPSS}, which studied cycle classes and tautological rings for the universal Picard stack $\mathfrak{Pic}_{g}$ over the stack $\M_{g}$. The stack $\mathfrak{Pic}_{g}$ parameterizes pairs $(C,\mathcal{L})$ of a prestable curve $C$ and a line bundle $\mathcal{L}$ on $C$. In \cite{BHPSS}, results from \cite{DRtarget} are used to prove a formula for the fundamental class of the closure of the zero section $\{(C, \mathcal{O}_C)\} \subseteq \mathfrak{Pic}_{g}$. By pulling back this equality under natural morphisms $\Mbar_{g,n} \to \mathfrak{Pic}_{g}$, new results about the classical double ramification cycles on the moduli of stable curves are established. 
%As an algebraic stack, $\M_{g,n}$ even has a finite type open substack which does not have a representation as a quotient stack, see \cite{kreschversal}. 

In the paper \cite{BHPSS}, the intersection theory of $\mathfrak{Pic}_{g}$ is studied using a definition of operational Chow groups modelled on \cite[Chapter 17]{Fulton1984Intersection-th}.
In our paper, we follow the approach \cite{kreschartin} by Kresch, who developed a cycle theory for algebraic stacks of finite type over a field $k$. This theory has many structural advantages over the operational theory of \cite{BHPSS}, such as projective pushforwards and an excision sequence, and for a smooth stack always admits a natural map to this operational theory.\footnote{See Appendix \ref{Sect:OperationalChow} for more details on the comparison of the definitions.}

We extend Kresch's theory from the case of finite-type stacks to the case of algebraic stacks $\mathfrak{X}$ locally of finite type over $k$ (such as $\M_{g,n}$) by defining their Chow groups\footnote{Unless stated otherwise, all Chow groups in the paper will be taken with $\mathbb{Q}$-coefficients.} as the limit 
\[\CH_*(\mathfrak{X}) = \varprojlim_{i \in I}  \CH_*(\mathcal U_i),\]
for $(\mathcal U_i)_{i \in I}$ a directed system of finite-type open substacks covering $\mathfrak{X}$.
Using this definition, we define the tautological ring $\R^*(\M_{g,n}) \subseteq \CH^*(\M_{g,n})$, extending the definition \cite{graberpandh} for the moduli spaces of stable curves.\footnote{There is a small caveat: the intersection theory of $\CH^*(\M_{1,0})$ is not covered by \cite{kreschartin} because the stabilizer group at the general point is not a linear algebraic group. In this paper, we exclude this case.}

\subsection*{Proper pushforwards of Chow groups of algebraic stacks}
When extending the definition of the tautological ring to the stacks of prestable curves, we immediately encounter a problem: for the spaces $\Mbar_{g,n}$ of stable curves, these rings can be defined as the smallest system of subrings of $\CH^*(\Mbar_{g,n})$ closed under pushforwards by gluing morphisms and the morphisms $\overline{\mathcal{C}}_{g,n} = \Mbar_{g,n+1} \to \Mbar_{g,n}$ giving the universal curve over $\Mbar_{g,n}$. However, for the stacks $\M_{g,n}$ of prestable curves, the analogous universal curve morphisms are in general proper, but not projective (see \cite[Example 2.3]{fulghesurational}). 
%\jocomment{Andrew remarked too that he things the gluing morphisms are projective in fact (since they are proper and quasi-finite they are actually finite). So I propose to just remove the mentioning of gluing not being projective here.}
Thus Kresch's Chow theory, which a priori only has projective pushforwards, cannot be applied immediately. Historically, this has been a major obstruction in the study of the Chow groups $\CH^*(\M_{g,n})$ and made it necessary to give many ad-hoc constructions of classes that are traditionally defined by proper pushforwards (see \cite{fulghesutaut, fulghesu3nodes}).

To overcome this obstacle, joint with Skowera, we define proper pushforwards for cycle groups of algebraic stacks. The corresponding {results are} included as Appendix \ref{sec:Properpushforward} to our paper. We state here the main properties of this construction.

\begin{theorem}[see Theorem \ref{thm:properpushforward} and Proposition \ref{prop:properCompatibility}]
Let $Y$ be a stack stratified by global quotient stacks, and let $f : X \to Y$ be a proper, representable morphism. Then there is a proper pushforward $f_* : \CH_d(X,\mathbb{Z}) \to \CH_d(Y,\mathbb{Z})$ for all $d$ which is functorial (with respect to compositions) %\jocomment{mention in main text?} \Ycomment{I think it is a good idea} 
and compatible with flat pullbacks and refined Gysin pullbacks.

If, instead, $f$ is proper and of relative Deligne-Mumford type, then there is a proper pushforward $f_* : \CH_d(X, \QQ) \to \CH_d(Y, \QQ)$ for all $d$, with the properties above.
\end{theorem}
% To describe the main idea of the construction, recall that a cycle in $\CH_*(X)$ is given by a tuple $(Z\to X, E \to Z, \alpha)$, where $Z \to X$ is projective, $E \to Z$ is a vector bundle and $\alpha$ is a naive Chow class on the bundle $E$. 
% \jocomment{Continue sketch of main idea? Or refer to place where it is sketched?}

\subsection*{The universal curve over the stack of prestable curves}
A second problem that we encounter when generalizing the definition of the tautological ring of $\Mbar_{g,n}$ to $\M_{g,n}$ is that the universal curve $\C_{g,n} \to \M_{g,n}$ is \emph{not} given by the forgetful map $\M_{g,n+1} \to \M_{g,n}$. In particular, since the forgetful maps are in general not proper, we cannot define $(\R^*(\M_{g,n}))_{g,n}$ as the smallest system of subrings of $(\CH^*(\M_{g,n}))_{g,n}$ closed under gluing and forgetful pushforwards.

To overcome this issue (and give a modular interpretation of $\C_{g,n}$ as a stack of $(n+1)$-pointed curves), we use the notion of prestable curves \emph{with values in a semigroup $\A$} from \cite{behrendmanin, costello}. Given a suitable (commutative) semigroup\footnote{See Section \ref{Sect:Avaluedprestable} for the precise definition and technical conditions we require for these semigroups.} $\A$ and an element $a \in \A$, these references define a stack $\M_{g,n,a}$ parameterizing tuples $(C,p_1, \ldots, p_n, (a_{C_v})_v)$ of a prestable curve $(C,p_1, \ldots, p_n)$ together with a value $a_{C_v} \in \A$ for each component $C_v$ of $C$ such that all $a_{C_v}$ sum up to $a$ in $\A$. Moreover, in contrast to the stack $\M_{g,n}$, the definition of $\M_{g,n,a}$ includes a stability condition: any component $C_v$ such that $a_{C_v}=\zero \in \A$ is the neutral element of $\A$ must actually be stable, i.e. have a finite group of automorphisms fixing all markings and nodes on $C_v$. The advantage of this stability condition is that the natural forgetful map $\pi: \M_{g,n+1,a} \to \M_{g,n,a}$, which forgets the last marking and contracts the component containing it if it becomes unstable, does define the universal curve over $\M_{g,n,a}$.

Applying this machinery to a particularly simple semigroup, we obtain the desired modular interpretation of $\C_{g,n}$. For this consider the semigroup
\[\A=\{\zero,\one\}\text{ with }\zero + \zero = \zero, \zero + \one = \one + \zero = \one + \one = \one.\]
Then we show the following.
\begin{proposition}[see Proposition \ref{Prop:Mgnsemigroup}, Corollary \ref{Cor:modinterpretunivcurve}] \label{Pro:modinterpretintro}
Let $g,n \geq 0$ and consider the semigroup $\A=\{\zero,\one\}$ above. Then the stack $\M_{g,n}$ is naturally contained inside $\M_{g,n,\one}$ as the open substack of $(C,p_1, \ldots, p_n, (a_{C_v})_v)$ such that $a_{C_v}=\one$ for all $v$. Thus the universal curve $\C_{g,n}$ is naturally contained as an open substack of $\M_{g,n+1,\one}$ sitting in the cartesian diagram
\begin{equation*}
\begin{tikzcd}[column sep=-4pt]
\C_{g,n} \arrow[d] & \subseteq & \M_{g,n+1,\one} \arrow[d,"\pi"]\\
\M_{g,n} & \subseteq & \M_{g,n,\one}
\end{tikzcd}
\end{equation*}
\end{proposition}
In particular, this proposition indeed gives an interpretation of $\C_{g,n}$ as a stack of $(n+1)$-pointed prestable curves together with some additional structure (see the paragraph below Corollary \ref{Cor:modinterpretunivcurve} for more details).

\subsection*{Tautological rings of stacks of prestable curves}
Having solved both the issues with proper pushforwards and the modular interpretation of the universal curve, we are now ready to define the tautological rings. Since the discussion in the last section shows that the spaces $\M_{g,n,a}$ appear naturally, we will in fact define the tautological rings for these spaces and obtain the rings for $\M_{g,n}$ by restriction. To write down the definition, we note that in addition to the forgetful maps 
\begin{equation} \label{eqn:Avalforget}
    \pi : \M_{g,n+1,a} \to \M_{g,n,a}
\end{equation}
mentioned above, there also exist gluing maps
\begin{equation} \label{eqn:Avalgluing}
    \xi_\Gamma : \M_\Gamma = \prod_{v \in V(\Gamma)} \M_{g(v),n(v),a(v)} \to \M_{g,n,a}
\end{equation}
for every prestable graph $\Gamma$ together with an $\A$-valuation $a: V(\Gamma) \to \A$ satisfying $\sum_{v \in V(\Gamma)} a(v) = a$. {Here $V(\Gamma)$ is the set of vertices of the graph $\Gamma$.}
\begin{definition} \label{Def:tautringintro}
The tautological rings $(\R^*(\M_{g,n,a}))_{g,n,a}$ are defined as the smallest system of $\QQ$-subalgebras with unit of the Chow rings $(\CH^*(\M_{g,n,a}))_{g,n,a}$ closed under taking pushforwards by the natural forgetful and gluing maps \eqref{eqn:Avalforget} and \eqref{eqn:Avalgluing}.

The tautological ring $\R^*(\M_{g,n}) \subseteq \CH^*(\M_{g,n})$ is defined as the image of the restriction of $\R^*(\M_{g,n,\one})$ to the open substack $\M_{g,n} \subseteq \M_{g,n,\one}$ from Proposition \ref{Pro:modinterpretintro}.
\end{definition}
Just as for the moduli spaces of stable curves, we define $\psi$ and $\kappa$-classes: given $1 \leq i \leq n$ we set
\[
\psi_i = c_1(\sigma_i^* \omega_\pi) \in \CH^1(\M_{g,n,a})\,,
\]
where $\sigma_i : \M_{g,n,a} \to \M_{g,n+1,a}$ is the $i$-th universal section and $\omega_\pi$ is the relative dualizing sheaf of $\pi$. Similarly, given $m \geq 0$ we set
\[
\kappa_m = \pi_*\left(\psi_{n+1}^{m+1} \right) \in \CH^m(\M_{g,n,a})\,.
\]
It is easy to see that both types of classes are in fact tautological.
Given any $\A$-valued prestable graph $\Gamma$, consider the products
\begin{equation} %\label{eqn:alphadecoration}
\alpha = \prod_{v\in V} \left( \prod_{i \in H(v)} \psi_{v,i}^{a_i} \prod_{a=1}^{m_v} \kappa_{v,a}^{b_{v,a}} \right)\in \CH^*(\M_\Gamma).\end{equation}
of $\psi$ and $\kappa$-classes on the space $\M_\Gamma$ above. Then we define the \emph{decorated stratum class} $[\Gamma,\alpha]$ as the pushforward
\[[\Gamma,\alpha] = (\xi_\Gamma)_* \alpha \in \R^*(\M_{g,n,a}).\]
The following result (generalizing \cite[Proposition 11]{graberpandh}) show that these classes additively generate the tautological rings.
\begin{theorem} \label{Thm:tautringintro}%[see Corollary \ref{Cor:tautproduct}]
The tautological ring $\R^*(\M_{g,n,a})$ is generated as a $\QQ$-vector space by the decorated strata classes $[\Gamma, \alpha]$. In addition to being closed under pushforwards by gluing and forgetful maps, the tautological rings are likewise closed under pullbacks by these maps, with explicit formulas describing all these operations on the generators $[\Gamma,\alpha]$.\footnote{For the precise statement and formulas from the theorem, we refer the reader to Section \ref{calculus}.}
\end{theorem}
This result gives an effective method to perform computations in the Chow rings of the stacks $\M_{g,n,a}$. Moreover, it shows that while both the {Chow} and the tautological rings of these stacks are in general infinite-dimensional, the individual graded pieces of $\R^*(\M_{g,n,a})$ always have a finite set of generators.

\subsection*{Relations to other work}
In this section we explain how our results relate to previous results on the intersection theory of the stacks $\M_{g,n}$.

As a first example, in \cite{gathmann} Gathmann used the pullback formula of $\psi$-classes along the stabilization morphism $\st\colon \M_{g,1}\to \Mbar_{g,1}$ to prove certain properties of the Gromov-Witten potential. In Section \ref{calculus} we compute arbitrary pullbacks of tautological classes under the stabilization map, in particular recovering Gathmann's result.

In \cite{oesinghaus}, Oesinghaus computed the Chow rings of the open locus $\mathcal{T} \subset \M_{0,3}$ of curves with dual graph of the shape
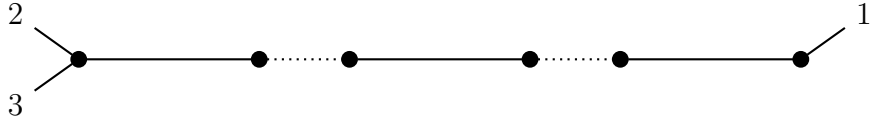
\begin{equation*}
\begin{tikzpicture}[scale=1.2, baseline=-3pt,label distance=0.3cm,thick,
    virtnode/.style={circle,draw,scale=0.5}, 
    nonvirt node/.style={circle,draw,fill=black,scale=0.5} ]
    \node at (-6,0) [nonvirt node] (F) {};
    \node at (-4,0) [nonvirt node] (E) {};
    \node at (-3,0) [nonvirt node] (D) {};
    \node at (-1,0) [nonvirt node] (C) {};
    \node [nonvirt node] (A) {};
    \node at (2,0) [nonvirt node] (B) {};
    \draw [-] (A) to (B);
    \draw [-, dotted] (C) to (A);
    \draw [-] (D) to (C);
    \draw [-, dotted] (E) to (D);
    \draw [-] (F) to (E);
    \node at (2.7,.5) (m1) {$1$};
    \draw [-] (B) to (m1);
    \node at (-6.7,.5) (n1) {$2$};
    \draw [-] (F) to (n1);
    \node at (-6.7,-.5) (n2) {$3$};
    \draw [-] (F) to (n2);    
    \end{tikzpicture}
\end{equation*}
This stack has a natural interpretation as the stack of \emph{expanded pairs} appearing in \cite{expandedpairs}. Oesinghaus showed that the Chow ring of $\mathcal{T}$ is given by the known algebra of \emph{quasi-symmetric functions} $\mathrm{QSym}$ (see \cite{quasisymmetric} for an overview). The ring $\mathrm{QSym}$ has a natural basis $M_J$ (as a $\QQ$-vector space) indexed by positive integer vectors $J = (j_1, \ldots, j_k) \in \mathbb{Z}_{\geq 1}^{k}$ of some length $k \geq 0$, and the product $M_J \cdot M_{J'}$ can be defined in terms of a certain \emph{shuffle rule} on the vectors $J, J'$ (see \cite[Proposition 2]{oesinghaus}).

Oesinghaus' proof worked by writing down an open exhaustion of $\mathcal{T}$ by quotient stacks, allowing to write the Chow ring as a certain projective limit of polynomial rings which is known to produce the algebra $\mathrm{QSym}$. However, due to the nature of this proof, a geometric interpretation for the generators $M_J$ was not immediately clear (see \cite[Remark 7]{oesinghaus}). Using the techniques of our paper, we can now answer this question, showing that the generators $M_J$ have a concrete interpretation as tautological classes.
\begin{proposition}[see Example \ref{Exa:Oesinghaus}]
For $J = (j_1, \ldots, j_k) \in \mathbb{Z}_{\geq 1}^{k}$, the generator $M_J \in \mathrm{QSym} \cong \CH^*(\mathcal{T})$ is given by the restriction of the tautological class
\begin{equation} \label{eqn:Oesinggenintro}
\begin{tikzpicture}[scale=1.2, baseline=-3pt,label distance=0.3cm,thick,
    virtnode/.style={circle,draw,scale=0.5}, 
    nonvirt node/.style={circle,draw,fill=black,scale=0.5} ]
    \node at (-6,0) [nonvirt node] (F) {};
    \node at (-4,0) [nonvirt node] (E) {};
    \node at (-3,0) [nonvirt node] (D) {};
    \node at (-1,0) [nonvirt node] (C) {};
    \node [nonvirt node] (A) {};
    \node at (2,0) [nonvirt node] (B) {};
    \draw [-] (A) to (B);
    \draw [-, dotted] (C) to (A);
    \draw [-] (D) to (C);
    \draw [-, dotted] (E) to (D);
    \draw [-] (F) to (E);
    \node at (-5,.3) {$(-\psi - \psi')^{j_1-1}$};
    %\node at (-4.3,.3) {$\psi_{h'_1}^{a'_1}$};
    \node at (-2,.3) {$(-\psi - \psi')^{j_\ell-1}$};
    %\node at (-1.3,.2) {$\scriptstyle h'_{2i}$};
    \node at (1,.3) {$(-\psi - \psi')^{j_k-1}$};
    %\node at (1.6,.3) {$\psi_{h'_{2l-1}}^{a'_{2l-1}}$};
    %\node at (-.7,.5) (n1) {$i$};
    %\draw [-] (A) to (n1);
    %\node at (0,.7) (n2) {$\downarrow$};
    \node at (2.7,.5) (m1) {$1$};
    \draw [-] (B) to (m1);
    \node at (-6.7,.5) (n1) {$2$};
    \draw [-] (F) to (n1);
    \node at (-6.7,-.5) (n2) {$3$};
    \draw [-] (F) to (n2);    
    \end{tikzpicture}
\end{equation}
on $\M_{0,3}$.
\end{proposition}
Furthermore, it is straightforward to see that the shuffle rule describing products $M_J \cdot M_{J'}$ is an immediate consequence of the product formula for the tautological classes \eqref{eqn:Oesinggenintro}.
Oesinghaus also computes the Chow rings of the loci $\M_{0,2}^\textup{ss}$ and $\M_{0,3}^\textup{ss}$ of semistable curves in $\M_{0,2}$ and $\M_{0,3}$, giving a description in terms of tensor products involving the rings $\mathrm{QSym}$. Again we give a description in terms of tautological classes in Example \ref{Exa:Oesinghaus}.

\subsection*{Tautological relations in genus zero}
The present paper lays down the foundations of the theory of the Chow rings $\CH^*(\M_{g,n})$. In the second part \cite{BaeSchmitt2} we use results of this paper to fully determine the Chow rings of $\M_{0,n}$ for all $n$: we prove that all classes are tautological and we compute all relations among generators of the tautological ring.
\subsection*{Structure of the paper}
In Section \ref{sec:calculus} we establish basic properties of the stacks $\M_{g,n}$. We discuss boundary gluing maps in Section \ref{Sect:gluingmaps}, introduce the stacks of prestable curves with values in a semigroup in Section \ref{Sect:Avaluedprestable}. In Section \ref{Sect:ChowandTaut} we establish basic properties of the Chow group of $\M_{g,n}$. In Section \ref{sec:definitions} we define Chow groups and tautological rings of such stacks. In Section \ref{calculus} we compute formulas for intersection products and pullbacks and pushforwards of tautological classes under natural maps. In Section  \ref{sec:previous} we compare our result with previous works by Gathmann \cite{gathmann}, Pixton \cite{Pixtonboundary} and Oesinghaus \cite{oesinghaus}.

In Appendix \ref{Sect:Chowlocfintype} we give some general treatment of Chow groups of locally finite type algebraic stacks. We give a definition of such Chow groups based on \cite{kreschartin} and show various basic properties. In Appendix \ref{sec:Properpushforward} (joint with J. Skowera) we construct  proper pushforwards of cycles, show basic compatibility properties of these pushforwards and explain how they extend to the setting of algebraic stacks locally of finite type. Finally, in Appendix \ref{Sect:OperationalChow} we give a definition and establish basic properties of operational Chow groups on locally finite type stacks, a technical tool needed for some of the computations in Section \ref{calculus}.
\subsection*{Acknowledgements}
We are indebted to Andrew Kresch for many stimulating conversations and considerable help and input concerning Chow groups of algebraic stacks.  We are also grateful to Jakob Oesinghaus, Rahul Pandharipande and Rachel Webb for many interesting conversations.
Finally, we want to thank the anonymous referee for an exceptionally thorough reading of the first version of this paper and numerous comments and suggestions that have helped to both improve the presentation as well as to clarify several technical issues. 
\section{The stack \texorpdfstring{$\M_{g,n}$}{Mgn} of prestable curves}\label{sec:calculus}
Throughout the paper we work over an arbitrary base field $k$.
Let $\M_{g,n}$ be the moduli stack of prestable curves of genus $g$ with $n$ marked points.  An object of $\M_{g,n}$ over a scheme $S$ is a tuple
\begin{equation*}
    (\pi:C\to S, \,\,\, p_1,\ldots,p_n: S\to C)\,,
\end{equation*}
where $C$ is an algebraic space, the map $\pi$ is a flat, proper morphism of finite presentation and relative dimension $1$. The geometric fibers of $\pi$ are connected, reduced curves of arithmetic genus $g$ with at worst nodal singularities. The morphisms $p_1, \ldots, p_n$ are disjoint sections of $\pi$ with image in the smooth locus of $\pi$, see \cite[0E6S]{stacks-project}. %\jocomment{I rearranged the definition (but I think I kept all the parts). We should verify that this agrees with the stacks project. Isn't 0E6S the more appropriate tag here?}\Ycomment{Yes, I changed to 0E6S. In order to make it precise, we have to say what is the morphism between objects, etc, but I guess this level of accuracy is enough?}

This stack is quasi-separated, smooth and locally of finite type over $k$ (\cite{kreschversal}) and of dimension $3g-3+n$ (\cite{behrendGW}). 
For $2g-2+n>0$ there is a natural stabilization morphism
\begin{equation*}
    \st: \M_{g,n} \to \Mbar_{g,n}
\end{equation*}
which contracts unstable rational components. This morphism is flat by \cite[Proposition 3]{behrendGW}. 

\subsection{Boundary gluing maps} \label{Sect:gluingmaps}
%\Ycomment{Referee wants to modify the exposition of the prestable graph. I think this definition is standard in the literature [JPPZ, PPZ, ...]. Maybe should leave it as it is?}
%\jocomment{I would indeed leave the L and E as they are; I added the formula for the first Betti number.}
A prestable graph $\Gamma$ of genus $g$ with $n $ markings consists of the data
$$\Gamma=(V,\, H,\, \ell: L \to \{1, \ldots, n\} , \ g \colon V \to \mathbb{Z}_{\geq 0}\, ,
\ v\colon H \to V\, , 
\ \iota : H\to H)$$
satisfying the properties:
\begin{enumerate}
\item[(i)] $V$ is a vertex set with a genus function $g\colon V\to \mathbb{Z}_{\geq 0}$,
\item[(ii)] $H$ is a half-edge set equipped with a vertex assignment $v:H \to V$ and an involution $\iota$,
\item[(iii)] $E$, the edge set, is defined by the 2-cycles of $\iota$ in $H$ (self-edges at vertices are permitted),
\item[(iv)] $L \subseteq H$, the set of legs, is defined by the fixed points of $\iota$ and corresponds to $n$ markings via the bijection $\ell : L \to \{1, \ldots, n\}$,
\item[(v)] the pair $(V,E)$ defines a  connected graph satisfying the genus condition 
$$\sum_{v \in V} g(v) + h^1(\Gamma) = g\,,$$
where $h^1(\Gamma) = |E| - |V| + 1$ is the first Betti number of the graph $\Gamma$.
\end{enumerate}
%\jocomment{I rephrased the paragraph below slightly:}
{A prestable graph $\Gamma$ is called \textit{stable} if the following additional condition is satisfied:
\begin{enumerate}
    \item[(vi)] for each vertex $w\in V$, we have 
    \[2g(w)-2+n(w)>0\,,\]
    where $n(w) = |v^{-1}(w)|$ is the valence of $w$ in $\Gamma$, i.e. the number of half-edges incident to $w$.
\end{enumerate}
}
Given a second graph $\Gamma' = (V',H', \ell',g',v',\iota')$ an \emph{isomorphism} $\varphi: \Gamma \to \Gamma'$ is the data of bijective maps
\[\varphi_V : V \to V'\,,\hspace{2mm} \varphi_H : H \to H'\]
which are compatible with the remaining data of the prestable graphs, in the sense that
\[
\ell' \circ \varphi_H|_{L} = \ell\,, \, g' \circ \varphi_V = g\,, \, v' \circ \varphi_H = \varphi_V \circ v\,, \, \iota' \circ \varphi_H = \varphi_H \circ \iota\,.
\]

For every vertex $v \in V(\Gamma)$ let $H(v)$ be the set of half-edges at $v$, with cardinality $n(v)$. Then there exists a natural gluing morphism
\begin{align*}
 \xi_\Gamma : \M_\Gamma=\prod_{v \in V(\Gamma)} \M_{g(v), n(v)} &\to \M_{g,n}\,,
\end{align*}
which assigns to a collection $((C_v, (p_h)_{h \in H(v)})$ the curve $(C,p_1, \ldots, p_n)$ obtained by identifying the markings $p_h, p_{h'}$ for each pair $(h,h')$ forming an edge of $\Gamma$.
\footnote{Note that while it is customary in the field to write the factors of the domain of $\xi_\Gamma$ as $\M_{g(v),n(v)}$, it would maybe be more appropriate to define prestable curves with markings indexed by the set $H(v)$ and write $\M_{g(v),H(v)}$, since otherwise we need to implicitly choose an ordering on the half-edges at $v$ to define the map $\xi_\Gamma$. This does not affect the arguments presented below and the reader may assume that an arbitrary such an ordering is chosen.}
Restricting to the preimage of the open substack $\Mbar_{g,n} \subset \M_{g,n}$ we get the usual gluing maps
\begin{align*}
 \xi_\Gamma : \Mbar_\Gamma=\prod_{v \in V(\Gamma)} \Mbar_{g(v), n(v)} &\to \Mbar_{g,n}\,,
\end{align*}
Note that unless $\Gamma$ is stable, the left hand side is empty.

On the other hand, given $m \geq 0$ we have the forgetful morphism
\begin{align*}
 F_m : \M_{g,n+m} \to \M_{g,n}, (C,p_1, \ldots, p_n, q_1, \ldots, q_m) \mapsto (C,p_1, \ldots, p_n)\,.
\end{align*}
Since the curve $C$ remains prestable after forgetting a subset of the markings, there is no stabilization procedure in the morphism $F_m$ and the underlying curve remains unchanged.
\begin{lemma} \label{Lem:forgetfulmorphism}
The morphism $F_m$ is smooth and representable of relative dimension $m$ and
the collection 
\begin{equation*}
    \left(F_m|_{\Mbar_{g,n+m}} : \Mbar_{g,n+m} \to \M_{g,n}\right)_{m \in \mathbb{Z}_{\geq 0}}
\end{equation*}
forms a smooth and representable cover of $\M_{g,n}$. 
%\jocomment{Previously we wrote "atlas" here, but We should be careful what we mean by atlas, since the $\Mbar_{g,n+m}$ are not schemes. If you are happy with the reformulation, you can delete this comment.}
The complement of the image of $\Mbar_{g,n+m}$ under $F_m$ in $\M_{g,n}$ has codimension $\lfloor \frac{m}{2} \rfloor +1$, except for finitely many $m$ in the unstable setting $2g-2+n \leq 0$.\footnote{The exceptions occur for $(g,n,m)=(1,0,0)$ and for $g=0, n \leq 2, m \leq 4$. We leave it as an exercise to the reader to work out the codimension in these cases.}
% \begin{equation*}
%     \mathrm{codim}(\M_{g,n} \setminus F_m(\Mbar_{g,n+m})) = \begin{cases}
%     \lfloor \frac{m}{2} \rfloor +1&\text{for }g>0\text{ or }g=0, n \geq 3
%     \lfloor \frac{m+n-3}{2} \rfloor &\text{for }g = 0, n \leq 2.
%     \end{cases}
% \end{equation*}
%\jocomment{The behaviour for $g=0, n \leq 2, m \leq 4$ is kind of annoying to describe (essentially one should give a table), so I put things in a footnote.} 
\end{lemma}
\begin{proof}
Except for the statement about the codimension of the complement of the image, this is \cite[Proposition 2]{behrendGW}. To show the formula for the codimension, observe on the one hand that in a prestable graph $\Gamma$, every unstable vertex can be stabilized by adding at most two legs. Conversely, consider the prestable graph $\Gamma_0$ formed by a central vertex of genus $g$ with all $n$ legs, connected via single edges to $c$ outlying vertices of genus $0$ with no legs. Then $\Gamma_0$ belongs to a codimension $c$ stratum and we need precisely $2c$ additional legs to stabilize it. Thus the stratum of $\M_{g,n}$ associated to $\Gamma_0$ lies in the complement of $F(\Mbar_{g,n+m})$ if and only if $c\geq \lfloor \frac{m}{2} \rfloor +1$. The finitely many exceptions in the unstable range arise from the fact that the central vertex of $\Gamma_0$ is not stable if $2g-2+n+c<0$.
\end{proof}

Let $\st_m(\Gamma)$ be the set of stable graphs $\Gamma'$ in genus $g$ with $n+m$ markings obtained from a prestable graph $\Gamma$ of genus $g$ with $n$ legs by adding $m$ additional legs, labeled $n+1, \ldots, n+m$, at vertices of $\Gamma$. As explained above, for a fixed prestable graph $\Gamma$ the set $\st_m(\Gamma)$ starts being nonempty for $m$ sufficiently large.

Given $\Gamma' \in \st_m(\Gamma)$ there is a natural map
\[F_{\Gamma' \to \Gamma} : \Mbar_{\Gamma'} \to \M_\Gamma\]
which is just the product of forgetful maps $F_{m_v} : \Mbar_{g(v),n(v)+m_v} \to \M_{g(v),n(v)}$ for each $v \in V(\Gamma)=V(\Gamma')$.

\begin{lemma} \label{Lem:gluingmaps}
For every prestable graph $\Gamma$ in genus $g$ with $n$ markings and every $m \geq 0$ there is a fibre diagram
\begin{equation} \label{eqn:gluingpullbackdiag}
 \begin{tikzcd} 
 \coprod_{\Gamma' \in \st_m(\Gamma)} \Mbar_{\Gamma'} \arrow[r,"\coprod \xi_{\Gamma'}"] \arrow[d,"\coprod F_{\Gamma' \to \Gamma}"] & \Mbar_{g,n+m} \arrow[d,"F_m"]\\
 \M_\Gamma \arrow[r,"\xi_\Gamma"] & \M_{g,n}\,.
 \end{tikzcd}
\end{equation}
In particular the map $\xi_\Gamma: \M_\Gamma \to \M_{g,n}$ is representable, proper and a local complete intersection.
\end{lemma}
\begin{proof}
An object of the fibre product of $\M_\Gamma$ with $\Mbar_{g,n+m}$ over a (connected) scheme $S$ is given by
\begin{itemize}
    \item a collection of families $(C_v, (p_h)_{h \in H(v)})$ of prestable curves over $S$ for each $v \in V(\Gamma)$,
    \item a family $(C',p'_1, \ldots, p'_n, q'_1, \ldots, q'_m)$ of stable curves over $S$,
    \item an isomorphism (of families of prestable curves) 
    \[\varphi: C=\coprod_v C_v / (p_h \sim p_{h'}, (h,h') \in E(\Gamma)) \to C'\]
    satisfying $\varphi(p_i)=p_i'$. 
\end{itemize}
By the assumption that $S$ is connected, for each $j=1, \ldots, m$ there exists a unique $v=v(j) \in V(\Gamma)$ such that $q'_j \in \varphi(C_v)$ at each point of $S$. This uses that via $\varphi$, the smooth unmarked points of $C_v$ ($v \in V(\Gamma)$) form a disjoint open cover of the smooth unmarked points of $(C',p_1, \ldots, p_n)$ in which $q'_j$ is always contained.

But for $j,v$ as above, we obtain a section $q_j=\varphi^{-1} \circ q_j' : S \to C_v$ landing in the smooth unmarked locus of $C_v$. Thus for every $v \in V(\Gamma)$ this allows us to define a family
\begin{equation} \label{eqn:Cvstable}\widehat{C}_v=(C_v, (p_h)_{h \in H(v)}, (q_j)_{v(j)=v}) \to S\end{equation}
of prestable curves over $S$. From the fact that $(C',p'_1, \ldots, p'_n,q'_1, \ldots, q'_m)$ is a family of stable curves, it follows that the family (\ref{eqn:Cvstable}) is actually a family of stable curves. Then one sees that the collection $({\widehat{C}}_v)_{v \in V(\Gamma')}$ is exactly an $S$-point of one of the spaces $\Mbar_{\Gamma'}$ for the suitable $\Gamma' \in \st_m(\Gamma)$ for which the marking $q_j$ is added at the vertex $v(j) \in V(\Gamma')=V(\Gamma)$.

The above operations defines a map from $\M_\Gamma \times_{\M_{g,n}} \Mbar_{g,n+m}$ to the disjoint union of the $\Mbar_{\Gamma'}$ and clearly this disjoint union also maps to the fibre product using the maps $F_{\Gamma' \to \Gamma}$ and $\xi_{\Gamma'}$. One verifies that these are inverse isomorphisms.

Since being proper and being a local complete intersection is local on the target %https://stacks.math.columbia.edu/tag/0CHM
%https://stacks.math.columbia.edu/tag/02L1
and since the maps $F_m$ form a smooth cover of $\M_{g,n}$, these properties of $\xi_{\Gamma}$ follow from the corresponding properties of the maps $\xi_{\Gamma'}|_{\Mbar_{\Gamma'}}$.
\end{proof}

%\jocomment{I added an approach to $\M^\Gamma$ that is hopefully more rigorous (though less conceptual). I think an alternative would be to say: a family of prestable curves over $S$ is said to have prestable graph $\Gamma$ if, after an etale base change $S' \to S$ the resulting map $S' \to \M_{g,n}$ factors through $\xi_\Gamma$. Then $\M^\Gamma$ is the full subcategory of $\M_{g,n}$ of curves with prestable graph $\Gamma$.  This is more conceptual (and geometric points are clear), but it is less clear to me that this is a locally closed substack. Maybe you can read below and say what you think.}\Ycomment{Great. I think the point is well resolved. }
Later we will need some stronger statements about the locus of curves whose prestable graph is exactly a given graph $\Gamma$.
This locus is a locally closed substack $\M^\Gamma$ of $\M_{g,n}$ whose geometric points are precisely the curves $(C,p_1, \ldots, p_n)$ with prestable graph isomorphic to $\Gamma$. 
However, since a family of prestable curves over an arbitrary base does not in general have a well-defined prestable graph, this definition is slightly tricky to write down in a functorial way. Thus we approach the definition from a different angle and then show that it defines the desired locus.
\begin{definition} \label{Def:Gammalocus}
Let $\Gamma$ be a prestable graph in genus $g$ with $n$ markings and let $e$ be the number of edges of $\Gamma$. Then we define
\[\M^\Gamma = \mathrm{im}(\xi_\Gamma) \setminus \bigcup_{\Gamma': |E(\Gamma')|=e+1} \mathrm{im}(\xi_{\Gamma'}),\]
where $\mathrm{im}$ denotes the stack theoretic image and the union goes over prestable graphs $\Gamma'$ with precisely $e+1$ edges.
% Then the full subcategory
% \[\M^\Gamma = \{(C,p_1, \ldots, p_n) \in \M_{g,n} : \Gamma(C,p_1, \ldots, p_n) \cong \Gamma\} \subset \M_{g,n}\]
% of stable curves with prestable graph $\Gamma$ is a locally closed substack of $\M_{g,n}$.
\end{definition}
By definition, we have that $\M^\Gamma$ is a locally closed substack of $\M_{g,n}$. In the following lemma, we check that its geometric points are as desired.
\begin{lemma} \label{Lem:MGamma}
The geometric points of $\M^\Gamma$ are precisely the $(C,p_1, \ldots, p_n)$ with prestable graph isomorphic to $\Gamma$.
\end{lemma}
\begin{proof}
First we note that since $\xi_\Gamma$ is proper, it is surjective onto its image. Then on the one hand each $(C,p_1, \ldots, p_n)$ with prestable graph isomorphic to $\Gamma$ is in $\M^\Gamma$, since it is in the image of $\xi_\Gamma$ but cannot be in the image of a gluing map for a graph $\Gamma'$ with more than $e$ edges (since its number of nodes is precisely $e$). Conversely, let $(C_v)_v = (C_v,(p_h)_{h\in H(v)})_{v \in V(\Gamma)} \in \M_\Gamma$ be a geometric point. Then if all $C_v$ are smooth, its image $\xi_\Gamma((C_v)_v)$ has prestable graph $\Gamma$. On the other hand, if any of the $C_v$ are not smooth, then the prestable graph of $\xi_\Gamma((C_v)_v)$ has at least $e+1$ edges. By contracting all but $e+1$ of them, we obtain one of the prestable graphs $\Gamma'$ in the definition of $\M^\Gamma$, and it is easy to see that $\xi_\Gamma((C_v)_v)$ is then in the image of $\xi_{\Gamma'}$.
\end{proof}

We have the following neat description of $\M^\Gamma$ which is a generalization of \cite[Lemma 5.1]{fulghesurational}. For the statement, let
\begin{equation*}
    \M_{g,n}^\textup{sm} \subset \M_{g,n}
\end{equation*}
be the open substack where the curve $C$ is smooth. For $g,n$ in the stable range, this is the usual stack $\mathcal{M}_{g,n}$ of smooth curves, but since the latter might be defined to be empty for $2g-2+n<0$ we use the notation $\M_{g,n}^\textup{sm}$ for clarity.
\begin{proposition} \label{Pro:MGamma}
For a prestable graph $\Gamma$, consider the open substack
\[\M_\Gamma^\textup{sm} = \prod_{v \in V(\Gamma)} \M_{g(v),n(v)}^\textup{sm} \subset \M_\Gamma. \]
Then the restriction of the gluing map $\xi_\Gamma$ to $\M_\Gamma^\textup{sm}$ factors through $\M^\Gamma$ and it is invariant under the natural action of $\aut(\Gamma)$. The induced map
\begin{equation} \label{eqn:xigammaquotient} \M_\Gamma^\textup{sm} / \aut(\Gamma) \xrightarrow{\xi_\Gamma} \M^\Gamma\end{equation}
from the quotient stack\footnote{Since $\M_\Gamma^\textup{sm}$ is not an algebraic space, one can either use the notion of group actions and quotients for algebraic stacks defined by Romagny \cite{romagny} to make sense of the quotient $\M_\Gamma^\textup{sm} / \aut(\Gamma)$ or one observes that $\M_\Gamma^\textup{sm}$ is itself a quotient stack and that the action of $\aut(\Gamma)$ can be lifted compatibly to write $\M_\Gamma^\textup{sm}/\aut(\Gamma)$ again as a quotient stack. See \cite[Section 5]{fulghesurational} for more details.} of $\M_\Gamma$ by $\aut(\Gamma)$ is an isomorphism.
\end{proposition}
\begin{proof}
For each point $(C_v)_v = (C_v,(p_h)_{h\in H(v)})_{v \in V(\Gamma)} \in \M_\Gamma^\textup{sm}$, the stabilizer $\aut(\Gamma)_{(C_v)_v}$ under the action of $\aut(\Gamma)$ is the set of automorphisms of $\Gamma$ such that there exist compatible isomorphisms of the curves $(C_v,(p_h)_{h\in H(v)})$. 
The stabilizer group of $[(C_v)_v] \in \M_\Gamma^\textup{sm} / \aut(\Gamma)$ is then an extension of the product of the automorphism groups of the $(C_v,(p_h)_{h\in H(v)})$ by the group $\aut(\Gamma)_{(C_v)_v}$.

On the other hand, for the curve $(C,p_1,\ldots,p_n)$ obtained from $(C_v)_v$ by gluing and an element $\sigma \in \aut(\Gamma)_{(C_v)_v}$, the isomorphisms between the curves $C_v$ that are compatible with $\sigma$ can be glued to an automorphism of $(C,p_1,\ldots,p_n)$. From this it follows that there exists an exact sequence
\[1\to\prod_{v\in V(\Gamma)} \aut(C_v,(p_h)_{h\in H(v)}) \to\aut(C,p_1,\ldots,p_n)\to \aut(\Gamma)_{(C_v)_v} \to 1\,.\]
From this sequence we see that $\aut(C,p_1,\ldots,p_n)$ is precisely the group extension defining the stabilizer of $[(C_v)_v] \in \M_\Gamma^\textup{sm} / \aut(\Gamma)$ and hence $\xi_\Gamma$ induces an isomorphism of each stabilizer. Thus the morphism $\xi_\Gamma$ in \eqref{eqn:xigammaquotient} is representable. It is easy to check that it is bijective on geometric points and it is separated by similar argument as in Lemma~\ref{Lem:gluingmaps}. So by \cite[0DUD]{stacks-project} it is enough to show that $\xi_\Gamma$ is an \'etale morphism to conclude that it is an isomorphism. 

Consider the atlas $F_m$ restricted to $\M^{\Gamma}$. Since being \'etale is local on the target, it is enough to show that $\xi_\Gamma$ is \'etale on each atlas.  On each atlas, the dimension of the fiber is constantly zero. The domain of $\xi_\Gamma$ is smooth because it can be written as a quotient of a smooth algebraic space by a group scheme  (\cite[0DLS]{stacks-project}). Following a slight variation of the proof of \cite[Proposition 10.11]{GeometryofCurves2}, the stack $\M^{\Gamma}$ is also smooth. Since the domain and the target of $\xi_{\Gamma}$ are smooth, the `miracle flatness' (\cite[00R3]{stacks-project}) implies that $\xi_\Gamma$ is flat. Furthermore the morphism is smooth because it is flat and each geometric fiber is smooth. Smooth and quasi-finite morphisms are \'etale and hence $\xi_\Gamma$ is an isomorphism. 
\end{proof}
\subsection{\texorpdfstring{$\A$}{A}-valued prestable curves} \label{Sect:Avaluedprestable}
For each $g,n$ there exists the universal curve $\C_{g,n} \to \M_{g,n}$. For later applications, it will be necessary to compute with tautological classes on $\C_{g,n}$ (and tautological classes on the universal curve over $\C_{g,n}$, etc). For the moduli spaces of stable curves, a separate theory is not necessary because the universal curve over $\Mbar_{g,n}$ is given by the forgetful map $\Mbar_{g,n+1} \to \Mbar_{g,n}$. The same is \emph{not} true for $\M_{g,n}$. Indeed, in Lemma \ref{Lem:forgetfulmorphism} we saw that the forgetful morphism $\M_{g,n+1} \to \M_{g,n}$ is smooth, so it cannot be the universal curve over $\M_{g,n}$. In this section we put an additional structure on prestable curves, called the $\A$-value, which allows to give a modular interpretation of the universal curve as a stack of $(n+1)$-pointed curves with additional structure. This realization will be convenient to compute tautological classes on $\C_{g,n}$.  

So let us start by recalling  the notion of prestable curves with values in a semigroup $\A$ from \cite{costello}. 
%This will be useful, since it will allow us to give a different modular interpretation of the universal curve $\pi: \C_{g,n} \to \M_{g,n}$ over $\M_{g,n}$ and to treat both the intersection theory on $\M_{g,n}$ and on $\C_{g,n}$ in the same framework.
In what follows let $\A$ be a commutative semigroup with unit $\zero \in \A$ such that
\begin{itemize}
    \item $\A$ has indecomposable zero, i.e. for $x,y \in \A$ we have $x+y=\zero$ implies $x=\zero$, $y=\zero$,
    \item $\A$ has finite decomposition, i.e. for $a \in \A$ the set $$\{(a_1,a_2) \in \A \times \A : a_1 + a_2 =a\}$$ is finite.
\end{itemize}
Classical examples include $\A=\{\zero\}$ or $\A=\mathbb{N}$, but later we are going to work with 
\[\A=\{\zero,\one\}\text{ with }\zero + \zero = \zero, \zero + \one = \one + \zero = \one + \one = \one.\]

Fixing $\A$ and an element $a \in \A$, Behrend-Manin \cite{behrendmanin} and Costello \cite{costello} define an algebraic stack $\M_{g,n,a}$. {A geometric point corresponds} to a prestable curve $(C,p_1, \ldots, p_n)$ together with a map $C_v \mapsto a_{C_v}$ from the set of irreducible components $C_v$ of the normalization of $C$ to $\A$ such that the sum of all $a_{C_v}$ equals $a$. The curve must satisfy the stability condition that for each $C_v$ either $a_{C_v} \neq \zero$ or that $C_v$ is stable, in the sense that for $g(C_v)=0$ it carries three special points and for $g(C_v)=1$ it carries at least one special point. Over an arbitrary base scheme the definition of $\A$-valued stable curves needs extra care, see \cite[p.569]{costello} for details.
As an example, for any $\A$ as above and $a=\zero$ we obtain $\M_{g,n,\zero}=\Mbar_{g,n}$.

Our main motivation for considering the moduli spaces $\M_{g,n,a}$ is the fact that we have a forgetful morphism $\pi : \M_{g,n+1,a} \to \M_{g,n,a}$ making $\M_{g,n+1,a}$ the universal curve over $\M_{g,n,a}$. The image of a point 
\[(C, p_1, \ldots, p_n, p_{n+1}, (a_{C_v})_v) \in \M_{g,n+1,a}\]
under $\pi$ is formed by first forgetting the marked point $p_{n+1}$. Then, if the component $C_v$ of $C$ containing $p_{n+1}$ then becomes unstable\footnote{This happens precisely for $a_{C_v}=\zero$ and $C_v$ being of genus $0$ with at most two special points apart from $p_{n+1}$.}, the component $C_v$ of $C$ is contracted.
With this notation in place, we summarize the relevant properties of $\M_{g,n,a}$ from \cite{costello}.
\begin{proposition} \label{Pro:Mgnaprop}
The stack $\M_{g,n,a}$ is a smooth, algebraic stack, locally of finite type and the morphism $\M_{g,n,a} \to \M_{g,n}$ forgetting the values in $\A$ is \'etale and relatively a scheme of finite type. The universal curve over $\M_{g,n,a}$ is given by the forgetful morphism $\pi : \M_{g,n+1,a} \to \M_{g,n,a}$.
\end{proposition}
\begin{proof}
See Proposition 2.0.2 and 2.1.1 from \cite{costello}.
\end{proof}
The fact that the universal curve is given by a moduli space of curves with an extra marked point turns out to be very convenient. Indeed, as discussed above this is not the case for the forgetful morphism $\M_{g,n+1} \to \M_{g,n}$. Indeed, it is easy to identify $\M_{g,n+1}$ as the open substack $\M_{g,n+1} \subset \C_{g,n}$ given as the complement of the set of markings and nodes.

Many other constructions we saw for prestable curves work in the $\A$-valued setting. For instance, for $g_1+g_2=g$, $n_1+n_2=n$ and $a_1, a_2 \in \A$ with $a_1 + a_2 = a$, we have a gluing morphism
\[\xi : \M_{g_1,n_1+1,a_1} \times \M_{g_2,n_2+1,a_2} \to \M_{g,n,a}.\]
These gluing maps are again representable, proper and local complete intersections. Indeed, we have a fibre diagram
\begin{equation*}
\begin{tikzcd}
 \coprod_{a_1+a_2=a} \M_{g_1,n_1+1,a_1} \times \M_{g_2,n_2+1,a_2} \arrow[r] \arrow[d] & \M_{g,n,a} \arrow[d]\\
 \M_{g_1,n_1+1} \times \M_{g_2,n_2+1} \arrow[r] & \M_{g,n}
\end{tikzcd}    
\end{equation*}
and the map at the bottom has all these properties by Lemma \ref{Lem:gluingmaps}. More generally, one defines the notion of an $\A$-valued stable graph and the corresponding gluing map has all the desired properties.

The following result allows us to apply the machinery of Costello to the moduli spaces of prestable curves.
\begin{proposition} \label{Prop:Mgnsemigroup}
Let $\A=\{\zero,\one\}$ with $\one + \one = \one$, then given $g,n$ the subset $\mathfrak Z_{g,n} \subset \M_{g,n,\one}$ of $\A$-valued curves $(C,p_1, \ldots, p_n; (a_{C_v})_v)$ such that one of the values $a_{C_v}$ equals $\zero$ is closed. Let $\mathfrak U_{g,n}=\M_{g,n,\one} \setminus \mathfrak Z_{g,n}$ be its complement. Then the composition
\begin{equation*} \label{eqn:Ugndef} 
\mathfrak U_{g,n} \hookrightarrow \M_{g,n,\one} \to \M_{g,n}
\end{equation*}
of the inclusion of $\mathfrak U_{g,n}$ with the morphism $\M_{g,n,\one} \to \M_{g,n}$ forgetting the $\A$-values defines an isomorphism $\mathfrak{U}_{g,n} \cong \M_{g,n}$.
\end{proposition}
\begin{proof}
The underlying reason why $\mathfrak{Z}_{g,n}$ is closed is that $\zero$ is indecomposable in $\A$ : given a curve $(C,p_1, \ldots, p_n; (a_{C_v})_v)$ such that some $a_{C_v}=\zero$, any degeneration of this curve still has some component with value $\zero$ since in a degeneration of $C_v$, $a_{C_v}$ must distribute to the components to which $C_v$ degenerates.

More concretely, we can write $\mathfrak{Z}_{g,n}$ as the union of images of gluing maps $\xi_\Gamma$ for suitable $\A$-valued prestable graphs $\Gamma$. Indeed, we exactly have to remove the images of $\xi_\Gamma$ for $\Gamma$ of the form 
    \begin{equation} \label{eqn:Gamma_unst_picture} \Gamma \ = \ 
    \begin{tikzpicture}[baseline=1.6 cm,label distance=0.5cm,thick,
    virtnode/.style={circle,draw,scale=0.5}, 
    nonvirt node/.style={circle,draw,fill=black,scale=0.5} ]
    % \node [virtnode] (A) {} node[below]{$A_1$}; works
    \node [virtnode,label=left:$(g_0{,}\zero)$] (A) {};
    \node at (-2,2) [nonvirt node,label=left:$(g_1{,}\one)$] (B) {};
    \node at (2,2) [nonvirt node,label=right:$(g_s{,}\one)$] (C) {};
    \node at (0,2) {$\cdots$};
    \draw [-] (A) to (B);
    \draw [-,bend left] (A) to node[midway,left] {$e_1$} (B) ;
    \draw [-,bend right] (A) to (B);
    \draw [-,bend left] (A) to (C);
    \draw [-,bend right] (A) to node[midway,right] {$e_s$} (C);    
    
    \node at (.5,-.7) (n1) {};
    \node at (0,-.86) (n0) {$I_0$};
    \node at (-.5,-.7) (n2) {};
    \draw [-] (A) to (n1);
    \draw [-] (A) to (n0);
    \draw [-] (A) to (n2);
    
    \node at (-2.5,2.7) (m1) {};
    \node at (-2,2.86) (m0) {$I_1$};
    \node at (-1.5,2.7) (m2) {};
    \draw [-] (B) to (m1);
    \draw [-] (B) to (m2);    
    
    \node at (2,2.86) (o) {$I_s$};
    \draw [-] (C) to (o);
    \end{tikzpicture}
    \end{equation}
where $s \geq 1$, $e_1, \ldots, e_s \in \mathbb{Z}_{>0}$, \[g_0+g_1+\ldots+g_s=g+\sum_{i=1}^s e_i-1\]
and $I_0 \coprod I_1 \coprod \ldots \coprod I_s = \{1, \ldots, n\}$. Note that for this locus to be nonempty, we must require $g_0>0$ or $|I_0|+\sum_i e_i>2$.

While the image of each $\xi_\Gamma$ is closed, we use infinitely many of them. But in the open exhaustion of $\M_{g,n,\one}$ by the substacks of curves with at most $\ell$ nodes, each of these open substacks only intersects finitely many of the images of $\xi_\Gamma$ nontrivially, so still the union of their images is closed.

The fact that $\mathfrak U_{g,n} \to \M_{g,n}$ is an isomorphism can be seen in different ways: its inverse is just given by the functor sending each prestable curve $(C,p_1, \ldots, p_n)$ to itself with value $a_{C_v} = \one$ on each component, i.e. $$\M_{g,n} \to \mathfrak U_{g,n}, (C,p_1, \ldots, p_n) \mapsto (C,p_1, \ldots, p_n; (\one)_v).$$
Alternatively one observes that $\mathfrak U_{g,n} \to \M_{g,n}$ is \'etale, representable and a bijection on geometric points.
\end{proof}

\begin{corollary} \label{Cor:modinterpretunivcurve}
The universal curve $\C_{g,n} \to \M_{g,n}$ is given by the morphism
\[\pi: \M_{g,n+1,\one} \setminus \pi^{-1}(\mathfrak Z_{g,n}) \to \M_{g,n,\one} \setminus \mathfrak Z_{g,n}\]
forgetting the marking $n+1$ and contracting the component containing it if this component becomes unstable. The $\A$-valued prestable graphs $\Gamma$ appearing in $\mathfrak{Z}_{g,n+1}$ but \emph{not} contained in $\pi^{-1}(\mathfrak Z_{g,n})$ are exactly of one of the three following forms
\begin{itemize}
    \item for $i=1, \ldots, n$ the graphs
    \begin{equation} \label{eqn:sectiongraph} 
    \Gamma \ = \ 
    \begin{tikzpicture}[baseline=-3pt,label distance=0.5cm,thick,
    virtnode/.style={circle,draw,scale=0.5}, 
    nonvirt node/.style={circle,draw,fill=black,scale=0.5} ]
    % \node [virtnode] (A) {} node[below]{$A_1$}; works
    \node [virtnode,label=below:$(0{,}\zero)$] (A) {};
    \node at (2,0) [nonvirt node,label=below:$(g{,}\one)$] (B) {};
    \draw [-] (A) to (B);
    \node at (-0.9,.5) (n1) {$i$};
    % \node at (-.86,0) (n0) {$I_1$};
    \node at (-1.3,-.5) (n2) {$n+1$};
    \draw [-] (A) to (n1);
    \draw [-] (A) to (n2);
    
    \node at (2.7,.5) (m1) {};
    \node at (3.86,0) (m0) {$\{1, \ldots, n\} \setminus \{i\}$};
    \node at (2.7,-.5) (m2) {};
    \draw [-] (B) to (m1);
    \draw [-] (B) to (m2);    
    \end{tikzpicture}
    \end{equation}
    corresponding to the $n$ sections of the universal curve $\pi: \C_{g,n} \to \M_{g,n}$,
    \item boundary divisors with edge subdivided, inserting a genus zero, value $\zero$ vertex carrying $n+1$
    
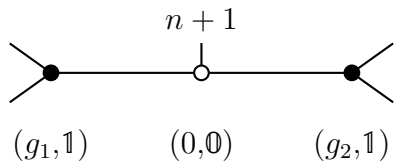
\begin{equation*} 
    \begin{tikzpicture}[baseline=-3pt,label distance=0.5cm,thick,
    virtnode/.style={circle,draw,scale=0.5}, 
    nonvirt node/.style={circle,draw,fill=black,scale=0.5} ]
    % \node [virtnode] (A) {} node[below]{$A_1$}; works
    \node at (0,0) [nonvirt node,label=below:$(g_1{,}\one)$] (A) {};
    \node at (2,0) [virtnode,label=below:$(0{,}\zero)$] (B) {};
    \node at (4,0) [nonvirt node,label=below:$(g_2{,}\one)$] (C) {};
    \draw [-] (A) to (B);
    \draw [-] (B) to (C);
    \node at (-.7,.5) (n1) {};
    \node at (-.7,-.5) (n2) {};
    \draw [-] (A) to (n1);
    \draw [-] (A) to (n2);
    
    \node at (2,.7) (k1) {$n+1$};
    \draw [-] (B) to (k1);
    \node at (4.7,.5) (m1) {};
    \node at (4.7,-.5) (m2) {};
    \draw [-] (C) to (m1);
    \draw [-] (C) to (m2); 
    \end{tikzpicture}
    \end{equation*}
    where $g_1+g_2=g$ and
    
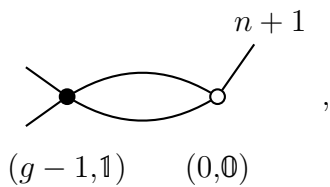
\begin{equation*} 
    \begin{tikzpicture}[baseline=-3pt,label distance=0.5cm,thick,
    virtnode/.style={circle,draw,scale=0.5}, 
    nonvirt node/.style={circle,draw,fill=black,scale=0.5} ]
    % \node [virtnode] (A) {} node[below]{$A_1$}; works
    \node at (0,0) [nonvirt node,label=below:$(g-1{,}\one)$] (A) {};
    \node at (2,0) [virtnode,label=below:$(0{,}\zero)$] (B) {};
    \draw [-,bend left] (A) to (B);
    \draw [-,bend right] (A) to (B);
    
    \node at (-.7,.5) (n1) {};
    \node at (-.7,-.5) (n2) {};
    \draw [-] (A) to (n1);
    \draw [-] (A) to (n2);
    
    \node at (2.7,) (k1) {$n+1$};
    \draw [-] (B) to (k1);    
    \end{tikzpicture}\,,
    \end{equation*}
    corresponding to the locus of nodes inside the universal curve $\pi: \C_{g,n} \to \M_{g,n}$.
\end{itemize}
%They correspond to the components of $\C_{g,n} \setminus \M_{g,n+1}$, namely the image of the sections $\sigma_i: \M_{g,n} \to \C_{g,n}$ and the loci of nodes in $\C_{g,n}$.
\end{corollary}
Corollary \ref{Cor:modinterpretunivcurve} shows that in order to develop the intersection theory of $\M_{g,n}$ and $\C_{g,n}$, it suffices to consider the general case of the intersection theory of $\M_{g,n,\one}$ (or even more generally, $\M_{g,n,a}$ for any semigroup $\A$ and $a \in \A$).
\section{Chow groups and the tautological ring of \texorpdfstring{$\M_{g,n}$}{Mgn}} \label{Sect:ChowandTaut}
\subsection{Definitions}\label{sec:definitions}
In this paper, we want to study the Chow groups (with $\mathbb{Q}$-coefficients) of the stacks $\M_{g,n}$ (and more generally, the stacks $\M_{g,n,a}$ for some element $a \in \A$ in a semigroup $\A$).

To define these Chow groups, recall that in \cite{kreschartin} Kresch constructed Chow groups $\CH_*(\mathcal{X})$ for algebraic stacks $\mathcal{X}$ of finite type over a field $k$. Moreover, there is an intersection product on $\CH_*(\mathcal{X})$ when $\mathcal{X}$ is smooth and \emph{stratified by {global} quotient stacks}\footnote{This means that there exists a stratification by locally closed substacks which are each isomorphic to a global quotient of an algebraic space by a linear algebraic group.}, see \cite[Theorem 2.1.12]{kreschartin}. This last condition can be checked point-wise:  a reduced stack $\mathcal{X}$ is stratified by {global} quotient stacks if and only if the stabilizers of geometric points of $\mathcal{X}$ are affine (\cite[Proposition 3.5.9]{kreschartin}).

Now the spaces $\M_{g,n,a}$ are in general not of finite type (only locally of finite type) and so we need to extend the definition of Chow groups above. Assume that 
 $\M$ is an algebraic stack, locally of finite type over a field $k$. Choose a directed system\footnote{Recall that this means that for all $\mathcal U_i, \mathcal U_j$ there exists a $\mathcal U_k$ containing both of them.} $(\mathcal{U}_i)_{i \in I}$  of finite type open substacks of $\M$ whose union is all of $\M$. Then we set
\begin{equation*}
    \CH_*(\M) = \varprojlim_{i \in I} \CH_*(\mathcal U_i),
\end{equation*}
where for $\mathcal U_i \subseteq \mathcal U_j$ the transition map $\CH_*(\mathcal U_j) \to \CH_*(\mathcal U_i)$ is given by the restriction to $\mathcal U_i$. In other words, we have
\begin{equation*}
    \CH_*(\M) = \{(\alpha_i)_{i \in I} : \alpha_i \in \CH_*(\mathcal U_i), \alpha_j|_{\mathcal U_i} = \alpha_i\text{ for }\mathcal U_i \subseteq \mathcal U_j \}.
\end{equation*}
We give the details of this definition in Appendix \ref{Sect:Chowlocfintype} and show that the Chow groups of locally finite type stacks inherit all the usual properties (e.g. flat pullback, projective pushforward, Chern classes of vector bundles and Gysin pullbacks) of the Chow groups from \cite{kreschartin}. Moreover, if $\M$ is smooth and has affine stabilizer groups at geometric points, the intersection products on the groups $\CH_*(\mathcal U_i)$ give rise to an intersection product on $\CH_*(\M)$. In this case, for $\M$ equidimensional we often use the cohomological degree convention
$$\CH^*(\M) = \CH_{\dim \M -*}(\M).$$

\begin{proposition}
% The stack $\M_{g,n,a}$ admits a good filtration by finite-type substacks. The gluing and forgetful maps respect these filtrations. Thus Chow groups $\CH_{d}(\M_{g,n,a})$ are well-defined. For $(g,n) \neq (1,0)$, the stack $\M_{g,n,a}$ admits a stratification by quotient stacks and thus there are operational Chow groups $\CH^{d}_{\textup{op}}(\M_{g,n,a})$.
Let $g,n \geq 0$, let $\A$ be a semigroup with indecomposable zero and finite decomposition as in Section \ref{Sect:Avaluedprestable} and $a \in \A$. Then the stacks $\M_{g,n}$ and $\M_{g,n,a}$ have well-defined Chow groups $\CH_*(\M_{g,n})$ and $\CH_*(\M_{g,n,a})$. For $(g,n) \neq (1,0)$ the stabilizer groups of all geometric points of $\M_{g,n}$ and $\M_{g,n,a}$ are affine and so the Chow groups have an intersection product.
\end{proposition}
\begin{proof}
The stacks $\M_{g,n}$ and $\M_{g,n,a}$ are locally of finite type (and smooth) by Proposition \ref{Pro:Mgnaprop} and thus satisfy the conditions of Definition \ref{Def:Chowlocallyft} from the appendix. For the existence of intersection products, we need to check that geometric points have affine stabilizers. The stabilizer group of such a prestable curve is a finite extension of the automorphism groups of its components. The only non-finite automorphism groups that can occur here are in genus $0$ (where they are subgroups of $\mathrm{PGL}_{2}$ and thus affine) and in genus $1$ with no special points. Since the prestable curves are assumed to be connected, the last case can only occur for $(g,n)=(1,0)$.
% The filtration on $\M_{g,n}$ can be seen from the images of the charts by $\Mbar_{g,n+m}$. Since $\M_{g,n,a} \to \M_{g,n}$ is relatively of finite type, this lifts to a finite type filtration of $\M_{g,n,a}$. The condition of being stratified by quotient stacks can be checked point-wise by \cite[Proposition 3.5.9]{kreschartin}. Namely, the stabilizers of geometric points should be affine, which is true for all $(g,n)$ except $(1,0)$.
\end{proof}

Now recall from Definition \ref{Def:tautringintro} that the tautological rings $(\R^*(\M_{g,n,a}))_{g,n,a}$ are defined as the smallest system of $\QQ$-subalgebras with unit of the Chow rings $(\CH^*(\M_{g,n,a}))_{g,n,a}$ closed under taking pushforwards by the natural forgetful and gluing maps.

We recall the following particular examples of tautological classes:
\begin{definition}\label{Def:psikappa}
Let $\pi: \M_{g,n+1,a} \to \M_{g,n,a}$ be the universal curve over $\M_{g,n,a}$ and for $i=1, \ldots, n$ let $\sigma_i : \M_{g,n,a} \to \M_{g,n+1,a}$ be the section corresponding to the $i$-th marked points. Let $\omega_\pi$ be the relative canonical line bundle on $\M_{g,n+1,a}$. Then we define
\begin{equation} \label{eqn:defpsi}
\psi_i = \sigma_i^* c_1\left(  \omega_\pi \right) \in \CH^1(\M_{g,n,a}) \; \text{ for }i=1, \ldots, n
\end{equation}
and
\begin{align} \label{eqn:defkappa}
\kappa_m &= \pi_* \big(\psi_{n+1}^{m+1}\big) \in \CH^m(\M_{g,n,a})\,.
\end{align}
\end{definition}

% \begin{definition}
% Let $\pi: \C_{g,n} \to \M_{g,n}$ be the universal curve over $\M_{g,n}$ and for $i=1, \ldots, n$ let $\sigma_i : \M_{g,n} \to \C_{g,n}$ be the section corresponding to the $i$-th marked points with image $D_i = \sigma_i(\M_{g,n}) \subset \C_{g,n}$. Let $\omega_\pi$ be the relative canonical line bundle on $\C_{g,n}$ and
% \[\omega_\pi^\textup{log} = \omega_\pi(\sum_{i=1}^n D_i)\]
% the log-canonical bundle on $\C_{g,n}$. Then we define
% \begin{equation} \label{eqn:defpsi}
% \psi_i = \sigma_i^* c_1\left(  \omega_\pi \right) \in \CH^1(\M_{g,n}) \text{ for }i=1, \ldots, n.
% \end{equation}
% On the other hand, on $\C_{g,n}$ we set
% \begin{align} \label{eqn:defpsiC}
% \psi_i &= \pi^*(\psi_i) + [D_i] \in \CH^1(\C_{g,n}) \text{ for }i=1, \ldots, n,\nonumber\\
% \psi_{n+1} &= c_1(\omega_\pi^\textup{log}) \in \CH^1(\C_{g,n}).
% \end{align}
% Moreover, we define
% \begin{align} \label{eqn:defkappa}
% \kappa_a &= \pi_* \left( c_1(\omega_\pi^\textup{log})^{a+1} \right) \in \CH^a(\M_{g,n}),\nonumber\\
% \kappa_a &= \pi^* \kappa_a + \psi_{n+1}^a \in \CH^a(\M_{g,n}).
% \end{align}
% \end{definition}

% \begin{remark}
% In contrast to the situation for moduli spaces of stable curves, the fact that the universal curve over $\M_{g,n}$ does not equal $\M_{g,n+1}$ forces us to have a separate definition of $\psi$ and $\kappa$ classes on $\C_{g,n}$. It is a straightforward check that the restriction of these classes to the open substack $\M_{g,n+1} \subset \C_{g,n}$ gives the corresponding $\psi$ and $\kappa$ classes there.
% \end{remark}

\begin{definition} \label{def:decstratclass}
Let $\Gamma$ be an $\A$-valued prestable graph in genus $g$ with $n$ markings with total value $a \in \A$. For $\M_\Gamma = \prod_{v \in V(\Gamma)} \M_{g(v),n(v),a(v)}$ a decoration $\alpha$ on $\Gamma$ is an element of $\CH^*(\M_\Gamma)$ given by a product of $\kappa$ and $\psi$-classes on the factors $\M_{g(v),n(v),a(v)}$ of $\M_\Gamma$. Thus it has the form
\begin{equation} \label{eqn:alphadecoration}
\alpha = \prod_{v\in V} \left( \prod_{i \in H(v)} \psi_{v,i}^{a_i} \prod_{a=1}^{m_v} \kappa_{v,a}^{b_{v,a}} \right)\in \CH^*(\M_\Gamma)\end{equation}
{where $a_i, b_{v,a}\geq 0$ and $m_v\geq 0$ are some integers.}
We define the decorated stratum class $[\Gamma,\alpha]$ as the pushforward
\[[\Gamma,\alpha] = (\xi_\Gamma)_* \alpha \in \CH^*(\M_{g,n,a}).\]
\end{definition}
One of the main goals of this section is to show that the set of tautological classes $\R^*(\M_{g,n,a}) \subseteq \CH^*(\M_{g,n,a})$ is the $\mathbb{Q}$-linear span of all classes $[\Gamma,\alpha]$.

\begin{remark} \label{Rmk:classicaltautclasses}
We define tautological classes on the spaces $\M_{g,n}$ and $\C_{g,n}$ by seeing these stacks as open subsets of $\M_{g,n,\one}$ and $\M_{g,n+1,\one}$ for $\A=\{\zero,\one\}$ as in Corollary \ref{Cor:modinterpretunivcurve}.
Then tautological classes on $\M_{g,n}$ and $\C_{g,n}$ are given by the restrictions of tautological classes on $\M_{g,n,\one}$ and $\M_{g,n+1,\one}$.

From the point of view of decorated strata classes, note that for $\M_{g,n}$ only $\A$-valued prestable graphs where all values are $\one$ can contribute (and these are in natural bijections with prestable graphs without valuation). On the other hand, for $\C_{g,n}$ we can have vertices $v$ with value $\zero$ contributing nontrivial classes. 
This happens exactly for the graphs shown in Corollary \ref{Cor:modinterpretunivcurve}, corresponding to the universal sections of $\C_{g,n} \to \M_{g,n}$ and the loci of nodes inside $\C_{g,n}$ over boundary strata of $\M_{g,n}$.
% this happens exactly for $g(v)=0$, $n(v)=3$ and the marking $n+1$ incident to $v$. Let $\overline \Gamma$ be the graph obtained by forgetting $n+1$ and removing/contracting $v$.\jocomment{Describe more clearly.} If $v$ has one half-edge and a marking $i \in \{1, \ldots, n\}$ then $\M_\Gamma \to \M_{g,n+1,\one}$ parametrizes the section $\sigma_i$ over $\M_{\overline \Gamma} \to \M_{g,n,\one}$. If $v$ has two half-edges, the map $\M_\Gamma \to \M_{g,n+1,\one}$ corresponds to the section of the universal curve over $\M_{\overline \Gamma}$ corresponding to the node for the edge $e$ of $\overline \Gamma$ obtained
\end{remark}

% \begin{definition}
% The set of tautological classes $\R^*(\M_{g,n}) \subset \CH^*(\M_{g,n})$ is the $\mathbb{Q}$-linear span of all decorated strata classes $[\Gamma,\alpha]$ as in Definition \ref{def:decstratclass}.
% \end{definition}

% While this definition is very analogous to the case of stable curves, we need some new tools to define tautological classes on the universal curve $\C_{g,n}$. 

\subsection{Intersections and functoriality of tautological classes} \label{calculus}
% Here, it might make sense to also introduce tautological classes on the universal curve $\C_{g,n}$, since this actually appears in the comparison theorems for $\kappa, \psi$-classes between $\Mbar$ and $\M$.
% \jocomment{ I suspect there are some connections to the paper \url{https://arxiv.org/abs/1804.05467} by Pixton, we should certainly ask for his opinion at some point.}\Ycomment{I completely agree with you. We discussed a while ago how to realize formal classes appear in Pixton's paper. I wrote something in Section 4 as very premature form.}
In this section we describe how the classes $[\Gamma,\alpha]$ behave under taking intersections as well as pullbacks and pushforwards under natural gluing, forgetful and stabilization maps.

\subsubsection*{Pushforwards by gluing maps}
Pushing forward by gluing maps is by far the easiest operation: given an $\A$-valued graph $\Gamma_0$ and classes $[\Gamma_v,\alpha_v] \in \R^*(\M_{g(v),n(v),a(v)})$ for $v \in V(\Gamma_0)$, the pushforward of the class $$\prod_{v \in V(\Gamma)} [\Gamma_v,\alpha_v] \in \CH^*(\M_{\Gamma})$$ is given by $[\Gamma,\alpha]$ where $\Gamma$ is obtained by gluing the $\Gamma_i$ into the vertices of the outer graph $\Gamma_0$ and $\alpha$ is obtained by combining the decorations $\alpha_v$ using that $V(\Gamma) = \coprod_{v \in V(\Gamma_0)} V(\Gamma_v)$. 

\subsubsection*{Pullbacks by gluing maps and intersection products}
The next natural question is how a class $[B,\beta]$ pulls back along a gluing morphism $\xi_A$ for an $\A$-valued graph $A$. This operation allows a purely combinatorial description, generalizing the description in $\Mbar_{g,n}$ from \cite{graberpandh} (and already discussed for graphs $A$ with exactly one edge in \cite[Section 4]{costello}).
As combinatorial preparation, we recall the notion of morphisms of $\A$-valued stable graphs.
\begin{definition} \label{Def:Astructure}
% $A$-structure on a prestable graph $\Gamma$ (write $\Gamma \to A$),
An {\em $A$-structure} on an $\A$-valued prestable graph $\Gamma$ (write $\Gamma \to A$) is a choice of subgraphs $\Gamma_v$ of $\Gamma$ such that $\Gamma$ can be constructed by replacing each vertex $v$ of $A$ by the corresponding $\A$-valued graph $\Gamma_v$. More precisely, the data of $\Gamma \to A$ is given by 
maps
\[V(\Gamma) \to V(A)\text{ and }H(A) \to H(\Gamma).\]
They must satisfy that $V(\Gamma) \to V(A)$ is surjective, such that the preimage of $v \in V(A)$ are the vertices of a subgraph $\Gamma_v$ of $\Gamma$ with total $\A$-value $a_v$. The map $H(A) \to H(\Gamma)$ of half-edges in the opposite direction is required to be injective, and {allows one to see} identify half-edges $h \in H(v)$ of $A$ with legs of the graph $\Gamma_v$. These maps must respect the incidence relation of half-edges and vertices and the pairs of half-edges forming edges. In particular, the injection of half-edges allows to see the set of edges $E(A)$ of $A$ as a subset of the set of edges $E(\Gamma)$ of $\Gamma$ (see e.g. \cite[Definition 2.5]{SchmittvanZelm} for more details in the case of stable graphs).%\jocomment{Expand}
\end{definition}
Given an $A$-structure $\Gamma\to A$, there exists a gluing morphism 
\[\xi_{\Gamma\to A}\colon \M_\Gamma \to \M_A\,.\]
For a decoration $\alpha$ on $\M_A$ as in (\ref{eqn:alphadecoration}), it is easy to describe $\xi_{\Gamma \to A}^* \alpha$ using that
\begin{itemize}
    \item $\xi_{\Gamma \to A}^* \psi_{v,i} = \psi_{w,j}$ if $\Gamma \to A$ maps half-edge $i$ in $A$ to half-edge $j$ in $\Gamma$,
    \item $\xi_{\Gamma \to A}^* \kappa_{v,\ell} = \sum_{w \mapsto v} \kappa_{w,\ell}$, where the sum {goes} over vertices $w$ of $\Gamma$ mapping to the vertex $v$ of $A$ on which $\kappa_{v,\ell}$ lives.
\end{itemize}
Both these properties follow immediately from the definitions\footnote{For the pullback of $\kappa$-classes the proof also uses Proposition \ref{Pro:glueforgetpullback} below.} of $\kappa$ and $\psi$-classes. 

Let $f_A : \Gamma \to A$, $f_B : \Gamma \to B$ be $A$ and $B$-{structures} on the prestable graph $\Gamma$. The pair $f = (f_A, f_B)$ is called a {\em generic $(A,B)$-structure} $f=(f_A,f_B)$ on $\Gamma$ if every half-edge of $\Gamma$ corresponds to a half-edge of $A$ or a half-edge of $B$.
Given a second $(A,B)$-structure $f'=(f_A': \Gamma' \to A, f_B': \Gamma' \to B)$, an isomorphism from $f$ to $f'$ is an isomorphism $\Gamma \to \Gamma'$ commuting with the maps to $A,B$. 
Let $\mathcal{G}_{A,B}$ be the set of isomorphism classes of prestable graphs $\Gamma$ together with a generic $(A,B)$-structures on $\Gamma$.
%\jocomment{rephrased to make clear $\Gamma$ varies}

\begin{proposition}
Let $A,B$ be $\A$-valued prestable graphs for $\M_{g,n,a}$, then the fibre product of the gluing maps $\xi_A : \M_A \to \M_{g,n,a}$ and $\xi_B : \M_B \to \M_{g,n,a}$ is given by a disjoint union 
\begin{equation} \label{eqn:gluingfibreprod}
\begin{tikzcd}
 \coprod_{\Gamma \in \mathcal{G}_{A,B}} \M_\Gamma \arrow[r,"\xi_{\Gamma \to B}"] \arrow[d,"\xi_{\Gamma \to A}"] & \M_B \arrow[d,"\xi_B"]\\
 \M_A \arrow[r,"\xi_A"] & \M_{g,n,a}
\end{tikzcd}
\end{equation}
of spaces $\M_{\Gamma}$ for the set of isomorphism classes of generic $(A,B)$-structures on prestable graphs $\Gamma$. The top Chern class of the excess bundle
\begin{equation}
    E_\Gamma=\xi_{\Gamma \to A}^* \mathcal{N}_{\xi_A} / \mathcal{N}_{\xi_{\Gamma \to B}}
\end{equation}
is given by
\begin{equation}
    c_{\mathrm{top}}(E_\Gamma) = \prod_{e=(h,h') \in E(A) \cap E(B) \subset E(\Gamma)} - \psi_h - \psi_{h'},
\end{equation}
where the product is over the edges of $\Gamma$ coming both from edges of $A$ and edges of $B$ in the generic $(A,B)$-structure.
\end{proposition}
\begin{proof}
The proof from \cite[Proposition 9]{graberpandh} of the analogous result for the moduli spaces of stable curves goes through verbatim (see also \cite[Section 2]{SchmittvanZelm} for a more detailed version of the argument).
\end{proof}

Using the projection formula, we can then also intersect tautological classes.
\begin{corollary} \label{Cor:tautproduct}
Given decorated stratum classes $[A,\alpha]$, $[B,\beta]$ on $\M_{g,n,a}$, their product is given by
\begin{equation}
    [A,\alpha] \cdot [B,\beta] = \sum_{\Gamma \in \mathcal{G}_{A,B}} (\xi_\Gamma)_* \left(\xi_{\Gamma \to A}^* \alpha \cdot \xi_{\Gamma \to B}^* \beta \cdot c_{\mathrm{top}}(E_\Gamma) \right).
\end{equation}
\end{corollary}

% \subsubsection*{Intersection products}

\subsubsection*{Pushforwards and pullbacks by forgetful maps of points}
In this section, we look at the behaviour of tautological classes under the forgetful map $\pi: \M_{g,n+1,a} \to \M_{g,n,a}$, which is the universal curve over $\M_{g,n,a}$. As such, it is both flat and proper, so we can compute pullbacks as well as pushforwards. We will start with pullbacks.

\begin{proposition} \label{Pro:glueforgetpullback}
Given an $\A$-valued prestable graph $\Gamma$ for $\M_{g,n,a}$, we have a commutative diagram
\begin{equation} \label{eqn:glueforgetpullback}
\begin{tikzcd}
 \coprod_{v \in V(\Gamma)} \M_{\widehat \Gamma_v} \arrow[r,"\xi_{\widehat \Gamma_v}"] \arrow[d,"\pi_v"] & \M_{g,n+1,a} \arrow[d,"\pi"]\\
 \M_\Gamma \arrow[r,"\xi_\Gamma"] & \M_{g,n,a}
\end{tikzcd}
\end{equation}
where the graph $\widehat \Gamma_v$ is obtained from $\Gamma$ by adding the marking $n+1$ at vertex $v$ and the map $\pi_v$ is the identity on the factors of $\M_{\widehat \Gamma_v}$ for vertices $w \neq v$ and the forgetful map of marking $n+1$ at the vertex $v$. The induced map
\begin{equation} \label{eqn:univcurvegluing} \coprod_{v \in V(\Gamma)} \M_{\widehat \Gamma_v} \to \M_\Gamma \times_{\M_{g,n,a}} \M_{g,n+1,a}\end{equation}
%is a local complete intersection morphism and 
satisfies that the fundamental class on the left pushes forward to the fundamental class on the right.
\end{proposition}
\begin{proof}
This follows from the definition of the gluing map $\xi_\Gamma$: giving the map $\xi_\Gamma$ is the same as giving the universal curve over $\M_\Gamma$ and this curve is obtained by gluing the universal curves $\M_{\widehat \Gamma_v}$ over the various factors along the half-edges connected in $\Gamma$. The map \eqref{eqn:univcurvegluing} is obtained by taking, for each edge $\{h_1,h_2\} \in E(\Gamma)$ the loci inside $\M_{\widehat \Gamma_{v_i}}$ where marking $p_{n+1}$ and marking $q_{h_i}$ are on a contracted component and identifying them. Thus, if $p_{n+1}$ is not on a contracted component, the map is an isomorphism in a neighborhood. Therefore the map (\ref{eqn:univcurvegluing}) is an isomorphism at the general point of each component of the right hand side and the fundamental class pushes forwards to the fundamental class. 
\end{proof}
\begin{corollary} \label{Cor:forgetfulpullback}
Given a tautological class $[\Gamma,\alpha]$ write $\alpha = \prod_{v \in V(\Gamma)} \alpha_v$ with $\alpha_v$ the factors of $\alpha$ located at vertex $v$ of $\Gamma$. Then we have
\[\pi^* [\Gamma,\alpha] = \sum_{v \in V(\Gamma)} [\widehat \Gamma_v, (\pi_v^* \alpha_v) \cdot \prod_{w \neq v} \alpha_w].\]
\end{corollary}
\begin{proof}
The class $[\Gamma,\alpha]$ is represented by $\xi_{\Gamma *}\left(\alpha\cap \,[\M_\Gamma]\right)$ where $\alpha$ is an operational Chow class in $\CHOP^*(\M_\Gamma)$. We refer the reader to Appendix \ref{Sect:OperationalChow} for definitions and properties of these operational classes.
By Proposition \ref{Pro:glueforgetpullback} the diagram (\ref{eqn:glueforgetpullback}) together with the map (\ref{eqn:univcurvegluing}) satisfies assumptions in Lemma \ref{Lem:OperationalChowcompatible}. Therefore the equality follows from Lemma \ref{Lem:OperationalChowcompatible}.
\end{proof}
%\jocomment{Here we are a bit sloppy: we use the compatibility of proper pushforward and flat pullback (which we have not proved) and we use the projection formula for the map \eqref{eqn:univcurvegluing}, which is proper but not flat (is it lci? the target is not normal). We use something similar in the proof of Proposition \ref{Pro:psikappaforgetpullback} below. I think maybe we should state these compatibilities either as part of Theorem \ref{Thm:properpushforward} or in a line below it.}
%\jocomment{Mmh, I guess that one solution could be to argue that decorated stratum classes actually live in operational Chow. But this would force us to be really careful everywhere.
%I wonder if the following statement is true: let 
%\[U \xrightarrow{p} W \xrightarrow{\pi'} X\]
%with $p$ proper and birational, $\pi'$ flat such that the composition $s=\pi' \circ p$ is also flat (and we can assume that $U,X$ are smooth, if we want). Then we have
%\[(\pi')^* \alpha = p_* s^* \alpha \text{ for }\alpha \in \CH_*(X).\]
%I think if we have this, then we can make all the various proofs in this section work (notation is taken from your diagram above, and I think a small computation shows that combining this with compatibility of proper pushforward and flat pullback in fibre diagrams gives us the statement we want, for $\alpha \in \CH_*(X)$) without trying to use a projection formula for the possibly-bad map $p$.}
The above corollary shows that to finish our understanding of pullbacks of tautological classes, it suffices to understand how $\kappa$ and $\psi$-classes pull back. 
\begin{proposition}  \label{Pro:psikappaforgetpullback}
For the universal curve morphism $\pi: \M_{g,n+1,a} \to \M_{g,n,a}$ we have
\begin{align}
    \pi^* \psi_i &= \psi_i - D_{i,n+1}\,,\\
    \pi^* \kappa_a &= \kappa_a - \psi_{n+1}^a\,,
\end{align}
where $D_{i,n+1} \subset \M_{g,n+1,a}$ is the image of the section $\sigma_i$ of $\pi$ corresponding to the $i$-th marked point. It can be seen as the tautological class corresponding to the (undecorated) graph (\ref{eqn:sectiongraph}) above.
\end{proposition}
\begin{proof}
The statement is a generalization of the classical pullback formulas for $\Mbar_{g,n}$ (which are the case $\A=\{\zero\}$). 
A convenient way to prove it is to use that 
\begin{equation} \label{eqn:psiboundarysquarepushforward}
\psi_i = -\pi_*(D_{i,n+1}^2)\,.  
\end{equation}
{To show this, we note that $\sigma_i$ can be identified with the gluing map
\[\sigma_i\colon \M_{g,n,a}\times \Mbar_{0,\{\bullet, i,n+1\},\zero}\to \M_{g,n+1,a}\,,\]
where we glue the $i$-th marking on $\M_{g,n,a}$ with the marking $\bullet$ on $\Mbar_{0,\{\bullet, i,n+1\},\zero}$.
Then indeed the locus $D_{i,n+1}$ is the image of the above gluing map (similar to the usual case of stable maps)} and equation \eqref{eqn:psiboundarysquarepushforward} follows from Corollary \ref{Cor:tautproduct}. On the other hand, it can also be seen directly from the fact that $\sigma_i$ is a closed embedding with normal bundle $\sigma_i^*(\omega_\pi^\vee)$.

Now we have a commutative diagram
\begin{equation} \label{eqn:forgetforgetpullback}
\begin{tikzcd}
 \M_{g,n+2,a} \arrow[r,"\pi_{n+1}"] \arrow[d,"\pi_{n+2}"] & \M_{g,n+1,a} \arrow[d,"\pi"]\\
 \M_{g,n+1,a} \arrow[r,"\pi"] & \M_{g,n,a}
\end{tikzcd}
\end{equation}
and the space in the upper left maps birationally to the fibre product of the two forgetful maps. Then 
\begin{align*}
\pi^* \psi_i &= - \pi^* \pi_* D_{i,n+1}^2 = - (\pi_{n+1})_* (\pi_{n+2}^* D_{i,n+1}^2)  \\
&=-(\pi_{n+1})_*\left(
\begin{tikzpicture}[scale=0.7, baseline=-3pt,label distance=0.3cm,thick,
    virtnode/.style={circle,draw,scale=0.5}, 
    nonvirt node/.style={circle,draw,fill=black,scale=0.5} ]
    % \node [virtnode] (A) {} node[below]{$A_1$}; works
    \node [nonvirt node,label=below:$(g{,}a)$] (A) {};
    \node at (2,0) [virtnode,label=below:$(0{,}\zero)$] (B) {};
    \draw [-] (A) to (B);
    \node at (-.7,.5) (n1) {};
    \node at (-1.2,.5) (y1) {$n+2$};
    % \node at (-.86,0) (n0) {$I_1$};
    % \node at (-.7,-.5) (n2) {$n+1$};
    \draw [-] (A) to (n1);
    % \draw [-] (A) to (n2);
    \node at (3,.5) (m1) {};
    \node at (3,-.5) (m2) {};
    \node at (3.2,.5) (x1) {$i$};
    \node at (3.5,-.5) (x2) {$n+1$};
    \draw [-] (B) to (m1);
    \draw [-] (B) to (m2);    
    \end{tikzpicture}
+
\begin{tikzpicture}[scale=0.7, baseline=-3pt,label distance=0.3cm,thick,
    virtnode/.style={circle,draw,scale=0.5}, 
    nonvirt node/.style={circle,draw,fill=black,scale=0.5} ]
    % \node [virtnode] (A) {} node[below]{$A_1$}; works
    \node [nonvirt node,label=below:$(g{,}a)$] (A) {};
    \node at (2,0) [virtnode,label=below:$(0{,}\zero)$] (B) {};
    \draw [-] (A) to (B);
    %\node at (-.7,.5) (n1) {};
    % \node at (-.86,0) (n0) {$I_1$};
    % \node at (-.7,-.5) (n2) {$n+1$};
    %\draw [-] (A) to (n1);
    %
    \node at (3,.5) (m1) {};
    \node at (3,0) (m3) {};
    \node at (3,-.5) (m2) {};
    \node at (3.2,.5) (p1) {$i$};
    \node at (3.5,0) (p2) {$n+1$};
    \node at (3.5,-.5) (p3) {$n+2$};
    \draw [-] (B) to (m1);
    \draw [-] (B) to (m2);    
    \draw [-] (B) to (m3);
    \end{tikzpicture}
\right)^2\\
&= \psi_i -2D_{i,n+1}+D_{i,n+1}=\psi_i-D_{i,n+1}.
\end{align*}
Similarly, using the same diagram, the definition of $\kappa_a$ and the pullback formula for $\psi$, one concludes the pullback formula for $\kappa_a$. 
\end{proof}
% \textbf{Note}: we should write the proof really carefully. One might be tempted to claim for example that \[\psi_i = \pi_*(- [(g=0;i,n+1,\zero) - (g;\{1,\ldots,n\}\setminus \{i\}; a)]^2)\]
% but note that one must be careful: for $\A=\{\zero,\one\}$ we have
% \[[(g=0;1,2,\zero) - (0;\emptyset; \one)]^2 = \text{expected stuff} + [(g=0;1,2,\zero) - (0;\emptyset; \one) - (0;\emptyset; \one)].\]
% I think this last term is killed by the pushforward, but one needs to be careful.

We now turn to the question how to push forward tautological classes $[\Gamma, \alpha] \in \R^*(\M_{g,n+1,a})$ under the map $\pi$.
\begin{proposition} \label{Pro:tautforgetpushforward}
Let $[\Gamma, \alpha] \in \R^*(\M_{g,n+1,a})$  with $\alpha= \prod_{v \in V(\Gamma)} \alpha_v$. Let $v \in V(\Gamma)$  be the vertex incident to $n+1$ and let $\Gamma'$ be the graph obtained from $\Gamma$ by forgetting the marking $n+1$ and stabilizing if the vertex $v$ becomes unstable. There are two cases:
\begin{itemize}
    \item if the vertex $v$ remains stable, then 
    \[\pi_*  [\Gamma,\alpha] = (\xi_{\Gamma'})_*\left((\pi_v)_* \alpha_v \cdot \prod_{w \neq v} \alpha_w \right), \]
    where $\pi_v$ is the forgetful map of marking $n+1$ of vertex $v$.
    \item if the vertex $v$ becomes unstable, then $g(v)=0$, $n(v)=3$ and $a(v)=\zero$. If $\alpha_v \neq 1$ then $[\Gamma, \alpha]=0$. Otherwise, we have 
    \[\pi_*  [\Gamma,\alpha] = [\Gamma', \prod_{w \neq v} \alpha_w]\,.\]
\end{itemize}
\end{proposition}
\begin{proof}
The result follows from the fact that the composition of the gluing map $\xi_\Gamma$ and the forgetful map $\pi$ factors through the gluing map $\xi_{\Gamma'}$ downstairs. In the second part, we use that $\M_{0,3,\zero}=\Mbar_{0,3}=\operatorname{Spec} k$, so any nontrivial decoration by $\kappa$ and $\psi$-classes on this space vanishes.
\end{proof}
The proposition allows us to reduce to computing forgetful pushforwards of products of $\kappa$ and $\psi$-classes. As in the case of $\Mbar_{g,n}$, these can be computed using the projection formula. Indeed, given a product
\[\alpha = \prod_{a} \kappa_a^{e_a} \cdot \prod_{i=1}^n \psi_i^{\ell_i} \cdot \psi_{n+1}^{\ell_{n+1}} \in \R^*(\M_{g,n+1,a})\]
we can use Proposition \ref{Pro:psikappaforgetpullback} and the known intersection formulas on $\M_{g,n+1,a}$ to write it as
\[\alpha = \pi^* \left( \prod_{a} \kappa_a^{e_a} \cdot \prod_{i=1}^n \psi_i^{\ell_i}\right)\cdot \psi_{n+1}^{\ell_{n+1}} + \text{boundary terms}. \]
Using the projection formula we conclude
\[\pi_*(\alpha) = \left( \prod_{a} \kappa_a^{e_a} \cdot \prod_{i=1}^n \psi_i^{\ell_i}\right)\cdot \kappa_{\ell_{n+1}-1} + \pi_*(\text{boundary terms}), \]
where $\kappa_0=2g-2+n$ and $\kappa_{-1}=0$.
The boundary terms are handled by induction on the degree together with Proposition \ref{Pro:tautforgetpushforward}. 

Together with the previous results of this section, this shows that the $\QQ$-linear span of the decorated strata classes $[\Gamma, \alpha]$ in $\CH^*(\M_{g,n,a})$ is closed under intersections as well as pushforwards under gluing and forgetful maps. Thus, by definition, it equals the tautological ring of $\M_{g,n,a}$, so that we finished the proof of Theorem \ref{Thm:tautringintro}.
%In particular, the space of tautological classes is closed under products, so $\R^*(\M_{g,n,a})$ is a countably generated $\mathbb{Q}$-subalgebra of $\CH^*(\M_{g,n,a})$, finite-dimensional in each degree.

\subsubsection*{Pullbacks by forgetful maps of $\A$-values}
\begin{proposition} \label{Pro:forgetfulAvaluespullback}
For the map $F_\A: \M_{g,n,a} \to \M_{g,n}$ forgetting the $\A$-values on all components of the curve, without stabilizing, we have
\[F_\A^* [\Gamma,\alpha] = \sum_{\substack{(a_v)_{v \in V(\Gamma)}\\\sum_v a_v=a}} [\Gamma_{(a_v)_v}, \alpha]\,,\]
where the sum is over tuples $(a_v)_v$ of elements of $\A$ summing to $a$, such that the $\A$-valuation $v \mapsto a_v$ on the vertices $v$ of $\Gamma$ gives a well-defined $\A$-valued graph $\Gamma_{(a_v)_v}$.
\end{proposition}
\begin{proof}
The fibre product of $F_\A$ and a gluing map $\xi_{\Gamma}$ is the disjoint union of the gluing maps $\xi_{\Gamma_{(a_v)_v}}$. A short computation shows that $F_\A^* \psi_i = \psi_i$ and $F_\A^* \kappa_a = \kappa_a$.
\end{proof}

\subsubsection*{Pullback by stabilization map}
We saw before that the stabilization morphism $\st: \M_{g,n} \to \Mbar_{g,n}$ is flat, so we can ask how to pull back tautological classes along this morphism. We start by computing the pullback of gluing maps under $\st$.

\begin{proposition} \label{Pro:stgluepullback}
Given a stable graph $\Gamma$ in genus $g$ with $n$ marked points (for $2g-2+n>0$), we have a commutative diagram
\begin{equation} \label{eqn:stgluepullback}
\begin{tikzcd}
 \M_\Gamma \arrow[r,"\prod_{v} \st_v"] \arrow[d,"\xi_\Gamma"] & \Mbar_\Gamma \arrow[d, "\xi_\Gamma"]\\
 \M_{g,n} \arrow[r, "\st"] &\Mbar_{g,n}
\end{tikzcd}    
\end{equation}
where $\st_v : \M_{g(v),n(v)} \to \Mbar_{g(v),n(v)}$ is the stabilization morphism at vertex $v$.
Moreover, the induced map
\begin{equation} \label{eqn:stfibreprodmap}
    \M_\Gamma \to \M_{g,n} \times_{\Mbar_{g,n}} \Mbar_\Gamma
\end{equation}
is proper and birational.
In particular
\begin{equation} \label{eqn:stdecstratumpullback}
    \st^* \left[\Gamma, \prod_v \alpha_v\right] = (\xi_\Gamma)_* \left( \prod_v {\st_v^*} \alpha_v \right).
\end{equation}
\end{proposition}
\begin{proof}
The commutativity of the diagram \eqref{eqn:stgluepullback} follows from the definition of the stabilization. The map \eqref{eqn:stfibreprodmap} is easily seen to be birational and its properness follows from the diagram
\begin{equation*}
\begin{tikzcd}
 \M_\Gamma \arrow[rr] \arrow[rd,"\xi_\Gamma", swap] && \M_{g,n} \times_{\Mbar_{g,n}} \Mbar_\Gamma\arrow[dl, "\mathrm{pr}_1"]\\
 & \M_{g,n} &
\end{tikzcd}
\end{equation*}
and the cancellation property of proper morphisms (in the diagram, the map $\xi_\Gamma$ and $\mathrm{pr}_1$ are proper). Equation \eqref{eqn:stdecstratumpullback} again follows by an application of Lemma \ref{Lem:OperationalChowcompatible}.
\end{proof}
The proposition above reduces the pullback of tautological classes under $\st$ to computing the pullback of $\kappa$ and $\psi$-classes.

\begin{proposition} \label{Pro:stpsipullback}
Let $g,n$ with $2g-2+n>0$ and let $\A=\{\zero, \one\}$. Then for $1 \leq i \leq n$ and the stabilization map $\st_\A: \M_{g,n,\one} \to \Mbar_{g,n}$ we have
\begin{equation} \label{eqn:stpsipullback}
    \st_\A^* \psi_i = \psi_i -  
    \begin{tikzpicture}[scale=0.7, baseline=-3pt,label distance=0.3cm,thick,
    virtnode/.style={circle,draw,scale=0.5}, 
    nonvirt node/.style={circle,draw,fill=black,scale=0.5} ]
    % \node [virtnode] (A) {} node[below]{$A_1$}; works
    \node [nonvirt node,label=below:$(0{,}\one)$] (A) {};
    \node at (2,0) [nonvirt node,label=below:$(g{,}\one)$] (B) {};
    \draw [-] (A) to (B);
    \node at (-.7,.5) (n1) {$i$};
    % \node at (-.86,0) (n0) {$I_1$};
    % \node at (-.7,-.5) (n2) {$n+1$};
    \draw [-] (A) to (n1);
    % \draw [-] (A) to (n2);
    \node at (2.7,.5) (m1) {};
    % \node at (3.86,0) (m0) {$\{1, \ldots, n\} \setminus \{i\}$};
    \node at (2.7,-.5) (m2) {};
    \draw [-] (B) to (m1);
    \draw [-] (B) to (m2);    
    \end{tikzpicture}
    + 
    \begin{tikzpicture}[scale=0.7, baseline=-3pt,label distance=0.3cm,thick,
    virtnode/.style={circle,draw,scale=0.5}, 
    nonvirt node/.style={circle,draw,fill=black,scale=0.5} ]
    % \node [virtnode] (A) {} node[below]{$A_1$}; works
    \node [nonvirt node,label=below:$(0{,}\one)$] (A) {};
    \node at (2,0) [virtnode,label=below:$(g{,}\zero)$] (B) {};
    \draw [-] (A) to (B);
    \node at (-.7,.5) (n1) {$i$};
    % \node at (-.86,0) (n0) {$I_1$};
    % \node at (-.7,-.5) (n2) {$n+1$};
    \draw [-] (A) to (n1);
    % \draw [-] (A) to (n2);
    \node at (2.7,.5) (m1) {};
    % \node at (3.86,0) (m0) {$\{1, \ldots, n\} \setminus \{i\}$};
    \node at (2.7,-.5) (m2) {};
    \draw [-] (B) to (m1);
    \draw [-] (B) to (m2);    
    \end{tikzpicture} \,.
\end{equation}
\end{proposition}
\begin{proof}
Consider the commutative diagram
\begin{equation} \label{eqn:stforgetdiagr}
\begin{tikzcd}
 \M_{g,n+1,\one} \arrow[r,"c"] \arrow[rd,"\pi_\A"] & \M_{g,n,\one} \times_{\Mbar_{g,n}} \Mbar_{g,n+1} \arrow[d] \arrow[r,"\widehat{\st}_\A"] & \Mbar_{g,n+1}\arrow[d,"\pi"]\\
 & \M_{g,n,\one} \arrow[r,"\st_\A"]&\Mbar_{g,n}
\end{tikzcd}
\end{equation}
where the right square is cartesian and the map $c$ is the map  contracting the unstable components of the universal curve $\M_{g,n+1,\one} \to \M_{g,n,\one}$. By the cancellation property of proper morphisms, the map $c$ is proper and easily seen to be birational. 

For computing the pullback of $\psi_i$ under $\st_\A$, we use that $\psi_i = -\pi_*( D_{i,n+1}^2)$ on $\Mbar_{g,n}$, where $D_{i,n+1} \subset \Mbar_{g,n+1}$ is the image of the $i$-th section. By Lemma \ref{Lem:OperationalChowcompatible}, we have
\[\st_\A^* \psi_i = - \st_\A^* \pi_*( D_{i,n+1}^2) = (\pi_\A)_* \left((c \circ \widehat{\st}_\A)^* D_{i,n+1} \right)^2.\]
The composition $c \circ \widehat{\st}_\A$ is just the usual stabilization map and the pullback of $D_{i,n+1}$ under this map is the sum of three boundary divisors of $\M_{g,n+1,\one}$ : their underlying graph is the same as for $D_{i,n+1}$ and the $\A$-values correspond to the three different ways $\zero + \one = \one + \zero = \one + \one$ to distribute the value $\one$ to the two vertices. A short computation using the rules for intersection and pushforward presented earlier gives the formula (\ref{eqn:stpsipullback}).  
% \jocomment{I put a scan of the computation at \url{https://drive.google.com/open?id=120xvZ68lwzkKePElVK4912Ku1fYniEyd}. But maybe you can do it yourself as a double check?}\Ycomment{Thanks! I checked your computation.}
% Thanks!
\end{proof}

The formula for the pullback of $\kappa$-classes is more involved and we need to introduce a bit of notation to state it. Fix $g,n$ with $2g-2+n>0$, then for $k \geq 0$ let $\widehat G_k, G_k$ be the following $(n+1)$ and $n$-pointed prestable graphs in genus $g$ with $k$ edges
\begin{align*} %\label{eqn:graphGk}
\widehat G_k &=& \begin{tikzpicture}[baseline=-3pt,label distance=0.2cm,thick,
    virtnode/.style={circle,draw,scale=0.5}, 
    nonvirt node/.style={circle,draw,fill=black,scale=0.5}, ampersand replacement=\& ]
    % \node [virtnode] (A) {} node[below]{$A_1$}; works
    \node [nonvirt node,label=below:$0$,label=above:$v_0$] (A) {};
    \node at (2,0) [nonvirt node,label=below:$0$] (B) {};
    \node at (4,0) [nonvirt node,label=below:$g$] (C) {};
    \draw [-] (A) to node[above, near start]{$h_0$} (B);
    \draw [-] (B) to (2.5,0) node[right]{$\cdots$};
    \draw [-] (3.5,0) to (C);
    %\node at (-.7,.5) (n1) {$i$};
    \node at (-1.36,0) (n0) {$n+1$};
    %\node at (-.7,-.5) (n2) {$n+1$};
    \draw [-] (A) to (n0);
    %\draw [-] (A) to (n2);
    \node at (4.7,.5) (m1) {};
    \node at (5.9,0) (m0) {$\{1, \ldots, n\}$};
    \node at (4.7,-.5) (m2) {};
    \draw [-] (C) to (m1);
    \draw [-] (C) to (m2);   
    \draw [-] (C) to (m0);   
    \end{tikzpicture}\,,\\
    G_k &=&\begin{tikzpicture}[baseline=-3pt,label distance=0.2cm,thick,
    virtnode/.style={circle,draw,scale=0.5}, 
    nonvirt node/.style={circle,draw,fill=black,scale=0.5}, ampersand replacement=\& ]
    % \node [virtnode] (A) {} node[below]{$A_1$}; works
    \node [nonvirt node,label=below:$0$,label=above:$v_0$] (A) {};
    \node at (2,0) [nonvirt node,label=below:$0$] (B) {};
    \node at (4,0) [nonvirt node,label=below:$g$] (C) {};
    \draw [-] (A) to node[above, near start]{$h_0$} (B);
    \draw [-] (B) to (2.5,0) node[right]{$\cdots$};
    \draw [-] (3.5,0) to (C);
    %\node at (-.7,.5) (n1) {$i$};
    % \node at (-1.36,0) (n0) {$n+1$};
    %\node at (-.7,-.5) (n2) {$n+1$};
    % \draw [-] (A) to (n0);
    %\draw [-] (A) to (n2);
    \node at (4.7,.5) (m1) {};
    \node at (5.9,0) (m0) {$\{1, \ldots, n\}$};
    \node at (4.7,-.5) (m2) {};
    \draw [-] (C) to (m1);
    \draw [-] (C) to (m2);   
    \draw [-] (C) to (m0);   
    \end{tikzpicture}\,.
\end{align*}
Here $v_0$ is the leftmost vertex and, for $k \geq 1$, $h_0$ is the unique half-edge incident to this vertex. For $k=0$, the graphs $\widehat{G}_k, G_k$ are the trivial graphs, respectively.
% Let $\mathcal{G} = \{G_k: k \geq 0\}$ and $\mathcal{G}^* = \mathcal{G} \setminus \{G_0\}$.

Also, in the proposition below we consider the power series 
\[\Phi(t)=\frac{\exp(t)-1}{t} = 1 + \frac{t}{2} + \frac{t^2}{6}+ \ldots .\]
We use the notation $[\Phi(t)]_{t^a \mapsto \kappa_a}$ to indicate that in the power series $\Phi$ the term $t^a$ is substituted with the class $\kappa_a$, getting the mixed-degree Chow class
\[[\Phi(t)]_{t^a \mapsto \kappa_a}= 1 + \frac{\kappa_1}{2} + \frac{\kappa_2}{6}+ \ldots .\]
\begin{proposition} \label{Pro:stkappapullback}
For $g,n$ with $2g-2+n>0$
% , then for the power series
% \[\Phi(t)=\frac{\exp(t)-1}{t} = 1 + \frac{t}{2} + \frac{t^2}{6}+ \ldots\] 
the stabilization morphism $\st: \M_{g,n} \to \Mbar_{g,n}$ satisfies the following equality of mixed-degree Chow classes on $\M_{g,n}$:
\begin{align} 
&\st^*\left[\Phi(t)\right]_{t^a \mapsto \kappa_a} \nonumber \\=& \left[\Phi(t)\right]_{t^a \mapsto \kappa_a} + \sum_{k \geq 1} (\xi_{G_k})_* \left( \left(\left[\Phi(t) \right]_{t^a \mapsto \kappa_{v_0,a}} + \psi_{h_0}^{-1} \right) \cdot \mathrm{Cont}_{E(G_k)} \right). \label{eqn:stkappapullback}
\end{align}
Here $\mathrm{Cont}_{E(G_k)}$ is the mixed-degree class
\[\mathrm{Cont}_{E(G_k)} =  \prod_{(h,h') \in E(G_k)} - \Phi(\psi_h + \psi_{h'}) \] %\frac{\exp(-(\psi_h + \psi_{h'}))-1}{\psi_h + \psi_{h'}}\]
on $\M_{G_k}$. In the formula above, the term $\psi_{h_0}^{-1}$ is understood to vanish unless it pairs with a term of $\mathrm{Cont}_{E(G_k)}$ containing a positive power of $\psi_{h_0}$ and we have $\kappa_{v_0,0}=2\cdot 0 -2+1=-1$.
\end{proposition}
To obtain the pullback of an individual class $\kappa_a$ under $\st$ we take the degree $a$ part of \eqref{eqn:stkappapullback} and obtain a formula of the form
\[\st^* \kappa_a = \kappa_a + \text{ boundary corrections}.\]
As an example, for $a=1,2$ we obtain
\begin{equation*}
    \st^*\kappa_1=\kappa_1+[G_1]  
\end{equation*}
and
\begin{equation*}
    \st^*\kappa_2=\kappa_2 -3[G_1,\kappa_{v_0,1}]+2[G_1,\psi_{h_0}]+[G_1,\psi_{h_1}] -3[G_2]
\end{equation*}
where $e=(h_0,h_1)$ is the unique edge of the graph $G_1$.
\begin{proof}[Proof of Proposition \ref{Pro:stkappapullback}]
Consider the following commutative diagram
\tikzset{
  symbol/.style={
    draw=none,
    every to/.append style={
      edge node={node [sloped, allow upside down, auto=false]{$#1$}}}
  }
}
\begin{equation} \label{eqn:stforgetdiagr2}
\begin{tikzcd}
 \M_{g,n+1,\one} \arrow[r,symbol=\supset] \arrow[d] & \C_{g,n} \arrow[d,"\pi'"] \arrow[r,"\widehat{\st}"] & \Mbar_{g,n+1}\arrow[d,"\pi"]\\
 \M_{g,n,\one} \arrow[r,symbol=\supset] & \M_{g,n} \arrow[r,"\st"]&\Mbar_{g,n}\,.
\end{tikzcd}
\end{equation}
Then as $\C_{g,n}$ maps proper and birationally to the fibre product in the right diagram, we have
\begin{equation} \label{eqn:kappacomp}
\st^* \kappa_a = \st^* \pi_* \psi_{n+1}^{a+1} =  \pi'_* \widehat{\st}^* \psi_{n+1}^{a+1} = \pi'_* \left(\psi_{n+1} - [\widehat{G}_1] \right)^{a+1}.
\end{equation}
Here we use that computations in $\C_{g,n}$ can be performed in $\M_{g,n+1,\one}$ together with the pullback formula from Proposition \ref{Pro:stpsipullback} (noting that the third term in (\ref{eqn:stpsipullback}) vanishes since it lies in the complement of the open substack $\C_{g,n} \subset \M_{g,n+1,\one}).$

From now on it will be more convenient working with mixed-degree classes and exponentials. In this language, equation (\ref{eqn:kappacomp}) translates to
\begin{equation} \label{eqn:kappacomp2}
    \st^*[\Phi(t)]_{t^a \mapsto \kappa_a} = \pi'_* \exp(\psi_{n+1} - [\widehat{G}_1]) = \pi'_* \left(\exp(\psi_{n+1}) \cdot \exp(- [\widehat{G}_1]) \right).
\end{equation}
The occurrence of the power series $\Phi$ is due to the discrepancy between the degree $a$ of $\kappa_a$ on the left of (\ref{eqn:kappacomp}) and the degree $a+1$ of the term on the right. Using the rules for intersections of tautological classes, one shows
\begin{equation}
    \exp(- [\widehat{G}_1]) = \sum_{k\geq 1} (\xi_{\widehat{G}_k})_* \left(  \prod_{(h,h') \in E(\widehat{G}_k)} - \Phi(\psi_h + \psi_{h'}) \right).
\end{equation}
Now in the pushforward (\ref{eqn:kappacomp2}) the only terms supported on the trivial graph are those from \[\pi'_*(\exp(\psi_{n+1}))=[\Phi(t)]_{t^a \mapsto \kappa_a}\,,\]
explaining the first term of the answer. All other terms of the product of the exponentials are supported on some $\widehat{G}_k$ for $k \geq 1$, where marking $n+1$ is on a rational component with just one other half-edge $h_0$. Using the formulas for the pushforward by the forgetful map $\pi'$ from Proposition \ref{Pro:tautforgetpushforward}, the only nontrivial pushforward we have to compute is the one by the universal curve $\pi_{0,1}: \C_{0,1} \to \M_{0,1}$, corresponding to forgetting $n+1$ on the $2$-marked genus $0$ component $v_0$ of $\widehat{G}_k$. Here, a short computation shows
\begin{equation} \label{eqn:kappacomp3}
(\pi_{0,1})_* \psi_1^a \psi_2^b = \psi_1^a \kappa_{b-1} + \delta_{a,0} \psi_1^{a-1}
\end{equation}
where $\delta_{a,0}$ is the Kronecker delta and we have the convention $\kappa_{-1}=\psi_1^{-1}=0$. Applying this formula for the pushforward, the first term in (\ref{eqn:kappacomp3}) gives rise to the term of the result involving $\left[\Phi(t) \right]_{t^a \mapsto \kappa_{v_0,a}}$, where again $\Phi$ appears due to the shift of degree from $b$ to $b-1$ in (\ref{eqn:kappacomp3}). The second term of (\ref{eqn:kappacomp3}) gives rise to the term involving $\psi_{h_0}^{-1}$, where due to the Kronecker delta $\delta_{a,0}$ only the constant term of $\exp(\psi_{n+1})$ survives in the pushforward.
\end{proof}

\begin{remark}
The following is a nontrivial check and application of the computations from the last sections: for $g,n,m$ with $2g-2+n>0$, consider the diagram
\begin{equation} \label{eqn:forgetstcommdiagr}
\begin{tikzcd}
 \Mbar_{g,n+m} \arrow[d,"F_m"] \arrow[dr,"\pi"] & \\
 \M_{g,n} \arrow[r,"\st"] & \Mbar_{g,n}
\end{tikzcd}
\end{equation}
where $F_m$ is the map forgetting the last $m$ markings (without stabilizing the curve), the map $\st$ is the stabilization map and their composition $\pi$ is the ``usual" forgetful map between moduli spaces of stable curves. The pullback of tautological classes under $\pi$ is known classically and the pullback by the two other maps has been computed in the previous sections. Since the pullbacks must be compatible, this gives rise to tautological relations, which we can verify in examples.

For instance, for the class $\kappa_1 \in \CH^1(\Mbar_{g,n})$ we have
\[\pi^* \kappa_1 = \kappa_1 - \sum_{i=n+1}^{n+m} \psi_i + \sum_{\substack{I \subset \{n+1, \ldots, n+m\}\\|I| \geq 2}} D_{0,I}\]
where $D_{0,I} \subset \Mbar_{g,n+m}$ is the boundary divisor of curves with a rational component carrying markings $I$. On the other hand we have
\begin{align*}
\st^* \kappa_1 &= \kappa_1 + [G_1] = \kappa_1 + D_{0,\emptyset}\\
\pi^* \st^* \kappa_1 &= \kappa_1 - \sum_{i=n+1}^{n+m} \psi_i +  \sum_{\substack{I \subset \{n+1, \ldots, n+m\}\\|I| \geq 2}} D_{0,I}\,.%|_{\Mbar_{g,n+m}}\,.
\end{align*}
So indeed we get the same answer. 
\end{remark}

\section{Relation to previous works}\label{sec:previous}
%Insert Example 3.40, 3.43, 3.44 from old paper
In this section we review several results in the literature relating to our study of the intersection theory of the stacks $\M_{g,n}$.
\begin{example}
In \cite[Lemma 1]{gathmann}, Gathmann used the pullback formula of $\psi$-classes along the stabilization morphism $\st\colon \M_{g,1}\to \Mbar_{g,1}$ to prove certain properties of the Gromov--Witten potential. It coincides with our calculation in Proposition \ref{Pro:stpsipullback}.
\end{example}

\begin{example}
In \cite{Pixtonboundary}, Pixton introduces classes $[\Gamma] \in \R^*(\Mbar_{g,n})$  indexed by \emph{prestable} graphs of genus $g$ with $n$ legs. In his construction, chains of unstable vertices encode insertions of $\kappa$ and $\psi$-classes in such a way that the formula for products $[\Gamma] \cdot [\Gamma']$ takes a particularly simple shape. While it is not a priori obvious how to relate his classes to the corresponding boundary strata classes $[\Gamma] \in \R^*(\M_{g,n})$ in the moduli stack of prestable curves, this is a question which we plan to investigate in future work.
\end{example}

\begin{example} \label{Exa:Oesinghaus}
In \cite{oesinghaus}, Oesinghaus computes the Chow ring (with integral coefficients) of a certain open substack $\TT$ of $\M_{0,3}$, defined by the condition that the curve is \emph{semistable} (i.e. every component of the curve has at least two distinguished points) and that the markings $2,3$ are on a stable component of the curve. As a consequence, the prestable graphs of geometric points of $\TT$ are all of the form
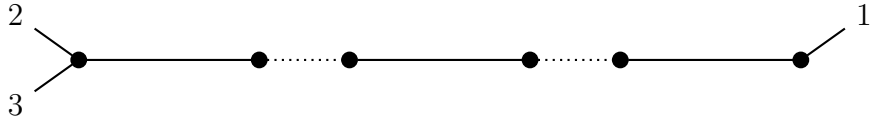
\begin{equation*}
\begin{tikzpicture}[scale=1.2, baseline=-3pt,label distance=0.3cm,thick,
    virtnode/.style={circle,draw,scale=0.5}, 
    nonvirt node/.style={circle,draw,fill=black,scale=0.5} ]
    \node at (-6,0) [nonvirt node] (F) {};
    \node at (-4,0) [nonvirt node] (E) {};
    \node at (-3,0) [nonvirt node] (D) {};
    \node at (-1,0) [nonvirt node] (C) {};
    \node [nonvirt node] (A) {};
    \node at (2,0) [nonvirt node] (B) {};
    \draw [-] (A) to (B);
    \draw [-, dotted] (C) to (A);
    \draw [-] (D) to (C);
    \draw [-, dotted] (E) to (D);
    \draw [-] (F) to (E);
    % \node at (-5.7,.3) {$\psi_{h_1}^{a_1}$};
    % \node at (-4.3,.3) {$\psi_{h'_1}^{a'_1}$};
    % \node at (-2.7,.2) {$\scriptstyle h_{2i}$};
    % \node at (-1.3,.2) {$\scriptstyle h'_{2i}$};
    % \node at (.3,.3) {$\psi_{h_{2l-1}}^{a_{2l-1}}$};
    % \node at (1.6,.3) {$\psi_{h'_{2l-1}}^{a'_{2l-1}}$};
    %\node at (-.7,.5) (n1) {$i$};
    %\draw [-] (A) to (n1);
    %\node at (0,.7) (n2) {$\downarrow$};
    \node at (2.7,.5) (m1) {$1$};
    \draw [-] (B) to (m1);
    \node at (-6.7,.5) (n1) {$2$};
    \draw [-] (F) to (n1);
    \node at (-6.7,-.5) (n2) {$3$};
    \draw [-] (F) to (n2);    
    \end{tikzpicture}
\end{equation*}
where we denote by $\Gamma_k$ the graph of the shape above with $k$ edges (for $k \geq 0$). The stack $\TT$ has an atlas given by 
$$\pi_n : \mathcal{A}^n = [\AA^n/\mathbb{G}_m^n] \to \TT\text{ for }n \geq 1.$$
Since $\mathcal{A}^n$ is a vector bundle over $B \mathbb{G}_m^n$, its Chow group\footnote{For the comparison with Oesinghaus' results we formulate everything in terms of $\mathbb{Q}$-coefficients since this is the convention of the present paper.} is given by
\[\CH^*(\mathcal{A}^n) = \mathbb{Q}[\alpha_1, \ldots, \alpha_n],\]
where $\alpha_\ell$ is the class of the $\ell$-th coordinate hyperplane 
\[\iota_\ell : [V(x_i) / \mathbb{G}_m^n] \hookrightarrow \mathcal{A}^n.\]
From a computation in \cite[Lemma 1]{oesinghaus}  it follows that the first Chern class of the normal bundle of $\iota_\ell$ is given by the restriction
\[c_1( \mathcal{N}_{\iota_\ell}) = \iota_\ell^* \alpha_\ell \]
of $\alpha_\ell$ to this hyperplane. Using the charts $\pi_n$, Oesinghaus shows that the Chow ring $\CH^*(\TT)$ is given by the ring QSym of {\em quasi-symmetric functions} on the index set $\mathbb{Z}_{>0}$. QSym can be seen as the subring of $\mathbb{Q}[\alpha_1, \alpha_2, \ldots]$ with additive basis given by
\begin{equation}
    M_{J} = \sum_{i_1 < \ldots < i_k} \alpha_{i_1}^{j_1} \cdots \alpha_{i_k}^{j_k} \text{ for }k \geq 1, J = (j_1, \ldots, j_k) \in \mathbb{Z}_{\geq 1}^{k}.
\end{equation}
Under the isomorphism $\CH^*(\TT) \cong \text{QSym}$, the element $M_J$ is a basis element of degree $\sum_\ell j_\ell$ in the Chow group of $\TT$. The pullback
\[\pi_n^* : \CH^*(\TT) \to \CH^*(\mathcal{A}^n) = \mathbb{Q}[\alpha_1, \ldots, \alpha_n]\]
is induced by the map sending $\alpha_m$ to zero for $m > n$. In particular, it is easy to see that it is injective in Chow-degree at most $n$.

With these preparations in place, we can now identify the generators $M_J$ of $\CH^*(\TT)$ with tautological classes. Indeed, we claim that $M_J$ corresponds to the class supported on the stratum $\M^{\Gamma_k}$ given by
\begin{equation} \label{eqn:oesinghauscomparisonpicture}
\begin{tikzpicture}[scale=1.2, baseline=-3pt,label distance=0.3cm,thick,
    virtnode/.style={circle,draw,scale=0.5}, 
    nonvirt node/.style={circle,draw,fill=black,scale=0.5} ]
    \node at (-6,0) [nonvirt node] (F) {};
    \node at (-4,0) [nonvirt node] (E) {};
    \node at (-3,0) [nonvirt node] (D) {};
    \node at (-1,0) [nonvirt node] (C) {};
    \node [nonvirt node] (A) {};
    \node at (2,0) [nonvirt node] (B) {};
    \draw [-] (A) to (B);
    \draw [-, dotted] (C) to (A);
    \draw [-] (D) to (C);
    \draw [-, dotted] (E) to (D);
    \draw [-] (F) to (E);
    \node at (-5,.3) {$(-\psi - \psi')^{j_1-1}$};
    %\node at (-4.3,.3) {$\psi_{h'_1}^{a'_1}$};
    \node at (-2,.3) {$(-\psi - \psi')^{j_\ell-1}$};
    %\node at (-1.3,.2) {$\scriptstyle h'_{2i}$};
    \node at (1,.3) {$(-\psi - \psi')^{j_k-1}$};
    %\node at (1.6,.3) {$\psi_{h'_{2l-1}}^{a'_{2l-1}}$};
    %\node at (-.7,.5) (n1) {$i$};
    %\draw [-] (A) to (n1);
    %\node at (0,.7) (n2) {$\downarrow$};
    \node at (2.7,.5) (m1) {$1$};
    \draw [-] (B) to (m1);
    \node at (-6.7,.5) (n1) {$2$};
    \draw [-] (F) to (n1);
    \node at (-6.7,-.5) (n2) {$3$};
    \draw [-] (F) to (n2);    
    \end{tikzpicture}
\end{equation}
To see this, we note that from the definition of the charts $\pi_n$ in \cite[Section 3.3]{oesinghaus} one can show that we have a fibre diagram
\begin{equation} \label{eqn:oesinghauscomparison}
\begin{tikzcd}
 \bigsqcup_{1 \leq i_1 < \cdots < i_k \leq n} V(x_{i_1}, \ldots, x_{i_k}) \arrow[r] \arrow[dd] & \mathbb{A}^n \arrow[d]\\
 & \left[\mathbb{A}^n/\mathbb{G}_m^n\right] \arrow[d,"\pi_n"]\\
 \M_{\Gamma_k} \arrow[r,"\xi_{\Gamma_k}"] &\M_{0,3}
\end{tikzcd}
\end{equation}
As a first example, this implies that $[\Gamma_k,1] = (\xi_{\Gamma_k})_* [\M_{\Gamma_k}]$ corresponds to the class
\[\sum_{1 \leq i_1 < \cdots < i_k \leq n} {\alpha_{i_1}\cdots\alpha_{i_k}}\]
on $\mathcal{A}^n$, which indeed is equal to $\pi_n^*(M_{(1, \ldots, 1)})$. For the comparison of more complicated classes, we observe that the decorations $(-\psi - \psi')^{j_\ell-1}$ are the $(j_\ell-1)$-st powers of the Chern class of the normal bundle associated to the $\ell$-th edge of $\Gamma_k$. On the other hand, in the diagram \eqref{eqn:oesinghauscomparison} the function $x_{i_\ell}$ around the linear subspace $V(x_{i_1}, \ldots, x_{i_k})$ is the smoothing parameter for the $\ell$-th node of the curve and the first Chern class of the normal bundle to the locus $V(x_{i_\ell})$ where the $\ell$-th node persists is given by (the restriction of) $\alpha_{i_\ell}$. Since $V(x_{i_1}, \ldots, x_{i_k})$ has class $\alpha_{i_1} \cdots \alpha_{i_k}$ in $\mathcal{A}_n$, we conclude that the element \eqref{eqn:oesinghauscomparisonpicture} of $\CH^*(\TT)$ pulls back via $\pi_n$ to
\[\sum_{1 \leq i_1 < \cdots < i_k \leq n} \alpha_{i_1} \cdots \alpha_{i_k} \cdot \alpha_{i_1}^{j_1-1} \cdots \alpha_{i_k}^{j_k-1} = M_J |_{\mathcal{A}_n}\,.\]
This shows the desired correspondence because this holds for all $n$ and $\pi_n^*$ is injective in degree at most $n$.

Using the correspondence, it is straightforward to see that the product formula for expressing $M_J \cdot M_{J'}$ in terms of basis elements $M_{J_i}$ discussed in \cite[Proposition 2]{oesinghaus} follows from the product formula for decorated strata classes discussed in Section \ref{calculus}. 

In \cite{oesinghaus}, Oesinghaus also computes the Chow group of the semistable loci $\M_{0,2}^{ss}$ and $\M_{0,3}^{ss}$ and there are correspondence results to the tautological generators of these spaces closely parallel to the above discussion. We leave the details to the interested reader.
\end{example}

\newpage
\appendix

%\section{(Operational) Chow groups of locally finite type algebraic stacks} \label{sect:appendixChowlft}
\section{Chow groups of locally finite type algebraic stacks} \label{Sect:Chowlocfintype}
The Chow group of a finite type algebraic stack over a field $k$ is defined in \cite{kreschartin}. We extend this notion to an algebraic stack which is not {necessarily} of finite type over $k$.
\begin{definition} \label{Def:Chowlocallyft}
Let $\M$ be an algebraic stack, locally of finite type over a field $k$. Choose $(\mathcal{U}_i)_{i \in I}$ a directed system\footnote{Recall that this means that for all $\mathcal U_i, \mathcal U_j$ there exists a $\mathcal U_\ell$ containing both of them.} of finite type open substacks of $\M$ whose union is all of $\M$. Then we define
\begin{equation*}
    \CH_*(\M) = \varprojlim_{i \in I} \CH_*(\mathcal U_i),
\end{equation*}
where for $\mathcal U_i \subseteq \mathcal U_j$ the transition map $\CH_*(\mathcal U_j) \to \CH_*(\mathcal U_i)$ is given by the restriction to $\mathcal U_i$. In other words, we have
\begin{equation*}
    \CH_*(\M) = \{(\alpha_i)_{i \in I} : \alpha_i \in \CH_*(\mathcal U_i), \alpha_j|_{\mathcal U_i} = \alpha_i\text{ for }\mathcal U_i \subseteq \mathcal U_j \}.
\end{equation*}
\end{definition}
For the existence of a system $(\mathcal{U}_i)_{i \in I}$ as above, observe that since $\M$ is locally of finite type, we can simply take the system of \emph{all} finite type substacks $\mathcal U \subset \M$. Moreover, given any two systems $(\mathcal{U}_i)_{i \in I}$, $(\mathcal{U}'_i)_{i \in I'}$, one uses Noetherian induction to show that they mutually dominate each other. By standard arguments, the Chow group $\CH_*(\M)$ is independent of the choice of $(\mathcal{U}_i)_{i \in I}$.

From the definition as a limit one sees that the Chow groups $\CH_*(\M)$ inherit all the usual properties (e.g. flat pullback, projective pushforward, Chern classes of vector bundles and Gysin pullbacks) of the Chow groups from \cite{kreschartin}. Moreover, if $\M$ is smooth and has affine stabilizer groups at geometric points, the intersection products on the groups $\CH_*(\mathcal U_i)$ give rise to an intersection product on $\CH_*(\M)$. In this case, if $\M$ is equidimensional we often use the cohomological degree convention
$$\CH^*(\M) = \CH_{\dim \M -*}(\M)\,.$$
The Chow group of a locally finite type algebraic stack is defined as taking a projective limit. Since taking projective limits is not an exact functor and does not commute with tensor products, some of the properties of Chow groups of finite type algebraic stacks do not (obviously) extend. In the following definition we present two finiteness assumptions on locally finite type stacks, which guarantee that the Chow groups we define continue to have some nice properties (like having an excision sequence).
\begin{definition}
Let $\M$ be an equidimensional algebraic stack, locally finite type over a field $k$.
\begin{enumerate}[label=\alph*)]
    \item We say $\M$ is {\em Lindelöf} if every cover of $\M$ by open substacks has a countable subcover.
    \item We say that $\M$ has a \emph{good filtration by finite type substacks}\footnote{This definition is taken from \cite[Definition 5]{oesinghaus}.} if there exists a collection $(\mathcal{U}_m)_{m \in\mathbb{N}}$ of open substack of finite type on $\M$ which is \emph{increasing} (i.e. $\mathcal{U}_m \subset \mathcal{U}_\ell$ for $m<\ell$) and such that $\textup{dim}(\M \setminus \mathcal{U}_m) < \textup{dim}\, \M -m$.
\end{enumerate}
\end{definition}
\begin{lemma}
A locally finite type algebraic stack $\M$ over $k$ is Lindelöf if and only if it has a countable cover $(\mathcal{U}_i)_{i\in \mathbb{N}}$ by finite type open substacks $\mathcal{U}_i \subseteq \M$. In this case the cover $\mathcal{U}_i$ can be chosen to be increasing. In particular, if $\M$ has a good filtration by finite type substacks it is automatically Lindelöf.
\end{lemma}
\begin{proof}
If $\M$ is Lindelöf, its cover by the system of \emph{all} finite type substacks has a countable subcover. Conversely assume that $(\mathcal{U}_i)_{i \in \mathbb{N}}$ is a countable cover of $\M$ by finite type open substacks. Given any open cover $(\mathcal V_j)_{j \in J}$ of $\M$, each single open $\mathcal{U}_i$ is covered by finitely many elements $\mathcal V_{j_{i,\ell}}$ of the second cover via Noetherian induction. Then the system $(\mathcal V_{j_{i,\ell}})_{i,\ell}$ is a countable subcover.
\end{proof}
\begin{example}
\begin{enumerate}[label=\alph*)]
    \item The stacks $\M_{g,n}$ of prestable curves have a good filtration by finite type substacks, given by the loci $\M_{g,n}^{\leq e}$ of curves having at most $e$ nodes.
    \item The universal Picard stack $\mathfrak{Pic}_{g,n}$ over $\M_{g,n}$ parameterizing tuples $$(C,p_1, \ldots, p_n, \mathcal{L})$$ of a prestable marked curve and a line bundle $\mathcal{L}$ on $C$ is Lindelöf, but does not have a good filtration by finite type substacks. 
    
    Indeed, we do get a countable cover $(\mathcal{U}_m)_{m \in \mathbb{N}}$ by finite type substacks, where $\mathcal{U}_m$ is the set of $(C,p_1, \ldots, p_n, \mathcal{L})$ such that $C$ has at most $m$ nodes and such that the absolute value of the degree of $\mathcal{L}$ on any component of $C$ is at most $m$. This cover is increasing, but does not satisfy that $\dim(\mathfrak{Pic}_{g,n} \setminus \mathcal{U}_m)$ goes to $-\infty$.
    
    The fact that no good filtration can exist follows from the observation that $\mathfrak{Pic}_{g,n}$ has infinitely many boundary divisors (corresponding to ways that the degree of $\mathcal{L}$ can split up on the components of a curve with two components) and no finite type stack $\mathcal{U}_1$ can contain all generic points of these divisors.
    
    The paper \cite{BHPSS} studied cycles and relations in the operational Chow ring $\CH^*_\text{op}(\mathfrak{Pic}_{g,n})$, see Section \ref{Sect:OperationalChow} for details. While we do not pursue this direction of study in the current paper, the Picard stack is one of our main motivations for introducing the property of being Lindelöf.
    \item For completeness, let us mention that an example of an irreducible scheme which is not Lindelöf is the \emph{line with uncountably many origins}, obtained from the disjoint union of uncountably many copies of the affine line by identifying them away from the origin. 
\end{enumerate}
\end{example}

% In order to show Chow groups behaves well under Chow-K\"unneth properties, we recall the following definition from \cite{oesinghaus}.

When $\M$ has a good filtration $(\mathcal{U}_m)_{m \in\mathbb{N}}$ by finite type substacks, for fixed $d$ we have
\[\CH^d(\M) = \CH^d(\mathcal{U}_m)\]
for $m>d$. This implies that, as long as we are interested in a fixed codimension, all computations can be carried out on a finite type stack and thus essentially all results for the Chow groups of such stacks carry over (e.g. the excision sequence, including the version extended on the left by one higher Chow group from \cite[Proposition 4.2.1]{kreschartin}).

For stacks which are Lindelöf, we get at least the first three terms of the excision sequence. 
%\Ycomment{In \cite{kreschartin} the proper pushforward is defined for projective morphisms between \textit{finite type} algebraic stacks over $k$. The definition of projective pushfoward naturally extends to a projective morphism between locally finite type algebraic stacks, see\label{Rmk:surproppush}.}
\begin{proposition} \label{Pro:excisionsequenceLindelof}
Let $\M$ be an equidimensional algebraic stack, locally finite type over a field $k$ which is Lindelöf. Let $j: \mathfrak{Z} \to \M$ be a closed substack with complement $i: \mathfrak{V} = \M \setminus \mathfrak{Z} \to \M$. Then there exists a complex
\begin{equation} \label{eqn:excisionLindelof}
    \CH_*(\mathfrak{Z}) \xrightarrow{j_*} \CH_*(\M) \xrightarrow{i^*} \CH_*(\mathfrak{V}) \to 0
\end{equation}
which is exact at $\CH_*(\mathfrak{V})$, i.e. $i^*$ is surjective. If moreover the stack $\mathfrak{V}$ is a {\textit{countable}} (finite or infinite) disjoint union of quotient stacks, the sequence is also exact at $\CH_*(\M)$.
\end{proposition}
Note that while the condition on $\mathfrak{V}$ being a countable union of quotient stacks might sound far-fetched, this \emph{is} the situation that we would encounter e.g. by taking $\M$ to be the boundary of the Picard stack $\mathfrak{Pic}_{g,n}$ and taking $\mathfrak{Z} \subset \M$ the closed substack where the curve has at least $2$ nodes.
\begin{proof}[Proof of Proposition \ref{Pro:excisionsequenceLindelof}]
Let $(\mathcal{U}_m)_{m \in \mathbb{N}}$ be an increasing cover of $\M$ by finite type substacks. Denote by $\M_m, \mathfrak{Z}_m, \mathfrak{V}_m$ the intersections of $\M, \mathfrak{Z}, \mathfrak{V}$ with $\mathcal{U}_m$ and by $j_m, i_m$ the restrictions of $j,i$. Then by the usual excision sequence for finite-type stacks we have that
\begin{equation} \label{eqn:exseq111}
    0 \to \CH_*(\mathfrak{Z}_m)/ \ker((j_m)_*) \xrightarrow{(j_m)_*} \CH_*(\M_m) \xrightarrow{i_m^*} \CH_*(\mathfrak{V}_m) \to 0\,.
\end{equation}
are exact sequences. From another application of the excision sequence we see that the restriction maps
\[\CH_*(\mathfrak{Z}_m) \to \CH_*(\mathfrak{Z}_{m'}) \text{ for }m' < m\]
are surjective. This implies that the system $(\CH_*(\mathfrak{Z}_m)/ \ker((j_m)_*))_m$ is Mittag-Leffler (see \cite[\href{https://stacks.math.columbia.edu/tag/0596}{Tag 0596}]{stacks-project}). Then it follows from \cite[\href{https://stacks.math.columbia.edu/tag/0598}{Tag 0598}]{stacks-project} that we obtain an exact sequence
\[
0 \to \varprojlim_m(\CH_*(\mathfrak{Z}_m)/ \ker((j_m)_*)) \to \CH_*(\M) \to \CH_*(\mathfrak{U}) \to 0\,.
\]
This finishes the proof of exactness of \eqref{eqn:excisionLindelof} at $\CH_*(\mathfrak{U})$. 

If $\mathfrak{V}$ is a {countable disjoint} union of quotient stacks, we can modify the exact sequence \eqref{eqn:exseq111}, extending it to the left to obtain
\begin{equation} 
    0 \to \CH_*(\mathfrak{V}_m,1)/ \ker(\partial_m) \xrightarrow{\partial_m} \CH_*(\mathfrak{Z}_m) \xrightarrow{(j_m)_*} \ker(i_m^*) \to 0\,.
\end{equation}
We claim that the directed system
\[(K_m)_m = \left(\CH_*(\mathfrak{V}_m,1)/ \ker(\partial_m) \right)_m\]
is Mittag-Leffler, i.e. that for $m$ fixed, the images of the restriction maps $K_{m'} \to K_m$ stabilize for $m'\gg m$. This follows from two easy observations:
\begin{itemize}
    \item Since by construction $\mathfrak{V}_m = \mathfrak{V} \cap \mathcal{U}_m$ is of finite type, it is supported on a finite number of connected components $\mathfrak{V}^1, \ldots, \mathfrak{V}^{e_m}$ of $\mathfrak{V}$.
    \item Since the stacks $\mathcal{U}_m$ cover $\M$, we know by Noetherian induction that for $m' \gg m$ the stack $\mathfrak{V}_{m'}$ contains the union of $\mathfrak{V}^1, \ldots, \mathfrak{V}^{e_m}$, so that we have
    \[
    K_{m'} = \bigoplus_{i=1}^{e_m} \CH_*(\mathfrak{V}^i, 1) \oplus \CH_*\left(\mathfrak{V}_{m'} \setminus \bigcup_{i=1}^{e_m} \mathfrak{V}^i , 1\right)\,.
    \]
\end{itemize}
Under the map $K_{m'} \to K_m$, the second direct summand always maps to zero (since it is supported on a different connected component). Thus indeed for $m' \gg m$, the image of $K_{m'} \to K_m$ stabilizes to the image of the restriction morphism
\[
\bigoplus_{i=1}^{e_m} \CH_*(\mathfrak{V}^i, 1) \to \CH_*(\mathfrak{V}_m, 1) = K_m\,.
\]
We conclude that the system $(K_m)_m$ is Mittag-Leffler, so again using \cite[\href{https://stacks.math.columbia.edu/tag/0598}{Tag 0598}]{stacks-project} we obtain an exact sequence
\[
0 \to \varprojlim_m(K_m) \to \CH_*(\mathfrak{Z}) \to \varprojlim_m \ker(i_m^*) = \ker(i^*) \to 0\,,
\]
where we use that taking directed inverse limits is left-exact to identify the limit of $\ker(i_m^*)$ as $\ker(i^*)$. Thus we conclude that $\CH_*(\mathfrak{Z})$ surjects onto the kernel of $i^*$, obtaining exactness of \eqref{eqn:excisionLindelof} at $\CH_*(\M)$.
% \jocomment{Problem: I don't see why the map $$\CH_*(\mathfrak{Z}) = \varprojlim_k \CH_*(\mathfrak{Z}_k) \to \varprojlim_k(\CH_*(\mathfrak{Z}_k)/ \ker((j_k)_*))$$
% should be surjective. This would require that $(\ker((j_k)_*))_k$ is Mittag-Leffler, which I currently don't see. Without this, we get a very short excision sequence indeed (just surjectivity of open pullback).}\Ycomment{I see. Unfortunately I also cannot see what kernels satisfies ML properties. Maybe this is the reason why Andrew was afraid of working with non finite stacks? Do you think it is true for particular example, e.g. $\mathfrak{Pic}_{0,n}^{sm}\subset \mathfrak{Pic}_{0,n}$?}
% \jocomment{I think I might see a path to proving this if the stack $\mathfrak{U}$ has the property that it is a countable disjoint union of finite-type quotient stacks, which I think would be the case for any inductive boundary-argument for $\mathfrak{Pic}$. The clue is that you can get $\ker((j_k)_*)$ as the image of the first higher Chow group of $\mathfrak{U}_k$. But $\mathfrak{U}_k$ only sees finitely many connected components of $\mathfrak{U}$ and for $k' \gg k$ the set $\mathfrak{U}_{k'}$ at some point contains all of those components and then any even higher $k''$ doesn't change the image anymore.}\Ycomment{Cool! Maybe we can include this assumption in the proposition?}
\end{proof}

%\jocomment{I added the following in response to the referee:}\Ycomment{Thanks! This comparison is reasonable to me.}
\begin{remark}
In the context of algebraic spaces, a different definition of Chow groups for  locally finite type spaces  is presented in \cite[\href{https://stacks.math.columbia.edu/tag/0EDZ}{Tag 0EDZ}]{stacks-project}. It works directly with formal linear combinations of cycles where locally only finitely many of the coefficients are nonzero. We expect that adapting the definitions of \cite{kreschartin}, one can give a similar definition in the case of algebraic stacks locally of finite type. It seems likely that for stacks which are far away from being finite type (e.g. stacks not being Lindel\"of) these alternative Chow groups have better formal properties than the groups we construct, since e.g. in the setting of algebraic spaces they satisfy an excision sequence without such  assumptions on the ambient space (see \cite[\href{https://stacks.math.columbia.edu/tag/0EP9}{Tag 0EP9}]{stacks-project}). The advantage of Definition \ref{Def:Chowlocallyft} is that it does not require us to reprove the standard constructions and properties of Chow groups since they descend easily to the projective limit. We do expect that the alternative definition of Chow groups modeled on \cite[\href{https://stacks.math.columbia.edu/tag/0EDZ}{Tag 0EDZ}]{stacks-project} coincides with our definition for algebraic stacks locally of finite type over a field which have a good filtration by finite type substacks.
\end{remark}

\section{Proper pushforward of cycles (joint with J. Skowera)}\label{sec:Properpushforward}
It is known that algebraic stacks of finite type over a field which can be stratified by locally closed substacks which are global quotient stacks admit an intersection theory with integral coefficients. The theory includes a pushforward operation for projective morphisms \cite{kreschartin}.  
In the definition of the tautological ring we frequently use pushforwards along forgetful and morphisms. For the moduli stacks of prestable curves, these morphisms are proper, but not projective. A priori this is a problem. Here we show that the pushforward operation for integral coefficients may be extended to proper representable morphisms. If the coefficients are rational, there is also a pushforward for morphisms of relative Deligne-Mumford type.  

\subsection{Definitions and terminologies}
Let a base field $k$ be fixed. {In this section} stacks are algebraic stacks of finite type over $k$. All morphisms are morphisms over $k$. A morphism of stacks $X \to Y$ is projective if it can be factored up to 2-isomorphism through a closed immersion into the projective bundle $\mathbb{P}(\mcE) \to Y$ for a coherent sheaf $\mcE$ of $\mcO_Y$-modules.

We briefly recall the construction of the Chow groups of an algebraic stack $Y$ of finite type over $k$. A representative of a $d$-cycle is a pair $(f, \alpha)$ for a projective morphism $f : X \to Y$ from an algebraic stack $X$ and a naive cycle $\alpha \in \CH^\circ_{d+\rk E}(E)$ on a vector bundle $E \to X$ of constant rank. The Chow groups of $Y$ are
\[
\CH_d(Y) =  \varinjlim_{X \to Y \in \mfA_{Y}} \left[ \widehat{\CH}_d (X) / \widehat{B}_d (X) \right]
\]
where
\begin{eqnarray*}
\widehat{\CH}_d(X) & = & \varinjlim_{E \in \mfB_X} \CH_{d+\rk E}^\circ E \\
\CH_d^\circ(E) & = & Z_d(E) / \delta \Rat_{d}(E) \\
Z_d(E) & = & \textrm{$d$-dimensional cycles} \\
\Rat_{d}(E) & = & \bigoplus_Z k(Z)^{*}, \textrm{$Z \subset E$ integral closed, $\dim Z = d+1$} \\
Ob(\mfA_Y) & = & \{ X \to Y \mid \textrm{ projective} \} \\
Mor(X, X') & = & \{  X \to X' \mid \textrm{$Y$-isom\mbox{.} onto an {open and closed substack} of $X'$} \} \\
Ob(\mfB_X) & = & \{ E \to X \mid \textrm{ vector bundle of constant rank} \} \\
Mor(E, F) & = & \{ E \to F \mid \textrm{ surjective vector bundle morphism } \}\,.
\end{eqnarray*}
The subgroup $\widehat{B}_d (X)$ is defined to be
\[
\widehat{B}_d(X) = \coprod_{\begin{subarray}{c}
p_1, p_2 : W \to X \\
g \circ p_1 \cong g \circ p_2
\end{subarray}} \coprod_{E, F \in \mfB_{X}} Z_{E, F}\,,
\]
\[
Z_{E, F} = \left\{ p_{2*}\beta_2 - p_{1*}\beta_1 \left| \begin{array}{c}
\beta_1 \in \CH_{d+\rk E}^\circ(p_1^{*}E), \quad \beta_2 \in \CH_{d+\rk F}^\circ(p_2^{*}F) \\
\beta_1 \sim \beta_2 \textrm{ in } \widehat{\CH}_d(W) 
\end{array} \right. 
\right\}
\]
where the first union is over {\it projective} morphisms and introduces a dependence on the morphism $g : X \to Y$ through the constraint $g \circ p_1 \cong g \circ p_2$. This accomplishes in a single step what the original definition \cite[Definition 2.1.4(iii)]{kreschartin} does in two by assuming the vector bundles to be of constant rank.

The above definition facilitates projective pushforwards: the cycle $(g, \alpha)$ pushes forward under $f$ to $(f\circ g, \alpha)$. A projective morphism $S \to U$ to an open substack $U \subset Y$ can be realized as the pullback of a projective morphism to $Y$ \cite[Corollary 2.3.2]{kreschartin}. It is unknown whether this property can be extended to proper morphisms.

\begin{definition}\label{def:restChow}
Let $f : X \to Y$ be a morphism of algebraic stacks. We further define restricted Chow groups formed from the pullbacks of vector bundles,
\[\CH^f_d (X) = \varinjlim_{Y' \to Y \in \mfA_{Y}} \left[ \widehat{\CH}_d^{f'} \left( X' \right) / \widehat{B}_d^{f'} \left( X' \right) \right],
\]
where $f' : X' \to Y'$ is the pullback of $f$ by the projective morphism $Y' \to Y$.
Recall that the restricted Edidin--Graham--Totaro Chow groups \cite[Definition 2.1.4(iv)]{kreschartin} are the groups defined by
\begin{equation}\label{e:restrictedvbchow}
\widehat{\CH}^{f'}_d (X') = \varinjlim_{E \in \mfB_{Y'}}\CH^\circ_{d + \rk E} (f^{'*} E)\,.
\end{equation}
The quotient group
\[
\widehat{B}_d^{f'}X' = \coprod_{\begin{subarray}{c} p_1, p_2 : W \to Y' \\ g \circ p_1 \cong g \circ p_2 \end{subarray}} \coprod_{E,F \in \mfB_{Y'}} Z_{E, F}^f\,,
\]
where
\[
Z_{E, F}^f = \left\{ 
p_{2*}' \beta_2 - p_{1*}' \beta_1 \left| 
\begin{array}{c}
\beta_1 \in \CH^\circ_{d + \rk E}({p'}_1^{*}{f'}^{*}E)\,, \,\beta_2 \in \CH^\circ_{d + \rk F}({p'}_2^{*}{f'}^{*}F) \\
\beta_1 \sim \beta_2 \textrm{ in } \widehat{\CH}_d^{f''}(W')
\end{array} \right.
\right\}\,, 
\]
depends on the notation,
\begin{equation*}
\begin{tikzcd} 
 W' \arrow[yshift=0.9ex,"p'_1"]{r} \arrow[yshift=-0.9ex,"p'_2"]{r} \arrow[d,"f''"] & X' \arrow[r,"g'"] \arrow[d,"f'"] & X \arrow[d,"f"]\\
 W  \arrow[yshift=0.9ex,"p_1"]{r} \arrow[yshift=-0.9ex,"p_2"]{r} & Y' \arrow[r,"g"] & Y\,.
\end{tikzcd}
\end{equation*}

Recall that we only consider vector bundles of constant rank. Therefore a cycle in $\CH^f_d (X)$ is represented by $(g, \alpha')$, $\alpha' \in \CH^\circ_{d + \rk E'} (E')$ as in the pullback diagram:
\begin{equation}\label{e:cycledef}
\begin{tikzcd}
    E' \ar[r] \ar[d,"f''"] & X' \ar[r] \ar[d,"f'"] & X \ar[d,"f\ \mathrm{rep}"] \\
    E \ar[r,"\pi"] & Y' \ar[r,"g\ \mathrm{proj}"] & Y\,.
\end{tikzcd}
\end{equation}
\noindent
There is a natural morphism from the restricted Chow groups to the usual Chow groups $$\iota_f : \CH^f_d(X) \to \CH_d(X)\,.$$
\end{definition}
When $f$ is proper and representable, the cycles in $\CH_d^f(X)$ can be pushed forward between naive Chow groups by direct generalization of the definition of pushforward for Deligne-Mumford stacks \cite[Definition 3.6]{vistoli}. If $E$ is a vector bundle on $Y'$, then $f_*$ pushes the class represented by a cycle $\alpha \in \CH^{\circ}_*(E')$ forward to ${f''}_*(\alpha) \in \CH^{\circ}_*(E)$ in the notation of diagram (\ref{e:cycledef}).

\begin{lemma}\label{p:restrictedpushforward}
If $f : X \to Y$ is a proper, representable morphism, then there is a proper pushforward $f_* : \CH^f_d(X) \to \CH_d (Y)$ for all $d$.
\end{lemma}

A morphism $f : X \to Y$ of stacks is \emph{of relative Deligne-Mumford type} if any morphism $T\to Y$ from a scheme $T$ to $Y$ pulls back to a Deligne-Mumford stack $X\times_Y T$. Note that representable morphisms are of relative Deligne-Mumford type. The proper pushforward of naive Chow groups $\CH^\circ(-)_\QQ$ along relative Deligne-Mumford type morphism is defined as follows: for $f\colon X\to Y$ a relative Deligne-Mumford type, proper morphism between algebraic stacks, we  construct a pushforward map
\begin{equation}\label{eqn:naivepush}
    f_*\colon Z_*(X)_\QQ\to Z_*(Y)_\QQ\,.
\end{equation}
%and
%\begin{equation}\label{eqn:naiverat}
%    f_*\colon \Rat_*(X)_\QQ \to \Rat_*(Y)_\QQ\,.
%\end{equation}
To define \eqref{eqn:naivepush}, the degree of $f$ defined in \cite{vistoli} should be extended to this setting. We first consider the case when $X$ is an integral algebraic stack and $Y$ is the image of $f$. 
For a smooth surjective morphism $u: U\to Y$ from an integral scheme $U$, consider the cartesian diagram
\begin{equation*}
\begin{tikzcd}
U_X\arrow[d]\arrow[r,"f'"] &U\arrow[d,"u"]  \\
X\arrow[r,"f"] & Y.
\end{tikzcd}
\end{equation*}
Since $f$ is relative Deligne-Mumford type, the fiber product $U_X$ is a Deligne-Mumford stack. Moreover $U_X$ is reduced because $U_X\to X$ is smooth and $X$ is reduced.
Let $U_X= \cup_i U_i$ be a finite number of irreducible components of $U$ and let $f_i'\colon U_i\to U$ be the restriction of $f'$ to each irreducible component. We define
\[\deg f =\begin{cases}
\sum_i \deg f'_i& \text{, when $\dim X = \dim Y$}\\
0& \text{, otherwise}
\end{cases}\]
where $\deg f'_i$ is the degree of $f'_i$ between Deligne-Mumford stacks defined in \cite{vistoli}\footnote{For a proper Deligne-Mumford type morphism, the degree is a rational number in general.}. This definition is independent of the choice of a smooth atlas $u$. 
For an arbitrary $\QQ$-linear combination of integral cycles $Z=\sum_j a_j\cdot Z_j$ the map \eqref{eqn:naivepush} is defined by
\[[Z] \mapsto \sum_j a_j\deg(f|_{Z_j}) \cdot [f(Z_j)]\,.\]
It is straightforward to check that $f_*$ is functorial. 

%\jocomment{I inserted the following attempt at a proof; please read carefully}
Similarly, we can define a proper pushforward
\[ f_*\colon \Rat_*(X) \to \Rat_*(Y) \]
on the space of rational equivalences. Let $Z \subseteq X$ be an integral closed substack of dimension $d+1$ and let $\varphi \in k(Z)^*$, giving an element $(Z, \varphi) \in \Rat_d(Y)$. Then for $Z'=f(Z)\subseteq Y$ we set $f_*(Z, \varphi)=0$ if $\dim Z'<\dim Z$ and $f_*(Z, \varphi)=(Z',N_{k(Z)/k(Z')}(\varphi))$ otherwise, where $N_{k(Z)/k(Z')}$ is the norm map of the finite field extension $k(Z)/k(Z')$. 

Before continuing, we want to argue that this pushforward of rational equivalences is compatible with taking smooth covers. Indeed, given a smooth surjective morphism $U \to Y$ of relative dimension $e$ consider the fibre diagram
\[
\begin{tikzcd}
V \arrow[d] \arrow[r,"\widehat{f}"] & U \arrow[d]\\
X \arrow[r,"f"] & Y
\end{tikzcd}\,.
\]
Pullbacks to $U,V$ induce maps $\Rat_*(Y) \to \Rat_{*+e}(U)$ and $\Rat_*(X) \to \Rat_{*+e}(V)$ and we claim that the natural diagram
\begin{equation} \label{eqn:fieldextension}
\begin{tikzcd}
\Rat_{*+e}(V) \arrow[r,"\widehat{f}_*"] & \Rat_{*+e}(U) \\
\Rat_*(X) \arrow[u] \arrow[r,"f_*"] & \Rat_*(Y) \arrow[u]
\end{tikzcd}\,.
\end{equation}
commutes. To see this, let $Z \subseteq X$ be an integral closed substack, write $Z'=f(Z)$ and denote by $Z_V, Z'_U$ their preimages in $V,U$. Then $\dim(Z')<\dim(Z)$ if and only if $\dim(Z'_U)<\dim(Z_V)$. On the other hand, when we have equality of dimensions and denote by $Z_{V,i}, Z'_{U,i}$ the irreducible components of $Z_V, Z'_U$, it holds that $k(Z_{V,i}) = k(Z) \otimes_{k(Z')} k(Z'_{U,i})$. 
%\jocomment{Are you happy with this version?}
Then we claim that for $\varphi \in k(Z)$ we have 
\begin{equation} \label{eqn:norm}
N_{k(Z)/k(Z')}(\varphi) = N_{k(Z_{V,i})/k(Z'_{U,i})}(\varphi)\,,
\end{equation}
which finishes the proof of the commutativity of the diagram \eqref{eqn:fieldextension}. But indeed, the claim \eqref{eqn:norm} follows from the definition of the norm, since the determinant of an endomorphism of a vector space over some field $K$ is unchanged by tensoring with an extension field of $K$.

\begin{lemma}
The pushforward $f_*$ commutes with the boundary map $\delta$ and thus passes through rational equivalence.
\end{lemma}
\begin{proof}
We first argue that it is enough to prove the compatibility on a smooth atlas of $Y$. Let $U\to Y$ be a smooth atlas of relative dimension $e$, then we have a commutative diagram
\begin{equation*}
\begin{tikzcd}
    0 \arrow[r] & Z_*(Y) \arrow[r] & Z_{*+e}(U) \arrow[r, r, yshift=0.5ex] \arrow[r, yshift=-0.5ex] & Z_{*+2e}(U \times_Y U)\,,\\
    0 \arrow[r]& \Rat_*(Y) \arrow[u,"\delta"] \arrow[r]& \Rat_{*+e}(U) \arrow[u,"\delta"] \arrow[r, r, yshift=0.5ex] \arrow[r, yshift=-0.5ex]& \Rat_{*+2e}(U \times_Y U) \arrow[u,"\delta"]
\end{tikzcd}
\end{equation*}
where the rows are equalizer sequences. There is a similar diagram for the induced map $U \times_Y X \to X$, and this diagram maps to the diagram above via the proper pushforward $f_*$.
But then due to the injectivity of the map $Z_*(Y) \to Z_{*+e}(U)$, the claimed equality of maps
\[
f_* \circ \delta = \delta \circ f_* : \Rat_*(X) \to Z_*(Y)
\]
can be checked on $U$. Thus we can reduce to checking the lemma for a target which is a scheme, in which case the corresponding argument was given in the proof of \cite[Proposition 3.7]{vistoli}.\footnote{We want to thank Andrew Kresch and Rachel Webb for many helpful discussions on technical details of this proof.}
\end{proof}
\subsection{Properties of restricted Chow groups}
%%%%%%%%%%%%%%%%%%%%%%%%%%%%%%%%%%%%%%

\begin{definition}[\cite{totaro-1999}]\label{d:approxvectorbundle}
An approximating vector bundle of $X$ in codimension $d$ is a vector bundle $E \to X$ such that $E \setminus S$ is an algebraic space for a closed substack $S$ with $\codim_E S > d$. It is called approximating, because $\CH_j(X) \isomarrow \CH^{\circ}_{j+\rk E}(E)$ for large enough $j$ (\cite[Corollary 2.4.9]{kreschartin}).
%\Ycomment{Why did Skowera put this sentence? This is not the content of [Cor 2.4.9] and this is what we have to prove in Lemma B.6.}
%\jocomment{The cited Corollary says that $\CH_j(X) \isomarrow \CH_{j+\rk E}(E)$ and the excision sequence shows that $\CH_{j+\rk E}(E)\isomarrow \CH_{j+\rk E}(E \setminus S)$, and since $E \setminus S$ is an algebraic space, its Kresch-Chow group agrees with its naive Chow group (by \cite[Remark 2.1.16]{kreschartin}). And if naive Chow groups still satisfy excision, then you can go back to the naive Chow group of $E$ and obtain the statement that Skowera claimed.}
%
%\jocomment{And I also agree with your last remark: I think the Lemma below just states that the inverse of the isomorphism $\CH_j(X) \isomarrow \CH^{\circ}_{j+\rk E}(E)$ claimed above is given by the map sending a naive cycle $\delta$ to $(1_X, \delta) \in \CH_j(X)$.}
\end{definition}
Let $X = [\widehat{X}/G]$ be a global quotient stack. Then for any $d$, $X$ has an approximating vector bundle in codimension $d$ which is the pullback of $\Hom_k(k^{N+n}, W)$ from $BG$ for a faithful representation $W$ of $G$ of dimension $n$ and $N$ large (\cite[Remark 1.4]{totaro-1999}).

%Indeed, take $S$ to be the closed subvariety defined by the vanishing of all $n$-minors. Then $S$ has codimension $(N+1)^2$ and $\Hom(k^{N+n}, W) \setminus S \isom GL(N+n) / GL(N+n)_U$ for any $n$-dimensional subspace $U \subset k^{N+n}$.\cite[Remark 1.4]{totaro-1999}
%\Ycomment{Honestly I could not understand the previous proof... maybe we can fix as follows?}
\begin{lemma}\label{l:equivcycles}
Let $f: X \to Y$ be a representable morphism. Let $(g, \alpha)$ be a cycle in $\CH_d(X)$ or $\CH^f_d(X)$ and let $F$ be an approximating vector bundle of $X$ in codimension $D$, respectively the pullback by $f$ of an approximating vector bundle on $Y$ in codimension $D$. When $D$ is sufficiently large, we have $(g, \alpha) \sim (1_X, \delta)$ for some naive cycle $\delta \in \CH^{\circ}_{d + \rk F}(F)$. 
In fact, the natural map $\CH^\circ _{d+\rk F}(F) \to \CH_d(X), \delta \mapsto (1_X, \delta)$ is an isomorphism.
\end{lemma}
\begin{proof}
We consider the case of $\CH_d(X)$, the case of $\CH^f_d(X)$ is analogous. Let $\alpha \in \CH_{d+\rk E}^\circ (E)$ be a representative of a cycle in $\CH_d(X)$. Consider the following cartesian diagram
\begin{equation*}
\begin{tikzcd}
  E \oplus H \arrow[d,"p"] \ar[r,"q"] & H \ar[r,"s"] \ar[d] & F \ar[d] \\
  E \ar[r,"\mathrm{v.b.}"] & X' \ar[r,"g~ \mathrm{proj}"] & X\,.
\end{tikzcd}    
\end{equation*}
Let $S\subset F$ be a closed substack with $\codim_F S>D$ and such that $F\setminus S$ is an algebraic space. 
Then the intersection of $S$ with any fibre of $F \to X$ has codimension at least $D-\dim(X)$ inside the fibre, and thus, after taking the cartesian diagram with $g: X' \to X$, we have
\[
\codim_H s^{-1}(S)>D - \dim(X)\,.
\]
Moreover, since $s$ is projective we know that $H$ is an algebraic space away from $s^{-1}(S)$. Let $r = \rk F$, then for $D>\dim(X)+\dim(X')-d$, we have that $\CH_{d+r}(H) \cong \CH_{d+r}^\circ(H)$ and likewise $\CH_{d+r+\rk E}(E \oplus H) \cong \CH_{d+r+\rk E}^\circ(E \oplus H)$. {Since the pullback $q^*$ on the usual Chow groups is an isomorphism by} \cite[Theorem 2.1.12 (vi)]{kreschartin}, we conclude that
\[q^*: \CH_{d+r}^\circ(H)\to \CH^\circ_{d+r+\rk E} (E\oplus H)\]
is {an isomorphism}.

Now starting with the cycle $\alpha$ on $E$, we have that 
$\beta=p^*\alpha$ is a cycle on $E\oplus H$.  {Since the map $q^*$ above is an isomorphism, there exist a unique} $\gamma \in \CH_{d+r}^\circ(H)$ such that $\beta = q^{*}\gamma$. Then on the one hand we have $(g, \alpha) \sim (g, \beta) \sim (g, \gamma)$ by \cite[Remark 2.1.5]{kreschartin}. On the other hand, defining $\delta = s_{*}\gamma$  we have $(g, \gamma) \sim (1_X, \delta)$ as in \cite[Remark 2.1.16]{kreschartin}, concluding the proof. Going through this construction, the association $(g,\alpha) \mapsto (1_X, \delta)$ gives a well-defined map, and it is immediate to check that it gives the inverse to the natural homomorphism $\CH^\circ _{d+\rk F}(F) \to \CH_d(X)$.
\end{proof}
%\Ycomment{Old version of the proof: Then $\beta=p^*\alpha$ is a cycle on $E\oplus H$. Let $r=\rk F$. Since $H$ is algebraic space in dimension $d+r$, the pullback $q^*$ is surjective.  Let $\beta = q^{*}\gamma$ and $\delta = s_{*}\gamma$ so that $(g, \alpha) \sim (g, \gamma)$. But $(g, \gamma) \sim (1_X, \delta)$ as in \cite[Remark 2.1.16]{kreschartin}.}

\begin{proposition}\label{p:quotstackisom}
If $Y$ is a global quotient stack and $f : X \to Y$ is representable, then for all $d$ the natural morphism $\iota_f: \CH^f_d(X) \to \CH_d(X)$ is an isomorphism.
\end{proposition}
\begin{proof}
In the notation of (\ref{e:cycledef}), {let $F$ be the pullback of an approximating vector bundle on $Y$.} Then $\iota_f$ factors as
\[\CH^f_d(X) \isomto \CH^{\circ}_{d+\rk F}(F) \isomto \CH_{d+\rk F}(F) \isomfrom \CH_d(X)\,.\]
The association $(g, \alpha) \mapsto (1_X, \delta)$ of Lemma \ref{l:equivcycles} defines the first isomorphism. Its inverse is the inclusion $\CH^{\circ}_{d+\rk F}(F) \to \CH^f_d(X)$. The second map is an isomorphism because $F$ is an algebraic space in dimension $d+\rk F$ (\cite[Theorem 2.1.12(i)]{kreschartin}). The third map is an isomorphism by homotopy invariance. \cite[Corollary 2.4.9]{kreschartin}.
% pullbacks of $F$ by $g$ and $\pi$ are approximating vector bundles and a cycle in $A^{\circ}_{d+\rk E'}(E')$ can be pulled and pushed through them to $A^{\circ}_{d+\rk F}(F)$. We can write out the inverse of this inclusion $A^f_d(X) \to A^{\circ}_{d + \rk F}(F)$, so it is an isomorphism. But by the definition of an approximating bundle, $A^{\circ}_{d+\rk F}(F) \isom A_d(X)$.
\end{proof}
% For the first isomorphism, let $(g, \alpha) \in A^f_dY$, and let $F' \to Y$ be the pullback of $F$. It is an appoximating vector bundle on $Y$. Forming pullbacks gives
%\[
%\xymatrix{
%E' \oplus F'' \ar[r]_q \ar[d]_p & F'' \ar[r]^{g'} \ar[d] & F' \ar[d] \\
%E' \ar[r] & Y' \ar[r]_g & Y
%}
%\]
%where $\alpha \in A^{\circ}_{d + \rk E'}(E')$. Since $g$ is representable, $F''$ is also an approximating vector bundle, and $E' \oplus F'' \to F''$ is a vector bundle on an algebraic space apart from a compliment of sufficiently high codimension. Hence $A^{\circ}_{d+\rk E'\oplus F''}E' \oplus F'' \cong A^{\circ}_{d+\rk F''}F''$. Upon pulling $\alpha$ back to $E' \oplus F''$, we may thus map it forward to some $\beta := (q^*)^{-1}p^*\alpha \in A^{\circ}_{d+\rk F'}F'$. So $(g, \alpha) = (g, \beta)$ in $A^f_dY$. Furthermore, $(g, \beta) = (1_Y, g'_* \beta)$ in $A_d^{f}(Y)$, because the argument

\begin{proposition}
If $f : X \to Y$ is of relative Deligne-Mumford type and $Y$ is a global quotient stack, then the natural morphism $\iota_f : \CH^f_d(X, \QQ) \to \CH_d(X, \QQ)$ is an isomorphism for all $d$.
\end{proposition}
\begin{proof}
By the fact that $\CH^{\circ}_d(X, \QQ) \cong \CH_d(X, \QQ)$ for any Deligne-Mumford stack $X$ \cite[Theorem 2.1.12(ii)]{kreschartin}, the morphism $\iota_f$ factors in the same way as $\iota_f$ in Proposition \ref{p:quotstackisom} over $\QQ$.
\end{proof}
Let $f:X\to Y$ be a proper morphism of relative Deligne-Mumford type and let $h:W\to Y$ be a morphism of algebraic stacks. Consider the following cartesian diagram
\begin{equation*}
\begin{tikzcd}
V \ar[r] \ar[d,"\tilde{f}"] & X \ar[d,"f"] \\
W \ar[r,"h"] & Y\,.
\end{tikzcd}
\end{equation*}
For each $d$, the groups $\CH^f_d (X)$ satisfy the following properties:
\begin{itemize}
\item[(i)] Let the above morphism $h:W \to Y$ be a flat morphism of relative dimension $r$. Then there is a functorial pullback homomorphism $\CH^f_d(X) \to \CH^{\tilde{f}}_{d+r}(V)$.
\item[(ii)] Let the above morphism $h:W \to Y$ be a projective morphism, then there is a functorial pushforward homomorphism $\CH^{\tilde{f}}_d(V) \to \CH^f_d(X)$.
\item[(iii)] %\jocomment{I rephrased the third point and included the diagrams, you think it helps?}\Ycomment{Nice!} 
The homomorphisms (i) and (ii) are compatible with the natural morphism from restricted Chow groups to the usual Chow groups; that is, we have commutative diagrams %$\CH^f_d(X)\to \CH_d(X)$.
\[
\begin{tikzcd}
\CH^f_d(X) \arrow[r,"h^*"] \arrow[d] & \CH^{\tilde{f}}_{d+r}(V) \arrow[d]\\
\CH_d(X) \arrow[r,"h^*"] & \CH_{d+r}(V)
\end{tikzcd}
\quad \text{and} \quad
\begin{tikzcd}
\CH^{\tilde{f}}_d(V) \arrow[r,"h_*"] \arrow[d] & \CH^f_d(X) \arrow[d]\\
\CH_d(V) \arrow[r,"h_*"] & \CH_d(X)
\end{tikzcd}\,.
\]
\end{itemize}
\begin{proposition}[Excision]\label{p:restrictedexcision}
Let $i : Z \to Y$ be a closed substack, and $j : U \to Y$ its complement. Let $Z' \to X$ be the preimage of $Z$ under $f$, and let $U'$ be its complement. Let $f' = f|_{Z'}$ and $\widetilde{f} = f|_{U'}$:
\begin{equation*}
\begin{tikzcd}
Z'\arrow[r]\arrow[d,"f'"] & X\arrow[d,"f"] & U'\arrow[l]\arrow[d,"\tilde{f}"]\\
Z \arrow[r,"i"] & Y & U\arrow[l,"j",swap]\,. 
\end{tikzcd}    
\end{equation*}
Then for each $d$ the flat pullback and projective pushforward fit together into an exact sequence
\begin{equation*}
\begin{tikzcd}
\CH^{f'}_d(Z') \ar[r,"i_*"] & \CH^f_d(X) \ar[r,"j^*"] & \CH^{\widetilde{f}}_d(U') \ar[r] & 0\,.
\end{tikzcd}    
\end{equation*}
\end{proposition}
\begin{proof}
This is \cite[Proposition 2.3.6]{kreschartin} where constructions are performed in the lower level $(Y, Z, U)$ and pulled back by $f$.
\end{proof}
Recall that the excision sequence for Chow groups may be extended using underlined Chow groups. Imitating the definition of Chow groups, these were defined in \cite{kreschartin}, beginning with a ``naive" variant $\underline{A}^\circ_d(X)$ in \cite[Definition 4.1.3]{kreschartin}, which satisfies functoriality under proper pushforwards and flat pullbacks. As before, the theory is extended via limits over projective morphisms $Z \to X$ and vector bundles on $Z$ in \cite[Corollary 4.1.10]{kreschartin}. Here we define their restricted analogues. 
\begin{definition}
Let $f : X \to Y$ be a morphism of algebraic stacks. Let
\[
\underline{A}^f_d (X) =
\varinjlim_{Y' \in \mfA_{Y}} \left[ \underline{\widehat{A}}_d^{f'} \left( X' \right) / \underline{\widehat{B}}_d^{f'} \left( X' \right) \right]
\]
where $\underline{\widehat{A}}_d^{f'}$ is defined in analogy with equation (\ref{e:restrictedvbchow}). 
%\jocomment{My guess is the following, which is just "take all the text below equation (\ref{e:restrictedvbchow}) and replace Chow groups by underlined Chow groups" :} 
The quotient group 
\[
\underline{\widehat{B}}_d^{f'}X' = \coprod_{\begin{subarray}{c} p_1, p_2 : W \to Y' \\ g \circ p_1 \cong g \circ p_2 \end{subarray}} \coprod_{E,F \in \mfB_{Y'}} \underline{Z}_{E, F}^f\,,
\]
where
\[
\underline{Z}_{E, F}^f = \left\{ 
p_{2*}' \beta_2 - p_{1*}' \beta_1 \left| 
\begin{array}{c}
\beta_1 \in \underline{A}^\circ_{d + \rk E}({p'}_1^{*}{f'}^{*}E)\,, \,\beta_2 \in \underline{A}^\circ_{d + \rk F}({p'}_2^{*}{f'}^{*}F) \\
\beta_1 \sim \beta_2 \textrm{ in } \widehat{\underline{A}}_d^{f''}(W')
\end{array} \right.
\right\}\,, 
\]
depends on the diagram used in the Definition \ref{def:restChow}.
There is a natural homomorphism $\iota_f: \underline{A}^f_d(X) \to \underline{A}_d(X)$.
\end{definition}

\begin{proposition}\label{p:vbondm}
Let $p : E \to Y$ be a vector bundle on a Deligne-Mumford stack. Then the natural flat pullback map $p^* : \underline{A}_d^\circ(Y, \QQ) \to \underline{A}_{d + \rk E}^\circ(E, \QQ)$ is a surjection for all $d$. %(Cf. \cite[Remark 2.1.16]{kreschartin})
\end{proposition}
\begin{proof}
First, we note that the existence of the map $p^*$ follows directly from the definition of $\underline{A}_d^\circ$, as was observed in \cite[Remark 4.1.4]{kreschartin}. To show that it is surjective, we first reduce to the case of global quotient stacks of the form $[W/G]$ for a scheme $W$ and a finite group $G$. Assuming the proposition holds for such quotient stacks, let $U \subset Y$ be such a stack, non-empty and open in $Y$. Such $U$ always exists by \cite[Corollaire 6.1.1]{laumon-2000}. By naturality of the long exact sequence in \cite[equation (4.2.1)]{kreschartin}, there are morphisms
\begin{equation*}
\begin{tikzcd}
    \underline{A}_*^{\circ}(Z, \QQ) \ar[d] \ar[r] & \underline{A}_*^{\circ}(Y, \QQ) \ar[d] \ar[r] & \underline{A}_*^\circ(U, \QQ) \ar[d,"\sim"] \ar[r,"\delta"] & \CH_*^\circ(Z, \QQ) \ar[d,"\sim"] \\
\underline{A}_*^{\circ}(E|_Z, \QQ) \ar[r] & \underline{A}_*^\circ(E, \QQ) \ar[r] & \underline{A}_*^\circ(E|_U, \QQ) \ar[r] & \CH_*^\circ(E|_Z, \QQ)
\end{tikzcd}    
\end{equation*}
where $Z = Y \setminus U$. The rightmost morphism is an isomorphism by homotopy invariance of Deligne-Mumford stacks with rational coefficients. By noetherian induction on $Y$, the leftmost vertical morphism is surjective. Hence $\underline{A}_*^\circ(Y,\QQ) \to \underline{A}_*^\circ(E, \QQ)$ is a surjection by the four lemma.

For the base case, consider a quotient $[W / G]$ with vector bundle $E$. Then $E$ has the form $[V / G]$ for a vector bundle $V \to W$. There is a diagram
\begin{equation*}
\begin{tikzcd}
\CH_*(W; 1)_{\QQ} \ar[r,"\sim"] & \CH_*(V;1)_{\QQ} \\
\underline{A}_*^\circ(W, \QQ) \ar[xshift=0.9ex,d,"q_*"] \ar[u,"\sim"] \ar[r,"\sim", "s^*"'] & \underline{A}_*^\circ(V, \QQ) \ar[xshift=0.9ex,d,"r_*"] \ar[u,"\sim"] \\
\underline{A}_*^\circ([W/G], \QQ) \ar[xshift=-0.9ex,u,"q^*"] \ar[r,"p^*"'] & \underline{A}_*^\circ(E, \QQ) \ar[xshift=-0.9ex,u,"r^*"]
\end{tikzcd}    
\end{equation*}
The upper square vertical isomorphisms follow from the natural isomorphism for schemes $\underline{A}_*^\circ(-) \isomarrow \CH_*(-;1)$, see \cite[Proposition 4.1.7]{kreschartin}.
%\jocomment{Skowera is lying here again, since the cited reference says "schemes". But this could be fixed by assuming that $W$ is a scheme, right, but is it true that every DM stack is Zariski locally = [quasiprojective scheme / finite group]? The problem is that finiteness of $G$ is used below.}
The lower squares exist, because the quotient morphisms are flat and proper, and they commute by compatibility of pullback and pushforward. The morphism $r_*$ is surjective, because $r_* \circ r^*$ is multiplication by $|G|$. But $r_* = p^* \circ q_* \circ (s^*)^{-1}$, so $p^*$ must also be surjective. By similar reasoning $p^*$ is injective, hence an isomorphism.
\end{proof}
\noindent 
Using Proposition \ref{p:vbondm}, we can generalize Lemma \ref{l:equivcycles} to the setting of underlined Chow groups.
\begin{lemma}\label{l:equivcyclesunderline}
Let $f: X \to Y$ be a representable (resp. relative Deligne-Mumford type) morphism  and let $F$ be the pullback by $f$ of an approximating vector bundle on $Y$ in codimension $D$. Let furthermore $(g, \alpha)$ be a cycle in $\underline{A}_d(X)$ or $\underline{A}^f_d(X)$ (resp. a cycle in $\underline{A}_d(X, \mathbb{Q})$). When $D$ is sufficiently large, we have $(g, \alpha) \sim (1_X, \delta)$ for some naive cycle $\delta \in \underline{A}^{\circ}_{d + \rk F}(F)$ (resp. $\delta \in \underline{A}^{\circ}_{d + \rk F}(F, \mathbb{Q})$).
In fact, in this case the natural maps \begin{equation} \label{eqn:jisomorphisms} j : \underline{A}^{\circ}_{d + \rk F}(F) \to \underline{A}_{d}(X),\quad j_f : \underline{A}^{\circ}_{d + \rk F}(F) \to \underline{A}^f_{d}(X)\end{equation}
are isomorphisms. 
\end{lemma}
%Using Proposition \ref{p:vbondm}, we see that the proof of Lemma \ref{l:equivcycles} works for $(g, \alpha)$ in $\underline{A}_d(X)$ or $\underline{A}_d(X, \QQ)$ with $F'$ an approximating bundle or, respectively, a Deligne-Mumford stack away from a closed substack of high codimension.
\begin{proof}
To show the lemma, one repeats the proof of Lemma \ref{l:equivcycles} verbatim, replacing (naive) Chow groups with (naive) underlined Chow groups. In the last step of the argument, we 
need to show that the pullback $q^*$ under a vector bundle map $q: E \oplus H \to H$ induces an isomorphism of naive underlined Chow groups in the correct degree, for $f$ representable. To see this, we note that $E \oplus H$ and $H$ are algebraic spaces away from a high-codimension subset and so by the excision sequence for naive underlined Chow groups (\cite[equation (4.2.1)]{kreschartin}) these groups are unchanged by restricting to these open subspaces. Then the isomorphism follows by identifying naive and non-naive underlined Chow groups of these algebraic spaces (\cite[Corollary 4.1.10]{kreschartin}) and using that pullback by a vector bundle induces an isomorphism for the latter (\cite[Proposition 4.3.1]{kreschartin}).\footnote{Note that while these results are stated with scheme assumptions, their proofs (in particular Remark 4.1.2 of \cite{kreschartin}) can be generalized to algebraic spaces. We thank Andrew Kresch for pointing this out.}

On the other hand, in case $f$ is of relative Deligne-Mumford type, we need to show that $q^*$ gives a surjection, which follows from  Proposition \ref{p:vbondm}.
\end{proof}
%\jocomment{If I see correctly, we just want to say here that we can now repeat the proofs of Lemma \ref{l:equivcycles} and Proposition \ref{p:quotstackisom} in the setting of underlined Chow groups; the statement below is the equivalent of Proposition \ref{p:quotstackisom}; maybe we can discuss again what to write here}

% \jocomment{Ok, I reformulated Lemmas \ref{l:equivcycles} and \ref{l:equivcyclesunderline} and simplified the proof of the next proposition. I am still a tiny bit uncertain about one point: in the proof of Lemma \ref{l:equivcycles} we use an excision sequence to remove a high-codimension set without changing Chow, and we use that naive to usual Chow agree on algebraic spaces. But for underlined Chow, we don't have an excision sequence and we only know the equality of naive and usual on schemes \cite[Corollary 4.1.10]{kreschartin}. Can we assume that our global quotient stacks are actually quasiprojective schemes modulo linear group actions? In the proof of Proposition 4.3.2. in \cite{kreschartin} Andrew seems to use that this is enough. And are then approximating line bundles actually schemes away from high codimension? This might get rid of the second problem. For the first I am not sure. }

\begin{proposition}\label{p:underlinerepgq}
If $f : X \to Y$ is representable and $Y=[V/G]$ is the global quotient stack of a quasi-projective scheme $V$ by a linear algebraic group $G$ acting linearly on $U$. Then the natural morphism $\iota_f : \underline{A}_d^f (X) \to \underline{A}_d (X)$ is an isomorphism for all $d$.
\end{proposition}
\begin{proof}
The assumptions on $Y$ imply that it possesses approximating vector bundles which are algebraic spaces away from subsets of arbitrarily large codimension by using \cite[Lemma 9, Proposition 23]{EdidinGraham98}.
Let $F$ be the pullback by $f$ of such a vector bundle, then the map $\iota_f$ is simply the composition $\iota_f = j \circ j_f^{-1}$ for the isomorphisms $j,j_f$ from equation \eqref{eqn:jisomorphisms}.
%\jocomment{I propose to replace $F'$ by $F$ again in this entire proof to be consistent with  Lemma \ref{l:equivcyclesunderline}}\Ycomment{Sure I fixed.}
% Let $F$ be an approximating vector bundle on $X$ pulled back from $Y$. In this case, the association $(g, \alpha) \mapsto (1_X, \delta)$ from Lemma \ref{l:equivcyclesunderline} defines an isomorphism from $\underline{A}^f_d(X)$ whose inverse is the inclusion $j_f : \underline{A}^{\circ}_{d + \rk F}(F) \to \underline{A}^f_d(X)$. Likewise, it defines an isomorphism from $\underline{A}_d(X)$ whose inverse is the inclusion $j : \underline{A}^{\circ}_{d+\rk F}(F) \to \underline{A}_d(X)$. Then $j = \iota_f \circ j_f$, so $\iota_f$ is an isomorphism.
\end{proof}
For relative Deligne-Mumford type morphisms, we have a weaker property which is enough for our purpose.
\begin{proposition}\label{p:underlinerelisom}
If $f : X \to Y$ is of relative Deligne-Mumford type and $Y$ is a global quotient stack, the natural morphism $\iota_f: \underline{A}^f_d(X, \QQ) \to \underline{A}_d(X, \QQ)$ is a surjection for all $d$.
\end{proposition}
\begin{proof}
Let $(g, \alpha) \in \underline{A}_d(X, \QQ)$. By Lemma \ref{l:equivcyclesunderline}, $(g, \alpha) \sim (1_X, \delta)$ for some naive cycle $\delta$ on the pullback of a vector bundle on $Y$ by $f$. Then $(1_X, \delta)$ maps to $(g, \alpha)$ under $\iota_f$.
\end{proof}

\begin{proposition}\label{p:restrictedconnhom}
With the notation of Proposition \ref{p:restrictedexcision}, there is a connecting homomorphism $\delta_f$ fitting into a commutative diagram with the connecting homomorphism of \cite[equation (4.2.2)]{kreschartin},
\begin{equation*}
\begin{tikzcd}
    \underline{A}^{\widetilde{f}}_d (U') \ar[r,"\delta_f"] \ar[d] & \CH^{f'}_d(Z') \ar[d] \\
\underline{A}_d (U') \ar[r,"\delta"] & \CH_d(Z')\,.
\end{tikzcd}    
\end{equation*}
\end{proposition}
\begin{proof}
The construction of $\delta$ in \cite[equation (4.2.2)]{kreschartin} goes through for restricted Chow groups.
\end{proof}

\noindent Propositions \ref{p:restrictedexcision} and \ref{p:restrictedconnhom} fit together in an exact sequence. 

\begin{proposition}\label{p:restrictedexactseq}
With the notation of Proposition \ref{p:restrictedexcision}, let $U$ be a global quotient stack. Then there is an exact sequence
\[
\underline{A}^{\widetilde{f}}_d (U') \stackrel{\delta_f}\to \CH^{f'}_d(Z') \to \CH^f_d(X) \to \CH^{\widetilde{f}}_d(U') \to 0\,.
\]
\end{proposition}
\begin{proof}
The proof is completely analogous to the proof of \cite[Proposition 4.2.1]{kreschartin}.
\end{proof}
\subsection{Main theorems}
Recall that a stack admits a stratification by global quotient stacks if and only if every geometric stabilizer is affine \cite[Proposition 3.5.9]{kreschartin}.
\begin{proposition}
Let $Y$ be a stack stratified by global quotient stacks and let $f : X \to Y$ be representable. Then $\CH^f_d(X) \to \CH_d(X)$ is an isomorphism for all $d$. If, alternatively, $f$ is of relative Deligne-Mumford type, then $\CH^f_d(X, \QQ) \to \CH_d(X, \QQ)$ is an isomorphism for all $d$.
\end{proposition}
\begin{proof}
The proof proceeds by Noetherian induction. Using the same notation as Proposition \ref{p:restrictedexcision}, let $U$ be a nonempty global quotient stack. There is a morphism from the exact sequence of Proposition \ref{p:restrictedexactseq} to the exact sequence of \cite[Proposition 4.2.1]{kreschartin}
\begin{equation*}
\begin{tikzcd}
    \underline{A}^{\widetilde{f}}_d (U') \ar[d,"c'"] \ar[r,"\delta_f"] \ar[d] & \CH^{f'}_d(Z') \ar[r] \ar[d,"a"] & \CH^f_d (X)\ar[r] \ar[d,"b"] & \CH^{\widetilde{f}}_d(U') \ar[r] \ar[d,"c"] & 0 \\
    \underline{A}_d (U') \ar[r,"\delta"] & \CH_d(Z') \ar[r] & \CH_d (X) \ar[r] & \CH_d (U') \ar[r] & 0\,.
\end{tikzcd}    
\end{equation*}
The morphism $a$ is an isomorphism by the induction hypothesis, and morphism $c$ by Proposition \ref{p:quotstackisom}. Then $c'$ is an isomorphism by Proposition \ref{p:underlinerepgq} if $f$ is representable, and is a surjection by Proposition \ref{p:underlinerelisom} if $f$ is of relative Deligne-Mumford type where all Chow groups are taken with $\QQ$-coefficients. By the five lemma, $b$ is also an isomorphism.
\end{proof}

\begin{theorem} \label{thm:properpushforward}
Let $Y$ be {an algebraic} stack stratified by global quotient stacks and let $f : X \to Y$ be a proper, representable morphism. Then there is a proper pushforward $f_* : \CH_d(X) \to \CH_d(Y)$ for all $d$. If, instead, $f$ is proper and of relative Deligne-Mumford type, then there is a proper pushforward $f_* : \CH_d(X, \QQ) \to \CH_d(Y, \QQ)$ for all $d$.
\end{theorem}
\begin{proof}
The pushforward arises from the factorization
\begin{equation*}
\begin{tikzcd}
\CH_d^f (X) \ar[rr] \ar[dr,"\sim"] && \CH_d(Y) \\
& \CH_d (X) \ar[ur] &
\end{tikzcd}\,.    
\end{equation*}
In case $f$ is of relative Deligne-Mumford type, Chow groups should be taken with $\QQ$-coefficients.
\end{proof}
\subsection{Properties of proper pushforward}
The proper pushforward satisfies the following expected properties.
\begin{proposition}\label{prop:properCompatibility}
{Let $Y$ be an algebraic stack stratified by global quotient stacks} and let $f\colon X\to Y$ be a proper, representable morphism as above. Then proper pushforward is compatible with representable flat pullback and refined Gysin maps for representable regular local {immersions}. The same compatibilities hold for Chow groups with $\QQ$-coefficients when $f$ is of relative Deligne-Mumford type.
\end{proposition}
\begin{proof}
We will prove the case when $f$ is a representable morphism. The proof for Deligne-Mumford type is similar. We first prove the compatibility of proper pushforward with flat pullbacks. Consider a cartesian diagram 
\begin{equation}\label{eqn:flatcartesian}
\begin{tikzcd} 
 W \arrow[r,"v"] \arrow[d,"g"] & X \arrow[d,"f"]\\
 Z \arrow[r,"u"] & Y
\end{tikzcd}
\end{equation}
where $u\colon Z\to Y$ is a representable flat morphism of relative constant dimension $\ell$. Since $u$ is representable, $Z$ is stratified by global quotient stacks (see \cite[Proposition 3.5.5]{kreschartin}).
%\jocomment{don't we need some criterion here for the target $Z$ of $g$ to be stratified by global quotients? I think we want that $u$ is a representable, flat morphism/representable regular local immersion. Then we can use \cite[Proposition 3.5.5]{kreschartin}.}\Ycomment{I think you are right.}
By property (iii) mentioned above Proposition \ref{p:restrictedexcision}  the natural isomorphism $\CH^f_*\to \CH_*$ is compatible with flat pullbacks. A class $(\rho,\alpha)$ in $\CH^f_*(X)$ is represented by the following cartesian diagram
\begin{equation}\label{eqn:representationProper}
\begin{tikzcd} 
 E' \arrow[r] \arrow[d,"f''"] & X' \arrow[r] \arrow[d,"f'"] & X \arrow[d,"f"]\\
 E \arrow[r,"\pi"] & Y' \arrow[r,"\rho"] & Y
\end{tikzcd}
\end{equation}
where $\rho$ is a projective morphism, $E$ is a vector bundle on $Y'$ and $\alpha$ is a class in $\CH^\circ_{d+\mathrm{rk} E'}(E')$.
% \footnote{$\CH^\circ_*(X)$ is so-called the naive Chow group of $X$ (\cite[Definition 2.1.4]{kreschartin}).}. 
By pulling back (\ref{eqn:representationProper}) along (\ref{eqn:flatcartesian}), we get
\begin{equation*}
\begin{tikzcd} 
 \CH^\circ_{d+\mathrm{rk} E'}(E') \arrow[r,"f''_*"] \arrow[d,"v^*"] & \CH^\circ_{d+\mathrm{rk} E'}(E) \arrow[d,"u^*"]\\
 \CH^\circ_{d+\mathrm{rk} E'+\ell}(F') \arrow[r,"g''_*"] & \CH^\circ_{d+\mathrm{rk} E'+\ell}(F)
\end{tikzcd}
\end{equation*}
where $F$ is the pullback of $E$ to $Z\times_Y Y'$ and $F'$ is the pullback of $E'$ to $W\times_X X'$. This diagram commutes because of the corresponding statement for naive Chow groups. Therefore the following diagram
\begin{equation*}
\begin{tikzcd} 
 \CH^f_d(X) \arrow[r,"f_*"] \arrow[d,"v^*"] & \CH_d(Y) \arrow[d,"u^*"]\\
 \CH^g_{d+\ell}(W) \arrow[r,"g_*"] & \CH_{d+\ell}(Z)
\end{tikzcd}
\end{equation*}
commutes.

The compatibility with regular local immersions
states that for a Cartesian diagram \eqref{eqn:flatcartesian} with $u$ a regular local immersion, we have $u^! f_* = g_* v^!$. Note that here the pushforward by $g$ is well-defined since the assumption that $Y$ is stratified by global quotient stacks together with the representability of $u$ implies that $Z$ likewise is stratified by global quotient stacks, using \cite[Proposition 3.5.5]{kreschartin}. Then the desired compatibility
also follows from a formal argument as above (see \cite[Proposition 18]{BHPSS}).
\end{proof}
The following proposition is a generalization of \cite[Lemma 3.8]{vistoli} to algebraic stacks.
\begin{proposition} \label{Pro:surproppush}
Let $X$ and $Y$ be algebraic stacks, finite type over a field $k$ and stratified by {global} quotient stacks. Let $f \colon X \to Y$ be a proper,  surjective morphism of relative Deligne-Mumford type. Then the pushforward \[f_*\colon \CH_d(X,\mathbb{Q}) \to \CH_d(Y,\mathbb{Q})\] is surjective for all $d$.
\end{proposition}
\begin{proof}
By the assumption of being stratified by global quotient stacks, there exists a nonempty open substack $U \subset Y$ isomorphic to a quotient $U \cong [T / G]$ of a quasi-projective scheme $T$ by a smooth, connected linear algebraic group $G$ acting linearly on $T$. Let $Z=Y \setminus U$ be the complement.
%\Ycomment{If we want to be super accurate we should say `there exists a nonempty open substack $U\subset Y^{red}$ but maybe this is too much?} 
% I think it's ok: I sent this proof to Andrew once for comments and he was fine with the formulation.
Consider a commutative diagram
\begin{equation} \label{eqn:spfexcision1}
\begin{tikzcd}
\CH_d(f^{-1}(Z),\mathbb{Q}) \arrow[r] \arrow[d]  & \CH_d(X,\mathbb{Q}) \arrow[r] \arrow[d] & \CH_d(f^{-1}(U),\mathbb{Q}) \arrow[r] \arrow[d] & 0\\
 \CH_d(Z,\mathbb{Q}) \arrow[r] & \CH_d(Y,\mathbb{Q}) \arrow[r] & \CH_d(U,\mathbb{Q}) \arrow[r] & 0
\end{tikzcd}
\end{equation}
with exact rows from the excision sequences and vertical arrows from the proper pushforwards by $f$ as defined above. Commutativity of the squares follows from the compatibility of proper pushforward with compositions of proper maps and pullbacks by flat maps. The pushforward $f_*$ is surjective over $Y$ if it is surjective over $U$ and $Z$ by the four lemma. Since the assumptions of the proposition hold for $Z$ and since $Y$ is of finite type, we can reduce by Noetherian induction to showing the statement over $U$.

By \cite[Proposition 3.5.10]{kreschartin}, the global quotient $U$ admits a vector bundle $\pi: E \to U$ such that $E$ is represented by a scheme off a locus of arbitrarily high codimension. Moreover, the pullback $\pi^*$ induces an isomorphism of Chow groups. Let $V=f^{-1}(U)$ and $E' = (f|_{V})^* E \to V$ be the pullback of $E$ under $f$. It suffices to show that the pushforward under $E' \to E$ is surjective because of the compatibility of proper pushforward with flat pullback. This computation can be done away from the locus where $E$ is not a scheme because we work in a fixed dimension $d$. Thus we have reduced to the case  where the target is a scheme.

Now assume that $f:X \to Y$ is as in the proposition and $Y$ is a scheme. Then the domain is a Deligne-Mumford stack 
because $f$ is assumed to be of relative Deligne-Mumford type. By \cite[Theorem 2.7]{kreschbrauer}, such a stack admits a finite surjective morphism from a scheme. %(note: the diagonal of $X$ is separated since $f$ is proper) 
This means that the machinery of Chow groups as developed by Vistoli in \cite{vistoli} is applicable. By \cite[Lemma 3.8]{vistoli}, the pushforward $f_*$ is surjective on naive Chow groups (with $\mathbb{Q}$-coefficients). Since $Y$ is a scheme, the naive Chow groups agree with those defined in \cite{kreschartin}, so for every cycle on $Y$ there exists a naive cycle on $X$ pushing forward to it. Finally, for naive cycles on $X$, the definition of proper pushforward agrees with the naive pushforward, so we are done.
\end{proof}
\begin{remark} \label{Rmk:quotmapprojective}
Even in very good situations, the pushforward by proper, surjective maps of finite type stacks is \emph{not} surjective on naive Chow groups. Indeed, for $n \geq 1$ consider the map
\[f: [\mathbb{P}^n / \mathrm{PGL}_{n+1}] \to B \mathrm{PGL}_{n+1},\]
where $\mathrm{PGL}_{n+1}$ acts in the usual way on $\mathbb{P}^n$. Then $f$ is a representable, proper surjective morphism of quotient stacks
and in fact we claim that it is also projective. To see the latter, note that the line bundle $\mathcal{O}_{\PP^n}(n+1)$ on $\PP^n$ is $\PGL_{n+1}$-linearizable (see e.g. \cite[Example 3.2.7]{brionlinearization}). Thus it descends to a line bundle on $[\mathbb{P}^n / \mathrm{PGL}_{n+1}]$ which is relatively very ample for $f$. Therefore, the proper morphism $f$ is indeed projective. However, even though $f$ has all these nice properties, the pushforward $f_*$ still vanishes on naive Chow groups. Let $V\to [\PP^n/\mathrm{PGL}_{n+1}]$ be an integral closed substack. Let $V_{\PP^n}$ be the fiber product
\begin{equation*}
    \begin{tikzcd}
    V_{\PP^n}\arrow[r]\arrow[d] & V\arrow[d]\\
    \PP^n\arrow[r] & \left[\PP^n/\mathrm{PGL}_{n+1}\right]\,.
    \end{tikzcd}
\end{equation*}
Then $V_{\PP^n}$ is a closed subscheme of $\PP^n$ which is invariant under $\PGL_{n+1}$ so $V_{\PP^n} = \PP^n$ and thus also $V=[\PP^n/\mathrm{PGL}_{n+1}]$. On the other hand $f$ is relative dimension $n$ hence $f_*[V]$=0. 
We are grateful to Andrew Kresch, who pointed out this example to us.
\end{remark}
\begin{remark} \label{Rmk:surproppush}
Since the proper pushforward constructed in this section is compatible with flat pullbacks, it follows immediately from the definition of Chow groups of locally finite type stacks in the previous section that these inherit the proper pushforward construction. In particular, Theorem \ref{thm:properpushforward} and Proposition \ref{prop:properCompatibility} remain true for stacks only locally of finite type.

On the other hand, Proposition \ref{Pro:surproppush} remains true for algebraic stacks $X,Y$ which are Lindel\"of and stratified by {global} quotient stacks (as is the case for the stacks $\M_{g,n,a}$ described in Section \ref{Sect:Avaluedprestable}). 

Indeed, let $U_i$ be a finite type cover of $Y$, then $V_i = f^{-1}(U_i)$ is also finite type since $f$ is proper. Let $K_i$ be the kernel of the proper pushforward $f_*: \CH_*(V_i) \to \CH_*(U_i)$, then applying the proposition we have a commutative diagram
\begin{equation}
\begin{tikzcd}
 0 \arrow[d] & 0\arrow[d] & 0\arrow[d] & \\
 K_{ij}\arrow[d] \arrow[r] & K_i \arrow[d] \arrow[r]& K_j \arrow[d] & \\
 \CH_*(V_i \setminus V_j)\arrow[d] \arrow[r] & \CH_*(V_i)\arrow[d] \arrow[r] & \CH_*(V_j)\arrow[d] \arrow[r] & 0\\
 \CH_*(U_i \setminus U_j)\arrow[d] \arrow[r] & \CH_*(U_i)\arrow[d] \arrow[r] & \CH_*(U_j)\arrow[d] \arrow[r] & 0\\
 0 & 0 & 0 &
\end{tikzcd}
\end{equation}
{The columns are exact by Proposition~\ref{Pro:surproppush} and the middle and the bottom rows are exact by \cite[Proposition  2.3.6]{kreschartin}.}
Applying a small variant of the Snake lemma\footnote{Note that in comparison to the usual situation of the Snake lemma, we don't have injectivity of the map $i_*: \CH_*(U_i \setminus U_j) \to \CH_*(U_i)$. This can be repaired by replacing $\CH_*(U_i \setminus U_j)$ with $\CH_*(U_i \setminus U_j)/\ker(i_*)$ and observing that the map from $\CH_*(V_i \setminus V_j)$ is still surjective.} we see that the maps $K_i \to K_j$ of the directed system $(K_i)_i$ are surjective, so this system is Mittag-Leffler (see \cite[\href{https://stacks.math.columbia.edu/tag/0596}{Tag 0596}]{stacks-project}). Then it follows from \cite[\href{https://stacks.math.columbia.edu/tag/0598}{Tag 0598}]{stacks-project} that the induced map
\[\CH_*(X) = \varprojlim_i \CH_*(V_i) \longrightarrow  \varprojlim_i \CH_*(U_i) = \CH_*(Y)\]
is indeed surjective. 
\end{remark}
%The following is a variant of Proposition \ref{Pro:surproppush} for higher Chow groups.
%\begin{lemma}\label{Pro:surproppushHigher}
%Let $X$ and $Y$ be quotient stacks. Let $f\colon Y\to X$ be a proper, surjective morphism of relative Deligne-Mumford type. Then the pushforward 
%\[f_*\colon \CH_d(Y,1)_\QQ\to \CH_d(X,1)_\QQ\]
%is surjective for all $d$.
%\end{lemma}
%\begin{proof}
%The reduction argument in the proof of Proposition \ref{Pro:surproppush} applies in this case and hence we may assume $X$ is a scheme and $Y$ is a Deligne-Mumford stack. In \cite[Lemma 3.8]{vistoli} the surjectivity of the proper pushforward is proven in the level of cycles. We have a surjective proper pushforward
%\[f_* \colon z^*(Y\times R)_\QQ \to z^*(X\times R)_\QQ\]
%where $z^*(X\times R)$ is the same as the kernel of the boundary map in the Bloch's cycle complex $z^*(X,1)\to z^*(X,0)$. Therefore $f_*$ is surjective on the first higher Chow groups.\Ycomment{I think this proof is wrong. The first higher Chow group is the cohomology of the complex
%\[...\to z^*(X\times R)\xrightarrow{\partial_X} z^*(X)\to 0\]
%so we have to prove that $f_*$ induces surjective map 
%\[ker \partial_X\to ker \partial_Y\,.\]
%This is not clear from the proof above. 
%}
%\end{proof}

\section{Operational Chow groups for algebraic stacks} \label{Sect:OperationalChow}
In this section we give a definition of operational Chow classes for algebraic stacks which we assume throughout to be locally finite type over $k$.
\begin{definition}\label{def:Operationalclass}
An operational class $c$ in the $p$-th operational Chow group $\CHOP^p(X)$ is a collection of homomorphisms
\[c(g)\colon \CH_m(B)\to \CH_{m-p}(B)\]
for all morphisms $g\colon B\to X$ where $B$ is an algebraic stack of finite type over $k$, stratified by {global} quotient stacks and for all integers $m$, compatible with representable proper pushforward, flat pullback, and refined Gysin pullback along representable lci morphisms (see \cite[Section 17.1]{Fulton1984Intersection-th}). In particular the compatibility with refined Gysin pullback means the following compatibility condition: 
consider a diagram
\begin{equation*}
    \begin{tikzcd}
    B' \arrow[r] \arrow[d,"f'"] & Z' \arrow[d,"f"]\\
    B \arrow[d,"g"] \arrow[r] & Z\\
    X & 
    \end{tikzcd}
\end{equation*}
where $g : B \to X$ is a morphism from an algebraic stack $B$ of finite type over $k$, stratified by global quotient stacks, $f: Z' \to Z$ is a representable lci morphism and $Z'$ is stratified by global quotient stacks and the square in the diagram is Cartesian.
Then we require that for all $c\in \CHOP^*(X)$ and $\alpha\in\CH_*(B)$ we have
\[f^!(c(g)\cap\alpha) = c(gf')\cap f^!\alpha \,\textup{ in } \CH_*(B')\,.\]
\end{definition}
This notion of operational Chow group of algebraic stacks shares the following formal properties of operational Chow group of schemes: it is a contravariant functor for all morphisms and has the structure of an associative $\mathbb{Q}$-algebra coming from composing two operations.
Note that for the functoriality under all morphisms, it is important that we did not restrict the morphisms $g : B \to X$ in the definition above e.g. to be representable. Otherwise we would only get functoriality under representable morphisms.

\begin{example}
Let $E$ be a vector bundle on $X$, then its \emph{$r$-th Chern class} $$c_r(E) \in \CHOP^r(X)$$
is naturally an operational class on $X$. Given $g: B \to X$ it acts by
\[(c_r(E))(g)\colon \CH_m(B)\to \CH_{m-p}(B), \alpha \mapsto c_r(g^* E)\cap \alpha\]
where the Chern class of $g^* E$ and its action on the cycle $\alpha$ on $B$ are as defined in \cite{kreschartin}. 
\end{example}

We start with a small observation: the operational Chow group of $X$ can be computed on a suitable finite-type cover of $X$.

\begin{lemma} \label{Lem:CHOPonftcover}
Let $X$ be a locally finite type algebraic stack over $k$. Let $(\mathcal{U}_i)_{i \in I}$ be a directed system of finite type open substacks of $X$ whose union is all of $X$. Then the flat pullbacks $j_i^*$ by the inclusions $j_i : \mathcal{U}_i \to X$ induce a map
\[
\Phi: \CHOP^*(X) \to \varprojlim_{i \in I} \CHOP^*(\mathcal{U}_i)
\]
and this map $\Phi$ is an isomorphism.
\end{lemma}
\begin{proof}
The proof, which uses the fact that each map $B \to X$ from a finite-type stack $B$ must factor through one of the $\mathcal U_i$ by Noetherian induction, goes verbatim as the proof presented in \cite[Corollary 15]{BHPSS}. 
\end{proof}

\begin{lemma}\label{Lem:properlcipushforward}
Let $f\colon X\to Y$ be a representable, proper, flat morphism of relative dimension $d$. Then there is a pushforward map
\[f_*\colon \CHOP^*(X)\to \CHOP^{*-d}(Y)\]
defined as follows. For a cartesian diagram
\begin{equation}\label{eqn:cartesianPushforward}
    \begin{tikzcd}
     C\arrow[r,"f'"]\arrow[d,"g'"] & B\arrow[d,"g"]\\
     X \arrow[r,"f"] & Y
    \end{tikzcd}
\end{equation}
and $c\in \CHOP^*(X)$, 
\[(f_*c)(g)\cdot\alpha = f'_*(c(g')\cdot (f')^*\alpha),\text{ for }\alpha\in \CH_*(B)\,.\]
If $f\colon X\to Y$ is a representable, proper, lci morphism of relative dimension $d$, the pushforward map is similarly defined by the formula
\[(f_*c)(g)\cdot\alpha = f'_*(c(g')\cdot f^!\alpha),\text{ for }\alpha\in \CH_*(B)\]
using the refined Gysin pullback. For a morphism $f$ which is representable, flat and lci, the two definitions coincide.
% \begin{enumerate}[label=(\alph*)]
%     \item If $f$ is a representable, proper, flat morphism then the pushforward is well-defined.
%     \item If $f$ is a representable, proper, lci morphism then the pushforward is well-defined.
%     \item If $f$ is a representable, proper, flat and lci morphism then two notions coincide.
% \end{enumerate}
\end{lemma}
\begin{proof}
{We check that the collection of maps $f_*c$ defines an operational Chow class.} We will only give a proof for the case that $f$ is lci, the proof for flat morphisms is similar. The fact that the two notions coincide for $f$ both flat an lci follows from the formula and the fact that the flat pullback and the lci pullback of cycles coincide.

Let $h\colon B'\to B$ be a representable proper morphism and consider the following cartesian diagram
\begin{equation}\label{eqn:compatibility}
    \begin{tikzcd}
     C'\arrow[r,"f''"] \arrow[d,"h'"] & B'\arrow[d,"h"]\\
     C\arrow[r,"f'"]\arrow[d,"g'"] & B\arrow[d,"g"]\\
     X\arrow[r,"f"] &Y.
    \end{tikzcd}
\end{equation}
For $c\in \CHOP^*(X)$ and $\alpha\in \CH_*(B')$ we have
\begin{align*}
 h_* (f_* c)(g \circ h)(\alpha) &= h_* \left( f''_*( c(g' \circ h') \cdot f^! \alpha)\right) \\
 &= f'_* \left( h'_*( c(g' \circ h') \cdot f^! \alpha)\right) \\
 &= f'_* \left( c(g' ) \cdot h'_* f^! \alpha\right) \\
 &= f'_* \left( c(g' ) \cdot f^! h_*  \alpha\right) = (f_* c)(g)(h_* \alpha),
\end{align*}
where the third equality uses the compatibility with proper pushforward for the operational class $c$ and the fourth equality uses compatibility of Gysin pullback and proper pushforward in Proposition \ref{prop:properCompatibility}.

% For $\alpha\in \CH_*(B')$ we have
% \begin{align*}
%     f'_*\left(f^!(h_*\alpha)\right)=f'_*\left( h'_*f^!\alpha\right)
%     =h_*f''_*\left( f^!\alpha\right)
% \end{align*}
% where the first equality uses compatibility of refined Gysin pullback and proper pushforward in Proposition \ref{prop:properCompatibility}.
Similarly let $h\colon B'\to B$ be a flat morphism and $\beta \in \CH_*(B)$, then we have
\begin{align*}
 (f_* c)(g \circ h)(h^* \beta) &=  f''_*( c(g' \circ h') \cdot f^! h^* \beta) \\
 &=  f''_*( c(g' \circ h') \cdot  (h')^* f^! \beta) \\
 &=  f''_*( (h')^* ( c(g') \cdot  f^! \beta) )\\
 &=  h^* f'_* ( c(g') \cdot  f^! \beta) = h^*(f_* c)(g )(\beta),
\end{align*}
where the second equality uses compatibility of Gysin maps with flat pullbacks, the third equality uses that $c$ is compatible with flat pullbacks and the fourth equality uses that pushforwards are compatible with flat pullbacks.

% \begin{align*}
%     f''_*\left(f^!h^*\alpha\right)=f''_*\left(h'^*f^!\alpha\right)=h^*\left(f'_*f^!\alpha\right)
% \end{align*}

Let $j\colon B\to Z$ be a morphism and $h\colon Z'\to Z$ be a (representable) regular local immersion where $Z'$ is stratified by global quotient stacks. Consider the fiber diagram 
\begin{equation*}
    \begin{tikzcd}
     C'\arrow[r,"f''"] \arrow[d,"h''"] & B'\arrow[r,"j'"]\arrow[d,"h'"] & Z'\arrow[d,"h"]\\
     C\arrow[r,"f'"]\arrow[d,"g'"] & B\arrow[r,"j"]\arrow[d,"g"] & Z\\
     X\arrow[r,"f"] &Y
    \end{tikzcd}\,.
\end{equation*}
For $\alpha\in \CH_*(B)$ we have
\begin{align*}
    h^!(f_*c)(g)(\alpha) &= h^!\left(f'_*(c(g')\cdot f^!\alpha)\right)\\
    &=f''_*\left(h^!(c(g')\cdot f^!\alpha)\right)\\
    &=f''_*\left(c(g'\circ h'')\cdot h^!f^!(\alpha)\right)\\
    &=f''_*\left(c(g'\circ h'')\cdot f^!h^!(\alpha)\right)=(f_*c)(g\circ h')(h^!\alpha)
\end{align*}
where the fourth equality comes from the commutativity of refined Gysin pullback (\cite[Section 3.1]{kreschartin}) and the second equality comes from Proposition \ref{prop:properCompatibility}.
\end{proof}
%\Ycomment{I agree with referee's point. We need to assume that $X$ is stratified by quotient stacks otherwise $c(j_i)$ does not act on $[\mathcal{U}_i]$.}
\begin{lemma}
Let $X$ be an equidimensional algebraic stack of dimension $n$ {which is stratified by global quotient stacks}. Then there exists a well-defined map
\begin{equation}\label{eqn:fundamentalclass}
    \cap \, [X]\colon \CHOP^*(X)\to \CH_{n-*}(X)\,.
\end{equation}
\end{lemma}
\begin{proof}
Let $(\mathcal{U}_i)_{i \in I}$ be a directed system of finite type open substacks of $X$ whose union is all of $X$. {Given $c \in \CHOP^*(X)$, for each open embedding $\iota_i\colon \mathcal{U}_i\to X$ we consider the cycle $c(\iota_i)\cdot [\mathcal{U}_i] \in \CH_*(\mathcal{U}_i)$. For a given $i\in I$ let $\ell\in I$ be an element such that $\mathcal{U}_\ell$ contains $\mathcal{U}_i$. Let $\iota_{i\ell}:\mathcal{U}_i\to \mathcal{U}_\ell$ be the open embedding. Since an operational Chow class commutes with the flat pullback, we have
\[\iota_{i\ell}^*(c(\iota_\ell)\cdot [\mathcal{U}_\ell]) = c(\iota_i)\cdot (\iota_{i\ell}^*[\mathcal{U}_\ell])=c(\iota_i)\cdot [\mathcal{U}_i]\,.\] 
Therefore the collection of cycles $c(\iota_i)\cdot [\mathcal{U}_i]$ gives a well-defined element in $\varprojlim_{i \in I} \CH_*(\mathcal U_i)$.}
\end{proof}
The following theorem is an analogy of the Poincar\'e duality for smooth stacks.
\begin{theorem}\label{Pro:Poincare}
Let $X$ be a smooth equidimensional algebraic stack of dimension $n$ stratified by {global} quotient stacks. Then the canonical map \eqref{eqn:fundamentalclass} with {$\mathbb{Q}$-coefficients} is an isomorphism of associative $\mathbb{Q}$-algebras.
\end{theorem}
\begin{proof}{In the following proof all Chow groups are with $\mathbb{Q}$-coefficients.} 
By Lemma \ref{Lem:CHOPonftcover} we see that both sides of \eqref{eqn:fundamentalclass} can be defined as the inverse limit over a cover of $X$ by  finite-type open substacks $\mathcal{U}_i$ and the map \eqref{eqn:fundamentalclass} is the map induced by the compatible system of maps
\begin{equation*}
    \cap \, [\mathcal{U}_i]\colon \CHOP^*(\mathcal{U}_i)\to \CH_{n-*}(\mathcal{U}_i), \, c\mapsto c(\mathrm{id})\cdot[\mathcal{U}_i]\,.
\end{equation*}
Thus it suffices to prove the result for $X$ of finite type over $k$.

%To start, we review a construction similar to the one presented in \cite[Section 2.3]{BHPSS} shows that there exists a map
{To start, there exists a map}
\begin{equation}
    \Phi \colon \CH_{n-*}(X) \to \CHOP^*(X)
\end{equation}
{constructed in \cite[Section 2.3]{BHPSS}.}
Given $\beta \in \CH_{n-*}(X)$ and $\varphi : B \to X$ from an algebraic stack finite type over $k$, stratified by global quotient stacks, consider the graph morphism $\varphi_B : B \to B \times X$. Then $\varphi_B$ is representable and regular local immersion because $X$ is smooth over $k$. The map $\Phi$ is defined by
\[\Phi(\beta)(\varphi) : \CH_*(B) \to \CH_*(B), \alpha \mapsto \varphi_B^! \left(\alpha \times \beta \right)\,.\]
We show that $\Phi$ is the inverse of $\cap \,[X]$ following the parallel argument in \cite[Chapter 17]{Fulton1984Intersection-th}. We note that from the definition of $\Phi$ it is easy to see that it is multiplicative, using that the product in $\CH_*(X)$ is defined by $\beta_1 \cdot \beta_2 = \Delta^! ( \beta_1 \times \beta_2)$, where $\Delta\colon X\to X\times X$ is the diagonal morphism.

Let  $p_2\colon X\times X\to X$ be the projection to the second factor. For all $\beta\in \CH_*(X)$, we have 
\[\Delta^!\left([X] \times \beta\right) =\Delta^!p_2^*\beta = \Delta^!p_2^!\beta = \beta\]
by the functoriality of the Gysin pullback. It shows that $\Phi(\beta)\cap [X]=\beta$.

We prove the other direction. For $c\in \CHOP^*(X)$ and $\alpha\in \CH_*(B)$ it is sufficient to prove that
\[\varphi_B^!\left(\alpha\times (c(\mathrm{id})\cdot[X])\right)= c(\varphi)\cdot\alpha\,.\]
Let $p_2 \colon B\times X\to X$ be the projection to the second factor. We first prove 
\begin{equation} \label{eqn:poincarehelpequation} \alpha\times \left(c(\mathrm{id})\cdot [X]\right) = c(p_2)\cdot (\alpha\times [X]) \textup{ in } \CH_*(B\times X)\,.\end{equation}
When $B$ is equidimensional and $\alpha$ is the class of the fundamental class $[B]$, the equality follows from the compatibility of $c$ with flat pullback by $p_2$:
\begin{equation} \label{eqn:fundclasscompat} [B]\times\left(c(\mathrm{id})\cdot [X]\right)=p_2^*\left(c(\mathrm{id})\cdot[X]\right) = c(p_2)\cdot [B\times X]\,.\end{equation}
A general class $\alpha$ can be represented as $(f,\alpha_0)$ where $f\colon Y\to B$ is a projective morphism and $E$ is a vector bundle on $Y$ and $\alpha_0\in \CH^\circ_*(E)$. Adding a trivial component to $Y$ we may assume $f$ is surjective. By Proposition \ref{Pro:surproppush} and the homotopy invariance property it is enough to check this equality for a class in $\CH^\circ_*(E)$. A class in $\CH^\circ_*(E)$ can be written as a linear combination of classes $[V]$ where $j:V\to E$ are integral closed substacks. Thus it suffices to show the statement for $\alpha_0 = [V]$. For this, consider the composition of maps
\[V\times X\xrightarrow{J=j\times\id} E\times X\xrightarrow{p_2} X\,.\]
By the definition of exterior product we have
\[J_*([V]\times(c(\id)\cdot [X])) = j_{*}[V]\times (c(\id)\cdot [X])\]
in $\CH_*(E\times X)$.
Then we get
\begin{align*}
    j_{*}[V] \times (c(\mathrm{id})\cdot [X]) 
    &= (J)_*([V]\times c(\id)\cdot [X])\\
    &= (J)_*( c(p_2\circ J)\cdot [V\times X])\\
    &= c(p_2)\cdot J_{*}[V\times X]\\
    &= c(p_2)\cdot (j_{*}[V]\times [X])\
\end{align*}
where the second equality follows from the proven case \eqref{eqn:fundclasscompat} and the third equality follows from the compatibility of proper pushforward.
Using \eqref{eqn:poincarehelpequation} we then have
\begin{align*}
    \varphi_B^!\left(\alpha\times (c(\mathrm{id})\cdot [X])\right) &= \varphi_B^!\left(c(p_2)\cdot (\alpha\times [X])\right)\\
    &= c(\varphi)\cdot\varphi_B^!\left(\alpha\times [X]\right) = c(\varphi)\cdot \alpha
\end{align*}
which proves the theorem. 
\end{proof}
The above theorem gives the commutativity of operational Chow group under assumptions.
\begin{corollary}
Let $X$ be an algebraic stack stratified by global quotient stacks.
\begin{enumerate}
    \item[a)] When $X$ is smooth over $k$, $\CHOP^*(X)_\QQ$ is a commutative ring.
    \item[b)] When $\textup{char} (k)=0$,  $\CHOP^*(X)_\QQ$ is a commutative ring. 
\end{enumerate}
\end{corollary}
%Let $X$ be an algebraic stack stratified by global quotient stacks. Then $\CHOP^*(X)_\QQ$ is a commutative ring. 
\begin{proof}
We adapt the proof of \cite[Example 17.4.4]{Fulton1984Intersection-th}. Part a) is a direct consequence of Theorem  \ref{Pro:Poincare} since the intersection product of \cite{kreschartin} is commutative:
\[
\alpha \cdot \beta = \Delta^! (\alpha \times \beta) = \Delta^! (\beta \times \alpha) = \beta \cdot \alpha
\]
for $\alpha, \beta \in \CH_*(X)$ and where $\Delta : X \to X \times X$ is the diagonal, which is invariant under switching the factors of $X \times X$.

The proof of b) relies on functorial resolution of singularities. When $\textup{char} (k)=0$, the resolution of singularities can be done functorially with respect to smooth morphisms, see \cite{EV98}. Therefore we can find a smooth stack $\widetilde{X}$ and a representable, surjective, birational morphism $p: \widetilde{X}\to X$. Since $X$ is stratified by global quotient stacks and $p$ is representable, $\widetilde{X}$ is also stratified by global quotient stacks. Therefore the commutativity of $\CHOP^*(X)_\QQ$ follows because the pullback $p^*$ on $\CHOP^*(-)_\QQ$ is injective.   
\end{proof}
When $k$ is a perfect field, the commutativity of rational operational Chow groups of schemes follows from de Jong's alteration (\cite{deJong96}). However, the authors do not know whether a functorial (with respect to smooth morphisms) construction of alteration is possible. Hence we cannot prove for now the commutativity of operational Chow groups for algebraic stacks over a perfect field of positive {characteristic}.

However, we note that nonetheless all results and formulas concerning tautological classes discussed in the main text are valid over {arbitrary} fields. Indeed, the operational classes only appear in intermediate steps of some of the computations and while these contain some examples of non-smooth spaces (like the universal curve over the space $\M_\Gamma$), we are never in the position of having to exchange orders of multiplication of operational classes on these singular spaces.

In the main part of the paper we use Theorem \ref{Pro:Poincare} above to realize tautological classes in $\CH_*(\M_{g,n,a})$ as operational Chow classes.
The following lemma is then useful when for doing calculus between tautological classes.
\begin{lemma}\label{Lem:OperationalChowcompatible}
Consider the following commutative diagram
\begin{equation*}
    \begin{tikzcd}
U
\arrow[drr, bend left, "r"]
\arrow[ddr, bend right, swap, "s"]
\arrow[dr, "p"] & & \\
& W \arrow[r, "f'"] \arrow[d, "\pi'"] & Z \arrow[d, "\pi"] \\
& X \arrow[r, "f"] & Y
\end{tikzcd}
\end{equation*}
where all stacks are locally of finite type over $k$, equidimensional, stratified by {global} quotient stacks and the square in the middle is a cartesian square. Suppose $f$ is representable, proper and $\pi$ is representable, proper, flat and $p$ is representable, proper, birational. For $\alpha\in \CHOP^*(X)$
\[\pi^*f_*(\alpha\cap[X])=r_*(s^*\alpha\cap [U]) \hspace{2mm} \textup{ in } \CH_*(Z,\mathbb{Q})\,.\]
\end{lemma}
\begin{proof}
Using the compatibility of proper pushforward and flat pullback of the cycle $\alpha \cap [X]$ in $\CH_*(X,\mathbb{Q})$ and the fact that, by definition, $\alpha$ is compatible with flat pullback by $\pi'$, we have
\begin{align*}
    \pi^* f_* (\alpha \cap [X]) &= f'_*(\pi')^* (\alpha \cap [X])\\
    &= f'_* \left( ((\pi')^*\alpha) \cap (\pi')^*[X] \right)\\
    &= f'_* \left( ((\pi')^*\alpha) \cap [W] \right)\,.
\end{align*}
Since $p_*[U]=W$, we use the projection formula for the proper morphism $p$ and obtain
\begin{align*}
    f'_* \left( ((\pi')^*\alpha) \cap [W] \right) &= f'_* \left( ((\pi')^*\alpha) \cap p_*[U] \right)\\
    &=f'_* \left( p_* \left( (p^* (\pi')^*\alpha) \cap [U] \right)\right)\\
    &=r_*  (s^*\alpha) \cap [U] .
\end{align*}
%========OLD ======
%For $\alpha\in\CH^*(X)$
%\begin{align*}
%    \pi^*f_*(\alpha\cap[X])&= \pi^*f_*(\alpha\cap f^![Y])\\
%    &=\pi^*(f_*\alpha\cap [Y])\\ 
%    &= \pi^*f_*\alpha\cap [Z]\\
%    &=f'_*(\pi^{'*}\alpha\cap f^![Z]).
%\end{align*}
%Since $\pi$ is flat, $f'$ is also lci and $f$ and $f'$ have the same relative dimension. Therefore $f^![Z]=f'^![Z] =[W]$. On the other hand, $p_*[U] = [W]$ by the assumption and 
%\begin{align*}
%    r_*(s^*\alpha\cap [U])&=f'_*p_*(p^*\pi'^{*}\alpha\cap[U])\\
%    &=f'_*(\pi'^*\alpha\cap p_*[U])\\
%    &=f'_*(\pi'^*\alpha\cap[W])\,.
%\end{align*}
This proves the identity. 
\end{proof}
We conclude the section by comparing the approach to operational classes above to the one taken in the paper \cite{BHPSS}. This paper studies the intersection theory of the universal Picard stack $\mathfrak{Pic}_{g,n}$ of the universal curve $\C_{g,n} \to \M_{g,n}$. However, instead of studying the operational Chow classes defined above, this paper considers operational Chow groups $\CH^*_\text{op}(\mathfrak{Pic}_{g,n})$  where the test spaces $B \to X$ are restricted to be finite type {\em schemes}. A class $c \in \CH^p_\text{op}(X)$ on a locally finite type algebraic stack $X$ over $k$ is a collection of operations
 \[c(\varphi) : \CH_*(B) \to \CH_{*-p}(B)\]
for every morphism $\varphi: B \to X$ where $B$ is a scheme of finite type over $k$, satisfying compatibility conditions as in Definition \ref{def:Operationalclass}, see \cite[Definition 10]{BHPSS} for details. We have comparison maps between $\CHOP^*, \CH^*_\text{op}$ and $\CH^*$.\footnote{For the remainder of the section we assume that $X$ is equidimensional and write $\CH^*(X)$ for the Chow ring indexed by codimension, to emphasize that the comparison maps below are morphisms of graded rings.}
%\jocomment{The referee complained here that $\CH^*$ should be $\CH_*$. Maybe we should add a footnote at the end of the last sentence saying:}""
As explained in \cite[Section 2.3]{BHPSS}, for $X$ smooth, equidimensional and admitting a stratification by global quotient stacks, there exists a natural map
\begin{equation} \label{eqn:CHCHopmorph}
     \CH^*(X) \to \CH^*_\text{op}(X)\,.
 \end{equation}
%which specializes to the isomorphism \cite[Corollary 17.4]{Fulton1984Intersection-th} for $X$ a scheme. 
On the other hand for any algebraic stack $X$, there exists a natural map \begin{equation}\label{eqn:OPopComparison}
    \CHOP^*(X)\to \CH^*_\text{op}(X)
\end{equation}
defined by the restriction. The following statement is a direct consequence of Theorem \ref{Pro:Poincare}.
%When $X$ is a quotient stack, we have the following statement.
%\begin{proposition}\label{pro:QuotientStackComparison}
%When $X$ is a quotient stack, \eqref{eqn:OPopComparison} is an isomorphism.
%\end{proposition}
%\begin{proof}
%We show that the inverse of \eqref{eqn:OPopComparison} exists. Let $\varphi \colon B\to X$ be a representable morphism. Since $\varphi$ is a representable morphism, $B$ is also a quotient stack by \cite[Lemma 2.12]{kreschbrauer}. $B$ has a vector bundle $E$ such that it is represented by a scheme off a locus of arbitrary high codimension. Any class of $B$ can be realized by a class supported on the scheme locus of $E$ by the homotopy invariance. Therefore we can define action of $\CH_\text{op}^*(X)$ to a class in $\CH_*(B)$. The action is well-defined because operational classes commute with flat pullbacks. It is straightforward to check that this extension defines a class in $\CHOP^*(X)$ and it gives the inverse of \eqref{eqn:OPopComparison}.
%\end{proof} 
\begin{corollary}
When $X$ is an equidimensional smooth Deligne-Mumford stack over $k$, the comparison maps \eqref{eqn:CHCHopmorph} and \eqref{eqn:OPopComparison} are isomorphisms.
\end{corollary}
\begin{proof}
Indeed, for the two maps
\[
\CH^*(X) \to \CHOP^*(X) \to \CH_\text{op}^*(X)
\]
we have that the first is an isomorphism by Theorem \ref{Pro:Poincare} and the composition is an isomorphism by \cite[Lemma 15]{BHPSS} and thus also the second map must be an isomorphism.
\end{proof}

\bibliographystyle{alpha}
\bibliography{main}
\end{document}